\documentclass[11 pt]{amsart}
\usepackage{amsfonts}
\usepackage{amsmath,amscd}
\usepackage{fullpage}
\usepackage{amssymb}
\usepackage{centernot} 
\usepackage{enumerate} 
\usepackage{parskip}
\usepackage{pb-diagram} 
\usepackage{mathrsfs}
\usepackage[OT2,T1]{fontenc}
\usepackage{seqsplit}
\usepackage{tikz}
\usepackage{tikz-cd}

\usepackage{color}
\usepackage{array}
\usepackage{verbatim}
\usepackage{url}

\DeclareSymbolFont{cyrletters}{OT2}{wncyr}{m}{n}
\DeclareMathSymbol{\Sha}{\mathalpha}{cyrletters}{"58}

\definecolor{refkey}{rgb}{1,1,1}
\definecolor{labelkey}{rgb}{1,1,1}
\definecolor{cite}{rgb}{0.9451,0.2706,0.4941}
\definecolor{ruri}{rgb}{0.0078,0.4022,0.8010}

\usepackage[%
bookmarks=true,bookmarksnumbered=true,%
colorlinks=true,linkcolor=ruri,citecolor=red%
 ]{hyperref}

\makeindex 

\def\F{{\rm \mathbb{F}}}
\def\Z{{\rm \mathbb{Z}}}

\def\Q{{\rm \mathbb{Q}}}

\def\C{{\rm \mathbb{C}}}

\def\P{{\rm \mathbb{P}}}

\def\p{{\rm \mathfrak{p}}}

\def\A{{\rm \mathbb{A}}}

\def\NS{{\rm NS}}

\def\PGL{{\rm PGL}}

\def\Aut{{\rm Aut}}

\def\NS{{\rm NS}}
\def\Pic{{\rm Pic}}

\def\Disc{{\rm Disc}}

\def\Sym{{\rm Sym}}

\def\GL{{\rm GL}}

\def\Gal{{\rm Gal}}

\def\Hom{{\rm Hom}}
\def\End{{\rm End}}

\def\Spec{{\rm Spec}}

\newcommand{\sheaf}[1]{\mathscr{#1}}

\numberwithin{equation}{section}

\newtheorem{theorem}{Theorem}[section]
\newtheorem{lemma}[theorem]{Lemma}

\newtheorem{question}{Question}
\newtheorem{remark}[theorem]{Remark}
\newtheorem{definition}[theorem]{Definition}
\newtheorem{example}[theorem]{Example}

\newtheorem{corollary}[theorem]{Corollary}
\newtheorem{proposition}[theorem]{Proposition}
\usepackage{tikz}

\DeclareMathOperator{\Mat}{Mat}
\DeclareMathOperator{\Jac}{Jac}
\newcommand{\define}[1]{{\fontfamily{cmss}\selectfont{#1}}}
\DeclareMathOperator{\Proj}{Proj}
\DeclareMathOperator{\sep}{sep}
\newcommand{\calO}{\mathcal{O}}
\DeclareMathOperator{\HH}{H}
\DeclareMathOperator{\Prym}{Prym}

\begin{document}
\setlength{\parskip}{2pt} % 1ex plus 0.5ex minus 0.2ex}
\setlength{\parindent}{8pt}
\title{The geometry and arithmetic of bielliptic Picard curves} 
\author{Jef Laga}
\address{Department of Pure Mathematics and Mathematical Statistics, Wilberforce Road, Cambridge, CB3 0WB, UK}
\email{jeflaga@hotmail.com}
\author{Ari Shnidman}
\address{Einstein Institute of Mathematics, Hebrew University of Jerusalem, Israel} 
\email{ariel.shnidman@mail.huji.ac.il}

\maketitle

\begin{abstract}
We study the geometry and arithmetic of  the curves $C \colon y^3 = x^4 + ax^2 + b$ and their associated Prym abelian surfaces $P$.  We prove a Torelli theorem in this context and give a geometric proof of the fact that $P$ has quaternionic multiplication (QM) by the quaternion order of discriminant $6$. This allows us to  describe the Galois action on the geometric endomorphism algebra of $P$. As an application, we classify the torsion subgroups of the Mordell-Weil groups $P(\Q)$, as both abelian groups and $\End(P)$-modules. 
\end{abstract}
%\vspace{.5cm}

\tableofcontents
\makeatletter
%\@starttoc{toc}
\makeatother

\section{Introduction}

Let $k$ be a field of characteristic neither $2$ nor $3$.  A \define{bielliptic Picard curve} over $k$ is a smooth projective curve $C$ with an affine model of the form 
\begin{align}\label{equation: biell pic curve intro}
y^3 = x^4+ax^2+b
\end{align}
for some $a,b\in k$.  Such a curve is equipped with both a $\mu_3$-action $(x,y)\mapsto (x,\omega y)$ and a commuting involution $\tau \colon (x,y)\mapsto (-x,y)$.  The induced involution $\tau^*$ on the Jacobian variety $J = \Jac_C$ allows us to define the Prym variety $P = \ker(1+\tau^*)$.
The abelian surface $P/k$ inherits a $\mu_3$-action and carries a $(1,2)$-polarization; it need not be principally polarizable over $k$.

The goal of this paper is to explore the remarkably rich geometry and arithmetic of bielliptic Picard curves and their associated Pryms.

\subsection{Basic results}
We first prove some foundational results for bielliptic Picard curves. We
\begin{enumerate}
    \item\label{item: torelli} prove a Torelli theorem for the association $C\mapsto P$, once suitable data on both sides is fixed (Theorem \ref{theorem: Torelli theorem bielliptic Picard curves});
    \item\label{item: QM property}  show $\End(P_{\bar{k}})$ contains a maximal order $\calO$ in a discriminant $6$ quaternion algebra (Proposition \ref{proposition: Prym has O-PQM}), in other words  $P_{\bar{k}}$ has quaternionic multiplication (QM) by $\calO$;
    \item\label{item: endo field} give an explicit description of the Artin representation $\rho_{\End}\colon \Gal_k \to \Aut(\calO) \hookrightarrow \GL_4(\Q)$ describing the Galois action on the endomorphism ring $\calO$ (Corollary \ref{corollary: determination endo field});
    \item\label{item: bielliptic Picard curves} relate the moduli space of bielliptic Picard curves to a unitary Shimura curve $Y/\Q$ and a quaternionic Shimura curve  (Section \ref{subsec: comparison with quaternionic curve});
    \item\label{item: CM points} determine the finitely many $\bar \Q$-isomorphism classes of Pryms $P/\Q$ which are geometrically non-simple (Proposition \ref{proposition: all CM j-invariants}), corresponding to the CM points on $Y$. 
\end{enumerate}

 Our Torelli result (\ref{item: torelli}) is perhaps surprising because Barth \cite{Barth-abeliansurfacespolarization} has shown that the Prym construction on all bielliptic genus $3$ curves does not satisfy a Torelli theorem, but has one-dimensional fibres.  There are other instances of algebraic families of abelian surfaces with QM in the literature \cite{BabaGranath-genus2curvesQM, HashimotoMarabayashi}, but (\ref{item: QM property}) is interesting because of the very simple nature of this family. 
Moreover, contrary to previous approaches, our proof constructs the quaternionic action explicitly and geometrically, using the automorphisms and the polarization of $P$.
Part (\ref{item: endo field}) gives an interesting source of abelian surfaces with large endomorphism field and, in the notation of \cite{FKRS-STgenus2},  produces the first published examples of geometrically simple abelian surfaces with Sato--Tate group \href{https://www.lmfdb.org/SatoTateGroup/1.4.E.6.1a}{$J(E_3)$} and \href{https://www.lmfdb.org/SatoTateGroup/1.4.E.12.4a}{$J(E_6)$} (Example \ref{example: sato-tate group large}).
Part (\ref{item: bielliptic Picard curves}) and Shimura reciprocity allow us to calculate all geometrically non-simple Prym varieties in (\ref{item: CM points}).

\subsection{Rational torsion subgroups}
To further illustrate the rich and accessible nature of these curves, we classify the finite torsion subgroups that arise in the Mordell-Weil groups $P(\Q)$ of Pryms of bielliptic Picard curves over $\Q$.

\begin{theorem}\label{thm:main}
   Let $P/\Q$ be the Prym surface of a bielliptic Picard curve $C/\Q$. Then 
   \[P(\Q)_{\mathrm{tors}} \simeq 
      \begin{cases}
\Z/n\Z & \mbox{ for some } n \in \{1,2,3,6\}, \mbox{ or }\\
   \Z/n\Z \times \Z/n\Z &\mbox{ for some } n \in \{2,3\}, \mbox{ or } \\
   \Z/6\Z \times \Z/3\Z. & \\

   \end{cases}
   \] 
   Conversely, for each finite abelian group $G$ above, there exist infinitely many $\bar{\Q}$-isomorphism classes of bielliptic Picard  Prym surfaces $P/\Q$ with $P(\Q)_{\mathrm{tors}} \simeq G$.  
\end{theorem}

Theorem \ref{thm:main} is the analogue of Mazur's classification of rational torsion points of elliptic curves \cite{Mazur-eisensteinideal} in our setting.
As with elliptic curves, each finite group allowed by Theorem \ref{thm:main} can be realized by an explicit family (in fact, sometimes multiple families) of Pryms whose coefficients can be rationally parameterized. 
This is in accordance with the philosophy of Mazur and Ogg, namely that torsion should only occur as a consequence of the ambient geometry of the relevant moduli space; see \cite{Mazur-Oggstorsionconjgcture} for a recent survey.  Theorem \ref{thm:main} is significantly easier to prove than Mazur's theorem, essentially because $P$ has everywhere potentially good reduction. Nonetheless, it appears to be the first classification of rational torsion on a universal abelian variety over a Shimura variety which is not a modular curve. 

\begin{remark}
{\em
In a separate work with Schembri and Voight \cite{LSSV-QMMazur}, we significantly extend some of our arguments here to show that for {\em any} abelian surface $A/\Q$ such that $\End(A_{\bar{\Q}})$ is a maximal order in a non-split quaternion algebra over $\Q$, there is the uniform bound $|A(\Q)_{\mathrm{tors}}| \leq 18$. Theorem $\ref{thm:main}$ shows that this bound is sharp. Contrary to \cite[Theorem 1.3]{LSSV-QMMazur}, Theorem \ref{thm:main} is a complete classification, which we achieve using arguments specific to the geometry of the abelian surfaces considered here.
}
\end{remark}

\begin{example}
{\em 
By Proposition \ref{prop: order 18 torsion} below, for each $t \in \Q\setminus\{0,1,-1\}$ the Prym $P_t$ associated to the curve 
% \[C_t \colon y^3 = (x + 2t^3 + 2)(x + 2t^3 - 2)(x - 2t^3 + 2)(x - 2t^3 - 2),\] 
% \[y^3 = x^4 -16(t^2+1)(t^2-1)^2x^2 + 2^8t^2(t^2 -1)^4\]
\[2(t^2 -1)^2y^3 = (x^2 -1)(x^2 -t^2)\]
satisfies $P_t(\Q)_{\mathrm{tors}} \simeq \Z/3\Z \times \Z/6\Z$, achieving the maximum order allowed by Theorem \ref{thm:main}. 
}
\end{example}
 \begin{remark}
     {\em
     Theorem \ref{thm:main} implies that the rational torsion subgroup $J(\Q)_{\mathrm{tors}}$ of the Jacobian of a bielliptic Picard curve is $12$-torsion. 
     A complete classification of the groups $J(\Q)_{\mathrm{tors}}$ would require parameterizing points of order $4$ in $J(\Q)$. 
     %We hope to return to this in follow-up work. 
     }
 \end{remark}

It is natural to ask how the torsion subgroup of $P(\Q)$ is affected by the presence of endomorphisms defined over $\Q$. Generically, we have $\End(P) = \Z$, but it can also happen that $\End(P)$ is an order in a quadratic field.
Such $P$ are said to be \define{of $\GL_2$-type}, and are modular by work of Khare-Wintenberger \cite{KhareWintenberger-Serresmodularity}.  After proving explicit conditions on $a$ and $b$ for $P$ to be of $\GL_2$-type (see \S \ref{subsec:gl2type}), we classify all the possible $\End(P)$-modules that arise as $P(\Q)_{\mathrm{tors}}$.

\begin{theorem}\label{thm:gl2main}
Let $P/\Q$ be the Prym surface of a non-CM bielliptic Picard curve  over $\Q$, with $\End(P) \neq \Z$. Then $\End(P) \simeq \Z[\sqrt{D}]$ for some $D \in \{2,6\}$. Furthermore,  there is an isomorphism of $\Z[\sqrt{D}]$-modules 
\[P(\Q)_{\mathrm{tors}} \simeq 
      \begin{cases}
\{0\} & \mbox{ or }\\
   \Z[\sqrt{D}]/\mathfrak{a_p} &\mbox{ for some $p \in \{2,3\}$ },
   \end{cases}
   \] 
where $\mathfrak{a}_p$ is the unique prime ideal of $\Z[\sqrt{D}]$ above $p$. Conversely, for each $D \in \{2,6\}$, and for each of the three cyclic $\Z[\sqrt{D}]$-modules $G$ that appear above,  there are infinitely many $\overline{\Q}$-isomorphism classes of Pryms $P/\Q$ such that $\End(P) \simeq \Z[\sqrt{D}]$ and $P(\Q)_{\mathrm{tors}} \simeq G$ as $\Z[\sqrt{D}]$-modules.   
\end{theorem}

When $D = 2$ and $p = 3$, Theorem \ref{thm:gl2main} gives examples of $\GL_2$-type Pryms such that $P(\Q)_{\mathrm{tors}} \simeq \F_9$, showing that the upper bound proven in \cite[Theorem 1.4]{LSSV-QMMazur} for the order of the rational torsion subgroup of a QM abelian surface of $\GL_2$-type over $\Q$ is in fact sharp.

\subsection{Methods}
The geometry of $(1,2)$-polarized surfaces was analyzed in detail by Barth \cite{Barth-abeliansurfacespolarization} and  his results are crucial to our study of $C$ and $P$. The key player in the story is the bigonal dual curve $\widehat{C} \colon y^3 = x^4 + 8ax^2 + 16(a^2 - 4b)$, which is the intersection of $P$ with a well chosen theta divisor on $J$. The corresponding line bundle $\sheaf{M} = \calO_P(\widehat{C})$ represents the $(1,2)$-polarization on $P$. 
% The curve $\widehat{C}$ turns out to be a bielliptic Picard curve as well, with equation $y^3 = x^4 + 8ax^2 + 16(a^2 - 4b)$.
The linear system $|\sheaf{M}|$ is a pencil of genus $3$ curves,
and the $\mu_3$-action on $C$ induces a $\mu_3$-action on  $P$ and on the pencil.
The technical heart of this paper is to analyze this $\mu_3$-action on the pencil $|\sheaf{M}|$, both theoretically and explicitly. 
Our Torelli theorem (\ref{item: torelli}) boils down to showing that $\widehat{C}$ is the unique bielliptic Picard curve in $|\sheaf{M}|$ of the correct signature.
For the QM property (\ref{item: QM property}), we give two different proofs. The first is indirect,  relating the moduli of bielliptic Picard curves to a unitary Shimura curve (in other words, we first prove (\ref{item: bielliptic Picard curves}) and then deduce (\ref{item: QM property})). The second proof constructs the QM explicitly, by showing that a specific sextic twist of the original curve $C$ also lives in $|\sheaf{M}|$, but carrying the opposite signature. This second proof allows us to prove (\ref{item: endo field}) as well.

To prove Theorems \ref{thm:main} and \ref{thm:gl2main} we must first exhibit explicit families with specified torsion subgroups, and then we must rule out all other finite abelian groups. To exhibit groups, we describe elements of $P[2]$ and $P[3]$ geometrically,  as coming from the fixed  points for the $\mu_6$-action. To rule out groups which are not $6$-torsion, we first use general arguments  for QM abelian surfaces which are developed in greater generality in \cite{LSSV-QMMazur}.  Instead of quoting the main results there, we give a simpler proof (tailored to our special family of surfaces), which also allows us to handle the geometrically non-simple case not treated in \cite{LSSV-QMMazur}.\footnote{For those interested in reading that paper as well, it is perhaps instructive to first read \S\ref{sec: classifying rational torsion subgroups} of this paper (consulting \cite{LSSV-QMMazur} for the details of a few key results), before reading the more elaborate proofs in \cite{LSSV-QMMazur}, where nothing is assumed about the geometry of the abelian surfaces involved.} However, these general arguments only go so far; for example, they are not enough to rule out $4$-torsion in general QM abelian surfaces $A/\Q$. To eliminate groups such as $\Z/4\Z$ and $(\Z/2\Z)^2 \times (\Z/3\Z)$, we use arguments that are very much specific to the geometry of bielliptic Picard curves (see Propositions \ref{proposition: no order 12} and \ref{proposition: no 4-torsion}).

\subsection{Previous results}
    Petkova--Shiga \cite{PetkovaShiga} considered bielliptic Picard curves from a complex analytic viewpoint.
    Using period matrix computations, they show that their Pryms over $\C$ have quaternionic multiplication. Hashimoto-Murabayashi \cite{HashimotoMarabayashi} and Baba-Granath \cite{BabaGranath-genus2curvesQM} have studied genus two curves whose Jacobians $J'$ have QM by $\calO$. Over $\overline{k}$, each surface $J'$ becomes isomorphic to a bielliptic Picard Prym $P$, but generally not over $k$.  There is also work of Bonfanti-van Geemen \cite{BonfantivanGeemen} which gives an algebraic description of the embedding of the Shimura curve inside the moduli space of $(1,2)$-polarized abelian surfaces.

 % Our first bit of evidence that Theorem \ref{thm:ceresa} might be true was the observation of Lilienfeldt and the second author that the Abel--Jacobi image of $\kappa_\infty(C_{0,b})$ is torsion \cite{LilienfeldtShnidman}. 
 % % The branch locus $X \cup O_E \subset E$ in this case is equal to $E[2]$. 
 % The theorem is also inspired by \cite[\S 1.3]{ShnidmanManinDrinfeld} via an analogy between bielliptic Picard curves over $E$ and shtukas over a curve over $\F_p$ (where the $\mu_3$-action plays the role of Frobenius). 
 % \jef{Not sure what this random paragraph is doing here.}

\subsection{Future work}
There are many other avenues of research related to bielliptic Picard curves and their Pryms that are ripe for investigation. In \cite{LagaShnidmanCeresa}, we use our description of the $\Gal_k$-module $\End(P_{k^{\sep}})$ to characterize the bielliptic Picard curves with torsion Ceresa cycle. One can also try to study the average rank of the Mordell-Weil group $P(\Q)$; the average rank of $P(\Q)$ in cubic and sextic twist families was already considered in \cite{AlpogeBhargavaShnidman} and \cite{ShnidmanWeissRankGrowth}.
To help explore other arithmetic questions, it would be worthwhile to develop more robust algorithms that are tailored towards these curves, such as explicit descent algorithms, explicit addition laws, an analogue of Tate's algorithm, $L$-functions, etc.  We hope this work encourages others to study some of these topics.

\subsection{Structure of paper} 
We start with some generalities on bielliptic Picard curves and their (dual) Prym varieties in \S\ref{section: Bielliptic picard curves}. In particular, we discuss bigonal duality in \S\ref{subsection: bigonal duality}.
We prove our Torelli theorem in \S\ref{section: Torelli}.
In \S\ref{connections to Shimura curves} we make the connection with unitary and quaternionic Shimura curves.
In  \S\ref{sec: identifying bielliptic pic curves in the pencil} we perform some explicit calculations with a pencil of curves on the Prym variety. 
These calculations are used in \S\ref{endomorphisms of the Prym variety} to explicitly construct the quaternionic multiplication.
In \S\ref{sec:2,3,6-torsion} we give explicit descriptions of the $2$ and $3$-torsion subgroups of Pryms of bielliptic Picard curves.
In \S\ref{sec: classifying rational torsion subgroups} and \S\ref{sec:gl2type} we classify rational torsion subgroups and prove Theorems \ref{thm:main}
and \ref{thm:gl2main}.

\subsection{Acknowledgements}
We thank Adam Morgan, Ciaran Schembri, Michael Stoll, John Voight for helpful conversations and remarks.  The second author was funded by the European Research Council (ERC, CurveArithmetic, 101078157).
Part of this research was carried out while the first author was visiting the second author in Jerusalem. We thank the Hebrew University of Jerusalem and the Einstein Institute of Mathematics for their hospitality.

\subsection{Notation and conventions}\label{subsec: notation and conventions}

\begin{itemize}
    \item Our base field will typically be denoted by $k$, with choice of separable and algebraic closure $k^{\sep} \subset \bar{k}$ and absolute Galois group $\Gal_k = \Gal(k^{\sep}/k)$.
    All fields in this paper will be assumed to have characteristic $\neq 2,3$ unless explicitly stated otherwise.
    Galois actions will typically be right actions.
    \item A \define{variety} over a field $k$ is a separated finite type scheme over $k$.
    A variety is called \define{nice} if it is smooth, projective and geometrically integral.
    \item If $X,Y/k$ are varieties, $f\colon X_{k^{\sep}} \rightarrow Y_{k^{\sep}}$ is a morphism of $k^{\sep}$-varieties and $\sigma \in \Gal_k$, then $f^{\sigma}$ denotes the $k^{\sep}$-morphism $x\mapsto f(x^{\sigma^{-1}})^{\sigma}$.
    \item For a nice variety $X/k$, denote by $\Pic(X)$ its Picard group and by $\Pic_X$ its Picard scheme. 
    \item If $X/k$ is a nice curve, let $\Pic^n (X) \subset\Pic (X)$ be the subset of line bundles of degree $n$, let $\Jac_X = \Pic^0_X$ denote its Jacobian, let $\Aut(X)$ denote the $k$-automorphism group of $X$ and let $\mathbf{Aut}(X)$ be the \emph{scheme} of automorphisms of $X$, with $\mathbf{Aut}(X)(K) = \Aut(X_K)$ for every field extension $K/k$.%(the latter is finite etale over $k$ \cite[Theorem (I.11)]{DeligneMumford-Mg}).
    %\item If $X/k$ is a nice curve of genus $\geq 2$, let $\Aut(X)$ be the set of $k$-automorphisms $X\rightarrow X$ and let $\mathbf{Aut}(X)$ be the \emph{scheme} of automorphisms of $X$; it is finite etale over $k$ \cite[Theorem (I.11)]{DeligneMumford-Mg}.
    \item For every integer $n\geq 1$ we define the group scheme $\mu_n = \Spec\left(k[t]/(t^n-1)\right)$. 
    If $X/k$ a variety, we define a \define{$\mu_n$-action} on $X$ to be a morphism of $k$-schemes $\mu_n \times_k X\rightarrow X$ satisfying the axioms of a (left) group action. 
    If $m,n$ are coprime, giving a $\mu_{mn}$-action is the same as giving commuting $\mu_m$ and $\mu_n$-actions, using the isomorphism $\mu_m\times \mu_n\rightarrow \mu_{mn}$ induced by the inclusion maps.
    \item We will typically write $\omega \in \bar{k}$ for a choice of primitive third root of unity and $k(\omega)$ for the smallest field extension of $k$ containing such $\omega$.
    \item If $V$ is a vector space over a field $k$, we let $\P(V) = \Proj (\Sym^{\bullet}(V^{\vee}))$ be the projective space parametrizing lines in $V$.
    If $X/k$ is variety, a morphism $X\rightarrow \P(V)$ is the same as a pair $(\sheaf{L}, \phi)$, where $\sheaf{L}$ is a line bundle on $X$ and $\phi \colon V^{\vee} \otimes_{k} \mathcal{O}_X \rightarrow \sheaf{L}$ is a surjection (up to a suitable notion of isomorphism).
    \item If $\sheaf{L}$ is a line bundle on a variety $X/k$ we denote by $|\sheaf{L}|:=\P(\HH^0(X, \sheaf{L})^{\vee})$ the linear system of effective divisors in $\sheaf{L}$.
    If $D$ is a Cartier divisor on $X$ we also write $|D|$ for $|\mathcal{O}_X(D)|$.
    If $\sheaf{L}$ has no base points, we thus get a natural morphism $\phi_{\sheaf{L}} \colon X\rightarrow |\sheaf{L}|$.
    \item If $A/k$ is an abelian variety, let $A^\vee =\Pic^0_A$ be its dual abelian variety. If $\sheaf{L} \in \Pic(A)$, denote by $\lambda_{\sheaf{L}} \colon A \to A^\vee$ the homomorphism $x \mapsto t_x^*\sheaf{L} \otimes \sheaf{L}^{-1}$.
\end{itemize}

\section{Bielliptic Picard curves and their Pryms}\label{section: Bielliptic picard curves}

\subsection{Basic definitions}\label{subsec: basic definitions}
Let $k$ be a field (always assumed of characteristic $\neq 2,3$).
\begin{definition}
    Let $C/k$ be a nice genus $3$ curve.
    We call a faithful $\mu_6$-action $\gamma \colon \mu_6 \hookrightarrow \mathbf{Aut}(C)$ on $C$ \define{bielliptic} if the quotient $C/\mu_2$ is a genus one curve.
    We say $C$ is a \define{bielliptic Picard curve} over $k$ if it admits a bielliptic $\mu_6$-action.
\end{definition}

For $a,b\in k$, consider the projective plane curve $C_{a,b}$ over $k$ with affine equation 
\begin{align}\label{equation: bielliptic picard curve equation}
  C_{a,b}\colon y^3 = x^4+ax^2+b.  
\end{align}
The curve $C_{a,b}$ has a unique ($k$-rational) point $\infty$ at infinity,  and $C_{a,b}$ is smooth if and only if $$\Delta_{a,b}:= 16b(a^2-4b)$$ is nonzero.
The $\mu_2$-action $\pm 1 \cdot (x,y)= (\pm x,y)$ and $\mu_3$-action $\omega\cdot (x,y) = (x,\omega y)$ (where $\omega^3=1$) combine to a faithful bielliptic $\mu_6$-action $\gamma_{a,b}$, given by $\zeta \cdot (x,y)= (\zeta^3 x ,\zeta^4 y)$, for every $\zeta \in \mu_6(k^{\sep})$. The unique fixed point for this action is $\infty$.

\begin{theorem}\label{theorem: curves with mu6 action are of the form Ca,b}
Let $C/k$ be nice genus $3$ curve with faithful $\mu_6$-action $\gamma$.
Then the following conditions are equivalent:
\begin{enumerate}[\hspace{3mm}$(1)$]
    \item $\gamma$ is a bielliptic $\mu_6$-action (and consequently $C$ is a bielliptic Picard curve).
    \item $\gamma$ has a unique fixed point.
    \item $C/\mu_3$ has genus zero.
    \item $C$ is nonhyperelliptic.
    \item There is an isomorphism $C \simeq C_{a,b}$ for some $a,b\in k$ such that $\gamma$ corresponds to $\gamma_{a,b}$ or $\gamma_{a,b}^{-1}$.
\end{enumerate}

\end{theorem}
\begin{proof}
For each $i$ dividing $6$, let $n_i$ be the number of $\mu_i$-fixed points. By Riemann-Hurwitz, we have $n_2 \in \{0,4,8\}$, $n_3 \in \{2,5\}$, and $n_6 \leq 3$. Moreover, $\mu_2$ permutes the $\mu_3$-fixed points and vice versa. So $n_2 - n_6$ is divisible by $3$ and $n_3 - n_6$ is divisible by $2$. There are then three possibilities:
\begin{enumerate}[\hspace{3mm}$(a)$]
    \item  $(n_2,n_3,n_6) = (0,2,0)$, or
    \item  $(n_2,n_3,n_6) = (4,5,1)$, or
    \item   $(n_2,n_3,n_6) = (8,2,2)$.
\end{enumerate}
The equivalence of $(1)$, $(2)$, and $(3)$ follows immediately from the above cases. To show the equivalence of $(1)$ and $(4)$, first observe that case $(c)$ is hyperelliptic because $C/\mu_2$ is genus $0$. Case $(a)$ is also hyperelliptic since an unramified double cover of a genus two curve is again hyperelliptic \cite{Farkas}. It remains to show that $C$ is non-hyperelliptic in case $(b)$. In that case, there is only one $\mu_6$-fixed point $\infty$, which is therefore $k$-rational. Moreover, $\infty$ must have different local monodromy type for the $\mu_3$-cover $C \to C/\mu_3 \simeq \P^1$ compared to the other four ramification points (which are permuted by the $\mu_2$-action). By Kummer theory, there is a model $C \colon y^3 = f(x)$, for some quartic polynomial $f(x) \in k[x]$, and in particular $C$ is a plane quartic curve. This shows that $(1)-(4)$ are equivalent. 

It is clear that $(5)$ implies all the rest, so it remains to show the converse. However, we have seen that such a $C$ has a model $y^3 = f(x)$ which we may assume takes the form 
\begin{equation}\label{eq:picard}
y^3 = x^4 + ax^2 + cx + b
\end{equation}
with the $\mu_3$-action $\omega\cdot (x,y) =(x,\omega y)$ or $\omega\cdot (x,y)=(x,\omega^{-1}y)$ for every $\omega\in \mu_3$. 
We see from this model (as in the proof of \cite[Lemma 1.21(a)]{BKSW}) that $C$ admits an  involution commuting with the $\mu_3$-action only if $c = 0$ and this involution is given by $(x,y)\mapsto (-x,y)$, proving $(5)$.
\end{proof}

% \begin{remark}
% {\em
% A  \define{Picard curve} is a plane quartic curve of the form \eqref{eq:picard}. 
% Such a curve is bielliptic (admitting a $\mu_3$-equivariant double cover to an elliptic curve) precisely when $c = 0$, explaining our terminology. 
% }

% \end{remark}

There is one $k^{\sep}$-isomorphism class of bielliptic Picard curves that behaves so differently that it deserves its own name. 
\begin{definition}
    A bielliptic Picard curve $C$ is \define{special} if it is $k^{\sep}$-isomorphic to $C_{0,1}: y^3 =x^4 +1$.
\end{definition}
A bielliptic Picard curve $C$ is special precisely when $\Aut(C_{k^{\sep}})$ is larger than $\Z/6\Z$; in that case, it is a group of order $48$ \cite[Theorem 3.1]{LercierRitzRovettaSijslin-planequartic}, with \texttt{GAP} label \href{https://people.maths.bris.ac.uk/~matyd/GroupNames/1/C4.A4.html}{(48,33)}. 
We will see in Lemma \ref{lemma: bielliptic picard special iff a=0} that $C_{a,b}$ is special if and only if $a = 0$.

We say $C$ a nice genus $3$ curve is a Picard curve if it admits a faithful $\mu_3$-action such that the quotient $C/\mu_3$ is isomorphic to $\P^1$. 
We say $C$ is a bielliptic curve if it admits an involution $\tau$ such that the quotient $C/\tau$ is a genus $1$ curve.
To avoid notational conflicts, we show:
\begin{lemma}
    If a nice genus $3$ curve $C/k$ is both a Picard curve and a bielliptic curve, then it is a bielliptic Picard curve in the sense of Definition \ref{subsec: basic definitions}. 
\end{lemma}
\begin{proof}
    Since $C$ is a Picard curve, it has an affine equation of the form $y^3 = f(x)$ for some monic quartic $f(x) =x^4+ ax^2+bx+c\in k[x]$ with nonzero discriminant, and the $\mu_3$-action on $C$ is given by $\omega\cdot (x,y) = (x, \omega y)$.
    By assumption, there exists an involution $\tau\colon C\rightarrow C$ such that $C/\tau$ has genus $1$.
    Let $C_0$ be the Picard curve with equation $y^3 = x^4+1$.
    If $C_{k^{\sep}}$ is not isomorphic to $C_{0, k^{\sep}}$, then a calculation using \cite[Lemma 1.21(a)]{BKSW} implies that $b=0$ and that $\tau$ maps $(x,y)$ to $(-x,y)$, so $C$ is indeed a bielliptic Picard curve.
    We may therefore assume that $C_{k^{\sep}}\simeq C_{0, k^{\sep}}$.
    Since every two order $3$ subgroups of $\Aut(C_{0,k^{\sep}})$ are conjugate (by an explicit computation in the group \href{https://people.maths.bris.ac.uk/~matyd/GroupNames/1/C4.A4.html}{(48,33)}), there is an isomorphism $C_{k^{\sep}}\xrightarrow{\sim} C_{0, k^{\sep}}$ that preserves the $\mu_3$-actions on both sides.
    Such an isomorphism must be of the form $(x,y)\mapsto (\alpha x, \beta y)$ for some $\alpha, \beta \in k^{\sep}\setminus\{0\}$, using the same reasoning as the proof of \cite[Lemma 1.21]{BKSW}.
    It follows that $a=b=0$ and $C: y^3 = x^4 + c$ is a (special) bielliptic Picard curve.
\end{proof}

\subsection{Marked bielliptic Picard curves}

It is often useful to pin down a $\mu_6$-action on a bielliptic Picard curve with specified signature around the (unique, by Theorem \ref{theorem: curves with mu6 action are of the form Ca,b}) $\mu_6$-fixed point $\infty$.

\begin{definition}
A \define{marked bielliptic Picard curve} is a bielliptic Picard curve $C$ with bielliptic $\mu_6$-action $\gamma$ such that $\mu_6$ acts on the tangent space $T_{\infty}C$ via the identity character $\mu_6\rightarrow \mu_6$. 
\end{definition}
An isomorphism of marked bielliptic Picard curves $(C, \gamma)\rightarrow (C',\gamma')$ is by definition an isomorphism of curves $C\rightarrow C'$ that is equivariant with respect to the $\mu_6$-actions.
A calculation using the generator $\frac{x}{y}\in (T_{\infty}C_{a,b})^{\vee}$ shows that the pair $(C_{a,b}, \gamma_{a,b})$ defined in \S\ref{subsec: basic definitions} is a marked bielliptic Picard curve.

\begin{lemma}\label{lemma: bielliptic Picard can be marked in unique way}
    Let $C/k$ be a bielliptic Picard curve with bielliptic $\mu_6$-action $\gamma$. 
    Then one of $(C, \gamma^{\pm 1})$ is a marked bielliptic Picard curve, and if $C$ is not special this marked $\mu_6$-action is the unique one.
\end{lemma}
\begin{proof}
    Since $\gamma$ is faithful, it acts on $T_{\infty}C$ via the identity character or its inverse, so either $\gamma$ or $\gamma^{-1}$ acts via the identity. 
    If $C$ is not special, $\mathbf{Aut}(C) \simeq \mu_6$ by \cite[Theorem 3.1]{LercierRitzRovettaSijslin-planequartic}, proving the uniqueness in that case.
\end{proof}

\subsection{Automorphisms}\label{subsection: automorphisms}

\begin{lemma}\label{lemma: all isos between marked bielliptic Picard curves}
    Every isomorphism $(C_{a,b}, \gamma_{a,b}) \rightarrow (C_{a',b'}, \gamma_{a',b'})$ of marked bielliptic Picard curves over $k$ is of the form $(x,y)\mapsto (\lambda^3x,\lambda^4 y)$ for some $\lambda \in k^{\times}$ such that $(a',b') = (\lambda^6 a,\lambda^{12} b)$.
\end{lemma}
\begin{proof}
    Let $\phi\colon C_{a,b} \rightarrow C_{a',b'}$ be an isomorphism preserving the $\mu_6$-actions.
    Since $C_{a,b}$ and $C_{a',b'}$ are canonically embedded, $\phi$ is induced by a linear isomorphism of the ambient projective space $\P^2_k$.
    This isomorphism preserves the point at infinity (being the unique $\mu_6$-fixed point) and the $\mu_6$-eigenspaces of $\HH^0(\P^2, \calO(1)) = \text{span}\{1,x,y\}$. 
    A short calculation shows that $\phi$ must be of the form above.
\end{proof}

\begin{lemma}\label{lemma: when are two bielliptic Picard curves isomorphic}
Every marked bielliptic Picard curve over $k$ is of the form $(C_{a,b}, \gamma_{a,b})$ for some $a,b\in k$.
Two marked bielliptic Picard curves $(C_{a,b}, \gamma_{a,b})$ and $(C_{a',b'}, \gamma_{a',b'})$ are isomorphic if and only if there exists $\lambda\in k^{\times}$ such that $a' = \lambda^6 a$ and $b' = \lambda^{12} b$.
\end{lemma}
\begin{proof}
    The first part follows from Theorem \ref{theorem: curves with mu6 action are of the form Ca,b}(5).
    The second part follows from Lemma \ref{lemma: all isos between marked bielliptic Picard curves}.
\end{proof}

\begin{lemma}\label{lemma: bielliptic picard special iff a=0}
The bielliptic Picard curve $C_{a,b}$ is special if and only if $a = 0$.    
\end{lemma}
\begin{proof}
    We may assume that $k$ is algebraically closed.
    Lemma \ref{lemma: when are two bielliptic Picard curves isomorphic} shows that $C_{a,b}$ is special if $a=0$.
    Conversely, suppose that $C_{a,b}$ is special.
    By \cite[Theorem 3.1]{LercierRitzRovettaSijslin-planequartic}, $G = \Aut(C_{a,b})$ is a group of order $48$ with \texttt{GAP} label \href{https://people.maths.bris.ac.uk/~matyd/GroupNames/1/C4.A4.html}{(48,33)}.
    Moreover, $\Aut(C_{a,b}, \gamma_{a,b})$ is the centralizer in $G$ of a cyclic order $6$ subgroup. 
    A group theory calculation similar to \cite[Lemma 4.1.5(c)]{BKSW} shows that all cyclic order $6$ subgroup of $G$ are conjugate and have centralizer of order $12$, so $\Aut(C_{a,b}, \gamma_{a,b})$ has order $12$.
    On the other hand, Lemma \ref{lemma: all isos between marked bielliptic Picard curves} shows that $\Aut((C_{a,b}, \gamma_{a,b})) = \{\lambda \in k^{\times} \mid (a,b) = (\lambda^6 a,\lambda^{12} b)\}$.
    If $a\neq 0$, the latter set has size $6$, so if $C$ is special then $a=0$.
\end{proof}

\begin{lemma}\label{lemma: automorphism group bielliptic Picard curves}
    Let $(C, \gamma)$ be a marked bielliptic Picard curve over $k$.
    Then $\mathbf{Aut}(C, \gamma) = \mu_6$ if $C$ is not special and $\mathbf{Aut}(C, \gamma) = \mu_{12}$ if $C$ is special.
\end{lemma}
\begin{proof}
    Combine Lemmas \ref{lemma: all isos between marked bielliptic Picard curves} and \ref{lemma: bielliptic picard special iff a=0}.
\end{proof}

Lemma \ref{lemma: when are two bielliptic Picard curves isomorphic} shows that the moduli stack of marked bielliptic Picard curves can be identified with the weighted projective stack $\P(6,12)$ minus the discriminant locus $\Delta_{a,b}=0$.
As with the moduli stack of elliptic curves, the coarse space has a simpler description.
Given a bielliptic Picard curve $C/k$, its \define{$j$-invariant} is defined as 
\begin{align}\label{equation: j-invariant}
j(C) := j_{a,b} := \frac{4b-a^2}{4b} \in k \setminus \{0\},
\end{align}
where $a,b\in k$ are such that $C \simeq C_{a,b}$.
Lemmas \ref{lemma: bielliptic Picard can be marked in unique way}, \ref{lemma: when are two bielliptic Picard curves isomorphic} and \ref{lemma: bielliptic picard special iff a=0} show that this is well defined and that $C_{k^{\sep}} \simeq C'_{k^{\sep}}$ if and only if $j(C) = j(C')$, even if $C$ or $C'$ is special.

\subsection{Twists}\label{subsection: twists}
Let $(C, \gamma)$ be a marked bielliptic Picard curve over $k$.
By the twisting principle, every cocycle $\xi \in \HH^1(k, \mathbf{Aut}(C, \gamma))$ determines a bielliptic Picard curve $(C_{\xi}, \gamma_{\xi})$ that is $k^{\sep}$-isomorphic to $(C, \gamma)$.
Its $k$-isomorphism class is characterized by the existence of a $k^{\sep}$-isomorphism $\phi\colon(C_{\xi}, \gamma_{\xi})_{k^{\sep}} \rightarrow  (C, \gamma)_{k^{\sep}}$ such that $\sigma \mapsto \phi^{\sigma} \circ \phi^{-1}$ represents the class of $\xi$. 

In particular, since $\gamma$ determines a subgroup $\mu_6 \subset \mathbf{Aut}(C, \gamma)$ and since $\HH^1(k, \mu_6) \simeq k^{\times}/k^{\times 6}$, every $\delta\in k^{\times}$ gives rise to a \define{sextic twist} $(C_{\delta}, \gamma_{\delta})$ of $(C, \gamma)$.
%Let's spell this out when $(C, \gamma) = (C_{a,b}, \gamma_{a,b})$.
%Given $\delta \in k^{\times}$, let $\epsilon \in k^{\sep}$ be such that $\delta = \epsilon^6$.
%Then the relevant cocycle is $\xi_{\sigma} = \epsilon^{\sigma}/\epsilon$.
%There exists an isomorphism $\phi\colon C_{\delta a, \delta^2 b} \rightarrow C_{a, b}, (x,y)\mapsto (\epsilon^3x, \epsilon^4y)$ over $k^{\sep}$, and $\phi^{\sigma}\circ \phi^{-1}$ is given by $(x, y) \mapsto (\xi_{\sigma}^3x, \xi_{\sigma}^4y)$. 
Concretely, if $C = C_{a,b}$ then the sextic twist of $C_{a,b}$ by $\delta$ is isomorphic to $C_{\delta a, \delta^2 b}$.
If $C$ is not special, Lemmas \ref{lemma: bielliptic Picard can be marked in unique way} and \ref{lemma: all isos between marked bielliptic Picard curves} show that $\mathbf{Aut}(C,\gamma) = \mathbf{Aut}(C) = \mu_6$ and so every twist of $C$ (as a marked and unmarked curve) is isomorphic to a sextic twist.

\subsection{The Prym variety \texorpdfstring{$P$}{P}}\label{subsection: the prym variety P}

Let $(C,\gamma)$ be a marked bielliptic Picard curve over $k$. We associate to $C$ an abelian surface $P$ whose study will occupy much of the rest of this paper.

Restricting the $\mu_6$-action to $\mu_2$ defines an involution $\tau\colon C\rightarrow C$.
The quotient $\pi \colon C\rightarrow E := C/\tau$ is a genus $1$ curve.
When endowed with the $k$-point $\pi(\infty)$ it has the structure of an elliptic curve, so we can identify $\Jac_E$ with $E$.
Let $J =\Jac_C$ be the Jacobian of $C$.
\begin{definition}
     The \define{Prym variety} of $C$ is defined as $P := \ker(1+\tau^*\colon J \rightarrow J)$.
\end{definition}
This is an abelian surface. 
We refer to \cite[\S3.3]{Laga-Prymsurfaces} for more details concerning Prym varieties of bielliptic genus $3$ curves.
The properties we will use can be deduced from \cite{Mumford-prymvars} and are summarized in the following commutative diagram:
% https://q.uiver.app/?q=WzAsMTAsWzIsMiwiSiJdLFsxLDIsIlAiXSxbMCwyLCIwIl0sWzQsMiwiMCJdLFszLDIsIkUiXSxbMiwzLCJBIl0sWzIsMSwiRSJdLFsyLDAsIjAiXSxbMiw0LCIwIl0sWzEsMSwiRVsyXSJdLFs5LDEsIiIsMCx7InN0eWxlIjp7InRhaWwiOnsibmFtZSI6Imhvb2siLCJzaWRlIjoidG9wIn19fV0sWzksNiwiIiwyLHsic3R5bGUiOnsidGFpbCI6eyJuYW1lIjoiaG9vayIsInNpZGUiOiJ0b3AifX19XSxbNyw2XSxbNiwwLCJcXHBpXioiLDJdLFswLDVdLFs1LDhdLFsyLDFdLFsxLDBdLFswLDQsIlxccGlfKiJdLFs0LDNdLFs2LDQsIlxcdGltZXMgMiJdLFsxLDUsIlxcbGFtYmRhIiwyXV0=
\[\begin{tikzcd}
	&& 0 \\
	& {E[2]} & E \\
	0 & P & J & E & 0 \\
	&& A \\
	&& 0
	\arrow[hook, from=2-2, to=3-2]
	\arrow[hook, from=2-2, to=2-3]
	\arrow[from=1-3, to=2-3]
	\arrow["{\pi^*}"', from=2-3, to=3-3]
	\arrow[from=3-3, to=4-3]
	\arrow[from=4-3, to=5-3]
	\arrow[from=3-1, to=3-2]
	\arrow[from=3-2, to=3-3]
	\arrow["{\pi_*}", from=3-3, to=3-4]
	\arrow[from=3-4, to=3-5]
	\arrow["{\times 2}", from=2-3, to=3-4]
	\arrow["\lambda"', from=3-2, to=4-3]
\end{tikzcd}\]
It depicts two dual exact sequences, where $\pi_*$ (resp.\ $\pi^*$) denotes pushforward (resp.\ pullback) of divisors.
The principal polarization on $J$ restricts to a $(1,2)$-polarization $\lambda \colon P\rightarrow A := P^{\vee}$ with kernel $P[\lambda](k^{\sep}) \simeq (\Z/2\Z)^2$.
The intersection of $P$ and $\pi^*(E)$ is $\pi^*(E[2]) = P[\lambda]$.
The sum of the two maps $P\rightarrow J$ and $E\rightarrow J$ is an isogeny $P\times E\rightarrow J$ with kernel $\{(\pi^*(x),x)\mid x\in E[2]\}\simeq E[2]$.

If $C=C_{a,b}$ for some $a,b\in k$, we may write $E_{a,b}$, $J_{a,b}$, $\lambda_{a,b}\colon P_{a,b}\rightarrow A_{a,b}$ et cetera, but we drop the subscripts when convenient.

\begin{remark}\label{rem: genus two curve}
{\em
    $P$ need not be principally polarizable over $k$; indeed we  prove in Corollary \ref{cor:principal polarization} that if $P$ is geometrically simple, then it is principally polarizable if and only if $16b(a^2 - 4b)$ is a sixth power in $k^\times$. 
    However, if $b = s^3$ is a cube in $k^\times$, then the Jacobian of the following genus two curve is in the isogeny class of $P_{a,s^3}$: 
\[-asy^2 = (x^2 + 2x - 2)(s^3x^4 + 4s^3x^3 + 2dx -d).\]
%\[-ay^2 = (x^2 + 2sx - 2s^2)(x^4 + 4sx^3 + 2dx + -sd).\]
where $d = a^2 - 4s^3$; this follows from \cite[Theorem 1.1]{RitzenthalerRomagny}.  
}
\end{remark}

We now incorporate the $\mu_3$-action on $C$ into the picture.
By Albanese functoriality, there is a unique $\mu_6$-action on $J$ such that the Abel--Jacobi map $\text{AJ}_{\infty}\colon C\hookrightarrow J$ is $\mu_6$-equivariant.
This induces a $\mu_3$-action on the subvariety $P$ and the quotient $E$.
We record the following important signature calculation. 
\begin{lemma}\label{lemma: mu3-action on P has char poly T^2+T+1}
    For every primitive third root of unity $\omega\in k^{\sep}$, the action of $\omega$ on $\HH^0(P_{k^{\sep}}, \Omega^1_{P_{k^{\sep}}})$ has characteristic polynomial $T^2+ T+ 1$.  
\end{lemma}
\begin{proof}
    We may assume that $(C, \gamma) = (C_{a,b}, \gamma_{a,b})$ by Lemma \ref{lemma: when are two bielliptic Picard curves isomorphic} and that $k = k^{\sep}$, so fix such a $\omega \in k$. 
    Then $\omega$ acts on $C_{a,b}$ via $(x,y)\mapsto (x,\omega y)$.
    The vector space $\HH^0(C, \Omega^1_C)$ has basis $\frac{dx}{y^2}, x\frac{dx}{y^2}, \frac{dx}{y}$, hence the $\omega$-action on this vector space has eigenvalues $\omega, \omega, \omega^2$.
    The vector space $\HH^0(E, \Omega^1_E)$ has basis $\frac{dx}{y^2}$, hence $\omega$ has eigenvalue $\omega$.
    Since the map $P\times E\rightarrow J, (p,e)\mapsto p+\pi^*(e)$ is an isogeny with kernel $E[2]$, it induces a $\mu_3$-equivariant isomorphism $\HH^0(J, \Omega^1_J) \simeq \HH^0(P, \Omega^1_P) \oplus \HH^0(E, \Omega^1_E)$.
    Combining the last three sentences with the isomorphism $\HH^0(C, \Omega^1_C) \simeq \HH^0(J, \Omega_J^1)$ proves the lemma.
\end{proof}

\begin{lemma}\label{lemma: mu3-fixed points Prym}
    The subgroup of $\mu_3$-fixed points of $P$ is of size $9$ and contained in $P[3]$.
\end{lemma}
\begin{proof}
    We may assume that $k = k^{\sep}$. 
    So let $\omega \in k$ be a third root of unity and $\alpha\colon P\rightarrow P$ the corresponding automorphism. 
    Let $\beta = 1-\alpha$.
    By Lemma \ref{lemma: mu3-action on P has char poly T^2+T+1}, $\beta$ induces an isomorphism on differentials so $\beta$ is an isogeny. 
    Since $\beta \circ (\alpha^2 + \alpha + 1) =0$ and $\beta$ is surjective, 
    \begin{align}\label{equation: endomorphism satisfies r^2+r+1=0}
       \alpha^2+ \alpha +1=0. 
    \end{align}
    Therefore $\beta^2 = 1-2\alpha+\alpha^2 = -3\alpha$, so $\beta$ has degree $9$ and $P[\beta]\subset P[-3\alpha] = P[3]$.
\end{proof}

\subsection{The dual Prym variety \texorpdfstring{$A$}{A}} \label{subsection: the dual prym A}

Keep the notations from \S\ref{subsection: the prym variety P}.
Since $P[\lambda] \subset P[2]$, there exists a unique isogeny\footnote{We warn the reader that $\widehat{\lambda}$ is not the same as the dual $\lambda^\vee$ of $\lambda$. Being a polarization, $\lambda$ is in fact self-dual!} $\widehat{\lambda}\colon A\rightarrow P$ such that $\widehat{\lambda} \circ \lambda = [2]$.
The next proposition describes this isogeny geometrically and characterizes $A$ in terms of the curve $C$ with its involution $\tau\colon C\rightarrow C$.

\begin{proposition}\label{proposition: properties of the embedding of C into A}
Let $i\colon C\rightarrow A$ be the composite of the Abel--Jacobi map $C\rightarrow J$ with respect to $\infty \in C(k)$ and the projection map $J\rightarrow A$. 
\begin{enumerate}
    \item The morphism $i\colon C\hookrightarrow A$ is a closed embedding.
    \item The divisor $i(C)$ is ample and induces the polarization $\widehat{\lambda}$.
    \item If $B/k$ is an abelian surface and $j\colon C\hookrightarrow B$ a closed embedding mapping $\infty$ to $0$ and such that $[-1]\circ j = j \circ\tau$, then there exists a unique isomorphism of abelian surfaces $\phi\colon A\rightarrow B$ such that $\phi \circ i  = j$.
\end{enumerate}
\end{proposition}
\begin{proof}
    This is due to Barth \cite{Barth-abeliansurfacespolarization}; see \cite[Proposition 3.5]{Laga-Prymsurfaces} for a detailed proof. 
\end{proof}

Let $\Sym^2 C$ be the symmetric square of $C$, a nice surface over $k$ parameterizing effective divisors of degree $2$ on $C$.
In analogy with Jacobians of genus $2$ curves, we will describe the fibres of the surjective map $i^{(2)}\colon \Sym^2 C\rightarrow A$ given by $p+p' \mapsto i(p) + i(p')$; this is due to Ikeda \cite{Ikeda-Biellipticcurvesofgenusthree}.

To this end, we define two involutions on $\Sym^2 C$. 
The first is  $\tau^{(2)}(p + p') = \tau(p) + \tau(p')$.  
The second involution $\kappa$ sends $p+p'$ to the unique effective degree $2$ divisor linearly equivalent to $4\infty -p-p'$; this is well defined by Riemann--Roch and the fact that $C$ is not hyperelliptic.
Since $\tau(\infty) = \infty$, the involutions $\tau^{(2)}$ and $\kappa$ on $\Sym^2 C$ commute.
Finally, the double cover $\pi\colon C\rightarrow E$ induces a map $\pi^* \colon E \rightarrow \Sym^2 C$, giving an embedding $E \simeq \pi^*(E)  \hookrightarrow \Sym^2 C$.

\begin{proposition}[Ikeda]\label{proposition: A is a quotient of Sym^2C}
The map $i^{(2)} \colon \Sym^2 C\rightarrow A$ contracts $\pi^*(E)$ to the origin. 
Two points $D, D' \in \Sym^2 C$ not in $\pi^*(E)$ map to the same point of $A$ if and only if $D' = \kappa(\tau^{(2)}(D))$, i.e. $D' + \tau(D) \sim 4\infty$.
\end{proposition}
\begin{proof}
This follows from the proof of \cite[Lemma 3.1]{Ikeda-Biellipticcurvesofgenusthree}. 
\end{proof}

We see that $A$ is obtained from $\Sym^2 C/\langle \tau^{(2)} \circ \kappa\rangle$ by contracting the rational curve $\pi^*(E)/\langle -1\rangle$. This description will be useful in Section \ref{subsec: 2-torsion in A using bitangents}, where we relate the bitangents of $C$ to the group $A[2]$. 

\subsection{Bigonal duality}\label{subsection: bigonal duality}

Given a marked bielliptic Picard curve $(C, \gamma)$, it turns out that we can define another marked bielliptic Picard curve $(\widehat{C}, \hat{\gamma})$ such that the role of the Prym variety and its dual are reversed: the Prym variety of $\widehat{C}$ is isomorphic to $A = P^{\vee}$.
This is called the bigonal dual of $C$, originally defined by Pantazis \cite{Pantazis-Prymvarsgeodesicflow} (inspired by ideas of Donagi \cite{Donagi-tetragonalconstruction}) and analyzed by Barth \cite{Barth-abeliansurfacespolarization} in this specific situation; see \cite[\S3.4]{Laga-Prymsurfaces} for more details.
Let $\Theta_{2\infty}\subset J$ be the image of the Abel--Jacobi map $\Sym^2 C\rightarrow J$ sending $P + P'$ to $P+ P' -2\infty$.
\begin{definition}
    The \define{bigonal dual} of $C$ is defined by $\widehat{C} := P\cap \Theta_{2\infty}$.
\end{definition}

The $\mu_3$-action on $P$ restricts to a $\mu_3$-action on $\widehat{C}$.
Inversion $[-1]$ on $P$ restricts to an involution $\hat{\tau}$ on $\widehat{C}$.
Using the isomorphism $\mu_2\times \mu_3\rightarrow \mu_6$ induced by the two inclusions, we obtain a $\mu_6$-action $\hat{\gamma}$ on $\widehat{C}$.

\begin{lemma}\label{lemma: bigonal dual is a bielliptic Picard curve}
    The pair $(\widehat{C}, \hat{\gamma})$ is a marked bielliptic Picard curve.
\end{lemma}
\begin{proof}
    By \cite[Lemma 3.5]{Laga-Prymsurfaces}, $\widehat{C}$ is a nice genus $3$ curve and $\widehat{C}/\mu_2$ is a genus $1$ curve.
    Since the $\mu_3$-action on $P$ has only isolated fixed points (Lemma \ref{lemma: mu3-fixed points Prym}), the $\mu_3$-action on $\widehat{C}$ is nontrivial. 
    Therefore $\hat{\gamma}$ is a faithful bielliptic action and $\widehat{C}$ is a bielliptic Picard curve.
    The origin $0\in P(k)$ is a $\mu_6$-fixed point of $\widehat{C}$, which is unique by Theorem \ref{theorem: curves with mu6 action are of the form Ca,b}.
    
    It remains to show that $(\widehat{C}, \widehat{\gamma})$ is a marked bielliptic Picard curve, in other words that $\mu_6$ acts on $T_0\widehat{C}$ via the identity character. 
    Equivalently, we will show that $\mu_6$ acts on $(T_0\widehat{C})^{\vee}$ via the \emph{inverse} of the identity character.
    Choose a uniformizer $t \in \calO_{C, \infty}$.
    This induces a $k$-basis $\{t_1,t_2\}$ of $(T_{(\infty, \infty)}(C\times C))^{\vee} = (T_{\infty}C)^{\vee} \oplus (T_{\infty}C)^{\vee}$ and an isomorphism of completed local rings $\widehat{\calO}_{C\times C, (\infty, \infty)} \simeq k[[t_1,t_2]]$.
    In turn, this induces an isomorphism $\widehat{\calO}_{\Sym^2 C, 2\infty} \simeq k[[u,v]]$, where $u = t_1 + t_2$ and $v = t_1 t_2$.
    It follows that $(T_{2\infty} \Sym^2 C)^{\vee} \simeq (T_0 \Theta_{2\infty})^{\vee}$ has basis $\{u,v\}$ and $[-1]$ sends $u,v$ to $-u,v$ respectively.
    Therefore $T_0\widehat{C} = T_0 \Theta_{2\infty} \cap T_0 P$ has basis $\{u\}$ which indeed has the correct character.
\end{proof}
It follows that the objects defined in \S\ref{subsection: the prym variety P} also apply to $(\widehat{C}, \hat{\gamma})$, giving $\widehat{J} = \Jac_{\widehat{C}}$, $\widehat{P}$, $\widehat{A}$, $\hat{\pi}\colon \widehat{C} \rightarrow \widehat{E}$ in this context.
The inclusion $\widehat{C} \hookrightarrow P$ induces a homomorphism $\widehat{J}\rightarrow P$.

\begin{proposition}\label{proposition: bigonal duality}
    The homomorphism $\widehat{J} \rightarrow P$ factors through the projection $\widehat{J} \rightarrow \widehat{A}$ and induces an isomorphism $\widehat{A} \xrightarrow{
    \sim
    }P$ of $(1,2)$-polarized surfaces.
\end{proposition}
\begin{proof}
    See \cite[Proposition 3.6]{Laga-Prymsurfaces}.
    The fact that the isomorphism preserves the polarizations means that the maps $A\rightarrow P$ and $\widehat{P}\rightarrow \widehat{A}$ that are both denoted by $\widehat{\lambda}$ are the same under these identifications.
\end{proof}

Choosing an isomorphism $(C, \gamma) \simeq (C_{a,b}, \gamma_{a,b})$ for some $a,b\in k$ using Lemma \ref{lemma: when are two bielliptic Picard curves isomorphic}, we can write down equations for $\widehat{C}$ \cite[\S3.4 Eq (3.7), (3.8)]{Laga-Prymsurfaces}: 
\begin{align}
    C_{a,b}&: y^3 = x^4+ax^2+b, \\
    E_{a,b}&: y^2 = x^3+16(a^2-4b), \\
    \widehat{C}_{a,b} &: y^3 = x^4 + 8ax^2+16(a^2-4b), \label{equation: bigonal dual curve}\\ 
    \widehat{E}_{a,b} &: y^2 = x^3+b. 
\end{align}
The coordinates on $E_{a,b}$ and $\widehat{E}_{a,b}$ are chosen so that the double covers $C_{a,b}\rightarrow E_{a,b}$ and $\widehat{C}_{a,b}\rightarrow \widehat{E}_{a,b}$ are given by $(x,y) \mapsto (4y,8x^2+4a)$ and $(x,y) \mapsto (y/4,x^2/8+a/2)$ respectively. 

The third equation \eqref{equation: bigonal dual curve} shows that bigonal duality takes the following elegant form on the $j$-invariant \eqref{equation: j-invariant}:
\begin{lemma}\label{lemma: j-invariant bigonal dual}
    Let $C$ be a bielliptic Picard curve over $k$.
    Then $j(\widehat{C}) = 1/j(C)$. 
\end{lemma}

The quartic polynomial $\hat{f} := x^4+8ax^2+16(a^2-4b)$ defining the bigonal dual curve can also be constructed from the Galois theory of the original quartic $f := x^4+ax^2+b$. 
A calculation shows:
\begin{lemma}\label{lemma: bigonal dual Galois theory}
    Suppose that $f$ has roots $\{\pm \alpha, \pm \beta\}$ in $k^{\sep}$. 
    Then $\hat{f}$ has roots $\{\pm 2\alpha \pm 2\beta\}$.
    %Consequently, $\hat{f}$ has a $k$-rational root if and only if $f$ can be written as $q(x)q(-x)$ for some monic quadratic polynomial $q \in k[x]$.
\end{lemma}
This description of $\hat{f}$ will be useful when analyzing $3$-torsion in $P$ in \S\ref{subsec: sqrt3-torsion}.

\section{A Torelli theorem}\label{section: Torelli}

A bielliptic Picard curve $C$ gives rise to two $(1,2)$-polarized abelian surfaces with $\mu_3$-actions: the Prym surface $P$ and its dual $A$. We show that one can recover $C$ from $A$, together with its polarization and $\mu_3$-action; in other words, we prove a Torelli-type theorem (Theorem \ref{theorem: Torelli theorem bielliptic Picard curves}).  We do this by studying the pencil of genus three curves in the linear system $|C|$ on $A$.
In fact, we will prove that \emph{any} indecomposable $(1,2)$-polarized abelian surface with compatible $\mu_3$-action is the dual Prym variety of a bielliptic Picard curve, so we work in this generality below.

\subsection{Generalities on \texorpdfstring{$(1,2)$}{(1,2)}-polarized surfaces}\label{subsection: generalities (1,2)-polarized surfaces}

We recall and complement some results of Barth \cite{Barth-abeliansurfacespolarization}.
Barth works over $\C$, but the proofs of the basic results quoted here are valid in arbitrary characteristic $\neq 2$.

Let $(A, \lambda)$ be a $(1,2)$-polarized abelian surface over $k$.
We assume that $(A,\lambda)$ is indecomposable, that is, not isomorphic to a product of two elliptic curves $(E,O) \times (E',2O')$ as polarized abelian varieties.
Recall our notations concerning linear systems and abelian varieties in \S\ref{subsec: notation and conventions}.

\begin{lemma}\label{lemma: properties pencil of (1,2)-polarized surface}
Let $\sheaf{L} \in \Pic(A_{k^{\sep}})$ be a line bundle representing the polarization: $\lambda_{\sheaf{L}} = \lambda$.
\begin{enumerate}
    \item $\dim_k \HH^0(A, \sheaf{L}) = 2$.
    \item The linear system $|\sheaf{L}|$ has four base-points on which $A[\lambda](k^{\sep})$ acts simply transitively.
    \item At most finitely many members of the pencil are singular (arithmetic) genus $3$ curves, in which case they are curves of geometric genus $2$ with one node or the union of two genus $1$ curves $X_1, X_2$ with intersection number $(X_1\cdot X_2) = 2$.
\end{enumerate}
\end{lemma}
\begin{proof}
    See \cite[\S1.2]{Barth-abeliansurfacespolarization}.
\end{proof}
%This is all in the first section of \cite{Barth-abeliansurfacespolarization}.

\begin{lemma}\label{lemma: (1,2) surfaces unique symmetric line bundle and each divisor is symmetric}
There exists a unique $\sheaf{L} \in \Pic(A_{k^{\sep}})$ representing $\lambda$ whose base locus is $A[\lambda]$.
This line bundle is symmetric: $[-1]^*\sheaf{L} \simeq \sheaf{L}$. Moreover, for every divisor $D\in |\sheaf{L}|$ we have $[-1]^*D = D$.
\end{lemma}
\begin{proof}
Denote the base locus of a line bundle $\sheaf{L}$ by $\text{BL}(\sheaf{L})$.
Given $a \in A(k^{\sep})$, let $t_a\colon A \rightarrow A$ be the translation map and $\lambda(a) \in A^{\vee}(k^{\sep})=\Pic^0(A_{k^{\sep}})$ be its image under $\lambda$. 
Then for any $\sheaf{L}$ representing $\lambda$, the following identity holds:
\begin{align}\label{equation: base locus translation line bundles}
    \text{BL}(\lambda(a) \otimes \sheaf{L}) = t_a^*\text{BL}(\sheaf{L}).
\end{align}
Indeed, by definition $\lambda(a) \otimes \sheaf{L} = t_a^*\sheaf{L}$. 
Let $D,D'$ be two distinct members of the pencil $|\sheaf{L}|$. Then $t_a^*D, t_a^*D'$ are two distinct members of $|t_a^*\sheaf{L}|$.
So $\text{BL}(t_a^*\sheaf{L})  = t_a^*D \cap t_a^*D' = t_a^*(D\cap D') = t_a^*\text{BL}(\sheaf{L})$, proving \eqref{equation: base locus translation line bundles}.
This identity, together with Lemma \ref{lemma: properties pencil of (1,2)-polarized surface}(2), proves the existence and uniqueness of $\sheaf{L}$.
The remainder of the lemma follows from \cite[Proposition 1.6]{Barth-abeliansurfacespolarization}.
\end{proof}
For the remainder of this section fix $\sheaf{L}\in \Pic(A_{k^{\sep}})$ satisfying the properties of Lemma \ref{lemma: (1,2) surfaces unique symmetric line bundle and each divisor is symmetric}. 
Since there is a unique such line bundle up to isomorphism and $A$ has a $k$-point, $\sheaf{L}$ is automatically Galois invariant and descends to a unique line bundle on $A$ that we will continue to denote by $\sheaf{L}$.
\begin{remark}
{\em
    This marks an interesting contrast with principal polarizations on abelian surfaces, which may not be represented by a $k$-rational line bundle.
    Indeed, this is the case for a positive proportion of genus 2 Jacobians over $\Q$ ordered by height \cite[Theorem 23]{PoonenStoll-CasselsTate}. 
}
\end{remark}

The next proposition shows that every smooth member of the pencil $|\sheaf{L}|$ realizes $A$ as the dual Prym variety of a bielliptic genus $3$ curve, as in Proposition \ref{proposition: properties of the embedding of C into A}.
A \define{bielliptic genus $3$ curve} $(X, \tau)$ is a nice genus $3$ curve together with an involution $\tau$ on $X$ with $4$ fixed points.
For such a curve, let $\Prym(X,\tau) = \ker(1 + \tau^* \colon \Jac_X\rightarrow \Jac_X)$ be its Prym variety.
\begin{proposition}\label{proposition: smooth member pencil determines the surface}
    Suppose that $X\in |\sheaf{L}|(k)$ is smooth, and let $\tau$ be the restriction of $[-1]$ to $X$.
    Then $(X, \tau)$ is a bielliptic genus $3$ curve and the inclusion $X\rightarrow A$ induces an isomorphism $\Prym(X,\tau)^{\vee} \simeq A$ of $(1,2)$-polarized surfaces.
\end{proposition}
\begin{proof}
    See \cite[Proposition (1.8) and Theorem (1.12)]{Barth-abeliansurfacespolarization}.
\end{proof}

The abelian surface $A$ can also be reconstructed from the singular members of the pencil.
By Lemma \ref{lemma: properties pencil of (1,2)-polarized surface}(3), there are two cases to consider. 

For the first case, let $X \in |\sheaf{L}|(k)$ be a curve with one node, whose normalization $\tilde{X}\rightarrow X$ is a genus $2$ curve. Let $p\in X(k)$ be the node and $q_1,q_2$ be its preimages in $\tilde{X}(k^{\sep})$. 
The morphism $\tilde{X} \rightarrow X\hookrightarrow A$ induces a homomorphism of abelian varieties $\phi \colon \Jac_{\tilde{X}}\rightarrow A$, sending a degree zero divisor class on $\tilde{X}$ to its sum in $A$.
\begin{proposition}\label{proposition: nodal curve recovers (1,2)-polarized surface}
The homomorphism $\phi\colon \Jac_{\tilde{X}} \rightarrow A$ is an isogeny with kernel $\{0,q_1-q_2\}$.
\end{proposition}
\begin{proof}
The image of $\phi$ is an abelian subvariety containing $X$, so $\phi$ is surjective. 
By construction, $\phi(q_1 - q_2) = 0$.
It suffices to prove $\phi$ has degree $2$.
This follows from an intersection number calculation.
Given divisors $D, E$ on a nice surface, denote by $[D], [E]$ their class in the Picard group and their intersection number by $(D\cdot E)$.
The origin $0\in A(k)$ lifts to a unique point $\infty \in \tilde{X}(k)$ and we identify $\tilde{X}$ with its image in $\Jac_{\tilde{X}}$ using the point $\infty$.
Then $\phi^{*}([X]) = (\deg \phi)[\tilde{X}]$ and by the projection formula
\begin{align*}
    (\deg \phi)(X,X) =  (\phi^* X , \phi^* X)  = (\deg \phi)^2 (\tilde{X},\tilde{X}).
\end{align*}
The adjunction formula tells us that $(X,X) = 2\cdot 3-2 = 4$ and $(\tilde{X},\tilde{X}) = 2$.
Therefore $\deg \phi = 2$, as claimed. 
\end{proof}

For the second case, let $X = X_1\cup X_2 \in |\sheaf{L}|(k)$ be a pair of genus $1$ curves with intersection number $2$ and let $\tilde{X} = X_1 \sqcup X_2 \rightarrow X$ be its normalization.
The map $\tilde{X}\rightarrow X \hookrightarrow A$ induces a homomorphism $\phi\colon \Jac_{\tilde{X}} := \Jac_{X_1} \times \Jac_{X_2}\rightarrow A$.

\begin{proposition}\label{proposition: pair of genus 1 curves recovers (1,2)-polarized surface}
The map $\phi$ is an isogeny of degree $2$.
\end{proposition}
\begin{proof}
The proof is very similar to that of Proposition \ref{proposition: nodal curve recovers (1,2)-polarized surface}; we omit the details. 
\end{proof}

\subsection{Compatible \texorpdfstring{$\mu_3$}{mu3}-actions on \texorpdfstring{$(1,2)$}{(1,2)}-polarized abelian surfaces}\label{subsection: comptible mu3-actions (1,2)-polarized surfaces}

\begin{definition}\label{definition: mu3-abelian surface}
    A \define{$\mu_3$-abelian surface} over $k$ is a triple $(A,\lambda,\alpha)$, where $(A, \lambda)$ is a $(1,2)$-polarized abelian surface and $\alpha \colon \mu_3\rightarrow \Aut(A_{k^{\sep}})$ is a $\Gal_k$-equivariant group homomorphism such that for every primitive third root of unity $\omega \in k^{\sep}$:
    \begin{enumerate}
    \item $\alpha(\omega)^{\vee} \circ  \lambda \circ \alpha(\omega)=  \lambda$;
    \item $\alpha(\omega)$ has characteristic polynomial $X^2 + X +1$ on $T_0 A = \HH^0(A, \Omega^1_A)^{\vee}$.
    %Can be reformulated as follows: the induced $\mu_3$-action on $V = \HH^0(A, \Omega^1_A)$ has signature $(1,2)$, in other words $V = V_1 \oplus V_2$ where $V_i$ is one-dimensional and $\omega\cdot v_i  =\omega^iv_i$ for all $i\in \{1,2\}$ and $v_i \in V_i$. 
\end{enumerate}
\end{definition}
The signature condition (2) is equivalent to $\alpha(\omega)^2+\alpha(\omega) +1=0$, and is motivated by Lemma \ref{lemma: mu3-action on P has char poly T^2+T+1}.
Condition (1) can be reformulated as follows: after extending $\mu_3\hookrightarrow \Aut(A_{k^{\sep}})$ to a $\Q$-algebra homomorphism $\Q(\omega) \hookrightarrow \End(A_{k^{\sep}})\otimes \Q$, the Rosati involution associated to $\lambda$ restricts to the conjugation involution $a + b\omega\mapsto a+ b \bar{\omega}$ on $\Q(\omega)$.
\begin{remark}
    {\em
    It would be more accurate to call such triples `$(1,2)$-polarized abelian surfaces with compatible $\mu_3$-action of signature $(1,1)$', but we stick with $\mu_3$-abelian surfaces for brevity.
    }
\end{remark}

Let $(A, \lambda,\alpha)$ be a $\mu_3$-abelian surface over $k$.
Assume that $(A, \lambda)$ is indecomposable, so the results of \S\ref{subsection: generalities (1,2)-polarized surfaces} apply. 
Let $\sheaf{L}$ be the unique line bundle on $A$ representing $\lambda$ with base locus $A[\lambda]$.
The compatibility of $\alpha$ and $\lambda$ and the uniqueness of $\sheaf{L}$ show that $\alpha(\omega)^* \sheaf{L} \simeq \sheaf{L}$ for all $\omega \in \mu_3(k^{\sep})$.
Each such isomorphism is unique up to $k^{\times}$-scaling so determines a \emph{canonical} isomorphism $|\alpha(\omega)^* \sheaf{L}| \simeq |\sheaf{L}|$
Via pullback this induces a $\mu_3$-action on the pencil $|\sheaf{L}|$. 
\begin{proposition}\label{proposition: mu3-fixed point in pencil is smooth}
Every $\mu_3$-fixed point of $|\sheaf{L}|$ is a smooth curve. 
\end{proposition}
\begin{proof}
We may assume that $k = k^{\sep}$.
Let $\omega \in k$ be a third root of unity and for ease of notation denote $\alpha(\omega)\colon A\rightarrow A$ by $\alpha$.
Let $X\in |\sheaf{L}|$ be fixed by $\alpha$.
Assume for the sake of contradiction that $X$ is singular, with normalization $\tilde{X}\rightarrow X$. 
By Lemma \ref{lemma: properties pencil of (1,2)-polarized surface}(3) and Propositions \ref{proposition: nodal curve recovers (1,2)-polarized surface} and \ref{proposition: pair of genus 1 curves recovers (1,2)-polarized surface}, the map $\tilde{X}\rightarrow A$ induces an isogeny $\phi\colon \Jac_{\tilde{X}} \rightarrow A$ whose kernel $\{0,D\}$ has order $2$.
Since $\alpha$ preserves $X$, it induces an automorphism $\tilde{\alpha}$ of $\Jac_{\tilde{X}}$ and $\phi$ is equivariant with respect to $\tilde{\alpha}$ and $\alpha$.
It follows that $\tilde{\alpha}$ fixes $D$.
Since $\alpha^2+ \alpha + 1=0$ (Condition (2) in Definition \ref{definition: mu3-abelian surface}) and $\phi$ is an isogeny, $\tilde{\alpha}^2+\tilde{\alpha}+1=0$.
Therefore the degree of $1 -\tilde{\alpha}$ is a power of $3$, contradicting the fact that $\Z/2\Z \simeq \{0,D\}\subset \Jac_{\tilde{X}}[1-\tilde{\alpha}]$.
\end{proof}

\begin{corollary}
    Assume that $k = k^{\sep}$.
    Then there are exactly $2$ curves in $|\sheaf{L}|$ that are preserved by the $\mu_3$-action. 
\end{corollary}
\begin{proof}
    Let $\omega\in k$ be a third root of unity.
    By Proposition \ref{proposition: mu3-fixed point in pencil is smooth}, every fixed point of $\alpha(\omega)$ on $|\sheaf{L}|$ is smooth. 
    Since $|\sheaf{L}|$ always contains singular curves \cite[\S1.2]{Barth-abeliansurfacespolarization}, this implies that $\alpha(\omega)$ acts nontrivially on $|\sheaf{L}|$.
    By the classification of order $3$ elements of $\PGL_2(k)$, we may choose coordinates on $\P^1$ so that $\alpha(\omega)$ is given by $(x:y) \mapsto (\omega x:y)$.
    This map clearly has $2$ fixed points.
\end{proof}

Let $X_1,X_2\in |\sheaf{L}_{k^{\sep}}|$ be the members of the pencil that are preserved by the $\mu_3$-action $\alpha$.
By Lemma \ref{lemma: (1,2) surfaces unique symmetric line bundle and each divisor is symmetric}, the involution $[-1]$ also preserves $X_1$ and $X_2$.
These combine to a $\mu_6$-action $\gamma_i$ on $X_i$ for each $i=1,2$.
The next proposition shows that each $X_i$ is a bielliptic Picard curve and the $\mu_6$-actions have `opposite' signature.

\begin{proposition}\label{proposition: two mu3-fixed points are bielliptic picard with opposite signature}
    Reordering $X_1$ and $X_2$ if necessary, the pairs $(X_1,\gamma_1)$ and $(X_2,\gamma_2^{-1})$ are marked bielliptic Picard curves.
    Moreover, these pairs are defined over $k$.
\end{proposition}
\begin{proof}
    The second sentence follows from the first, since $\Gal_k$ preserves the $\mu_3$-fixed points of $|\sheaf{L}_{k^{\sep}}|$ and the signatures of the $\mu_6$-action.
    To show that $(X_1,\gamma_1)$ and $(X_2,\gamma_2)$ are bielliptic Picard curves of opposite signature, we may assume that $k = k^{\sep}$. 

    Let $(X,\gamma) = (X_i, \gamma_i)$ for some $i=1,2$. 
    Proposition \ref{proposition: mu3-fixed point in pencil is smooth} shows that $X$ is smooth, and Proposition \ref{proposition: smooth member pencil determines the surface} shows that $(X,[-1]|_X)$ is a bielliptic genus 3 curve. 
    Since $\alpha^2+\alpha+1=0$, the $\mu_3$-action on $A$ has isolated fixed points, so the induced $\mu_3$-action on $X$ is faithful.
    It follows that $\gamma$ is a faithful bielliptic action and hence $X$ is a bielliptic Picard curve. Note that $0\in X(k) \subset A(k)$ is the unique $\mu_6$-fixed point.
    
    To prove the claim about signatures, it suffices to show that the $\mu_3$-action on $T_0X_1$ and $T_0X_2$ acts via inverse characters.
    Since $\sheaf{L}$ has self-intersection $4$ and $|\sheaf{L}|$ has four base points (Lemma \ref{lemma: properties pencil of (1,2)-polarized surface}), $X_1$ and $X_2$ intersect transversally at all these base points.
    Since $0 \in A(k)$ is such a base point,
    \begin{align}\label{equation: tangent space surface direct sum}
        T_0 A = T_0 X_1 \oplus T_0X_2.
    \end{align}
    Condition (2) of Definition \ref{definition: mu3-abelian surface} says that the $\mu_3$-action on $T_0A$ is a direct sum of the two nontrivial characters of $\mu_3$.
    Since the decomposition \eqref{equation: tangent space surface direct sum} is $\mu_3$-stable, the proposition follows.
\end{proof}

\subsection{A Torelli theorem for marked bielliptic Picard curves}\label{subsec: torelli}

Let $(C, \gamma)$ be a marked bielliptic Picard curve over $k$. The $\mu_6$-action on $C$ induces a $\mu_6$-action on $J = \Jac_C$ such that the Abel--Jacobi map $C\hookrightarrow J$ is $\mu_6$-equivariant. 
This induces $\mu_3$-actions $\alpha$ on $P$ and $\hat{\alpha}$ on\footnote{A calculation shows that the $\mu_3$-actions $\hat{\alpha}$ (making $J\rightarrow A$ $\mu_3$-equivariant) and $\alpha^{\vee}$ (the dual of the $\mu_3$-action on $P$) on $A$ are inverse to each other.} $A$. 
\begin{lemma}\label{lemma: Prym and dual are mu3-abelian surfaces}
    $\Prym(C, \gamma) := (P, \lambda, \alpha)$ and $\Prym(C, \gamma)^{\vee} := (A, \widehat{\lambda}, \hat{\alpha})$ are indecomposable $\mu_3$-abelian surfaces over $k$.
\end{lemma}
\begin{proof}
    The first condition of Definition \ref{definition: mu3-abelian surface} for $\Prym(C,\gamma)^{\vee}$ follows from Proposition \ref{proposition: properties of the embedding of C into A} and the fact that $\hat{\alpha}$ preserves $C\hookrightarrow A$.
    Moreover, since $C$ is irreducible, $(A, \widehat{\lambda})$ is indecomposable.
    The second condition for $\Prym(C, \gamma)$ follows from Lemma \ref{lemma: mu3-action on P has char poly T^2+T+1}.
    The lemma now follows from bigonal duality (Proposition \ref{proposition: bigonal duality}).
\end{proof}

The next theorem is a strong Torelli statement for the association $(C, \gamma)\mapsto \Prym(C, \gamma)^{\vee}$.
To state it, we upgrade the set of marked bielliptic Picard curves (respectively $\mu_3$-abelian surfaces) to a category, where the morphisms are isomorphisms of curves (respectively polarized abelian surfaces) respecting the $\mu_6$-structure (respectively $\mu_3$-structure).

\begin{theorem}\label{theorem: Torelli theorem bielliptic Picard curves}
    The association $(C,\gamma) \mapsto \Prym(C, \gamma)^{\vee}$
    induces an equivalence of categories: 
    \begin{align}\label{equation: equivalence of categories torelli}
    \begin{Bmatrix}
        \text{Marked bielliptic} \\
        \text{Picard curves over }k
    \end{Bmatrix}
    \rightarrow 
    \begin{Bmatrix}
        \text{Indecomposable} \\
        \mu_3\text{-abelian surfaces over }k
    \end{Bmatrix}.
    \end{align}
\end{theorem}

\begin{proof}
    This association is well defined by Lemma \ref{lemma: Prym and dual are mu3-abelian surfaces}.
    It is functorial: an isomorphism of marked bielliptic Picard curves $(C,\gamma)\rightarrow (C',\gamma')$ induces an isomorphism $\Prym(C, \gamma)\rightarrow \Prym(C',\gamma')$ of marked $\mu_3$-abelian surfaces.
    We explicitly construct a quasi-inverse.
    To a triple $(A, \lambda, \alpha)$, in the notation of Proposition \ref{proposition: two mu3-fixed points are bielliptic picard with opposite signature}, we associate the marked bielliptic Picard curve $(X_1, \gamma_1)$.
    Since the constructions of \S\ref{subsection: generalities (1,2)-polarized surfaces} and \S\ref{subsection: comptible mu3-actions (1,2)-polarized surfaces} are canonical, this association is also functorial. 
    The fact that this is a quasi-inverse follows from Propositions \ref{proposition: properties of the embedding of C into A} and \ref{proposition: smooth member pencil determines the surface}.
\end{proof}
By Lemma \ref{lemma: when are two bielliptic Picard curves isomorphic}, the left hand side of \eqref{equation: equivalence of categories torelli} is in bijection with the set of equivalence classes $\{(a,b)\in k\times k\mid b(a^2-4b)\neq 0 \}/\sim$, where $(a,b)\sim (a',b')$ if $(a,b) = (\lambda^6a',\lambda^{12}b')$ for some $\lambda\in k^{\times}$.
If $k = k^{\sep}$, the $j$-invariant \eqref{equation: j-invariant} maps this set bijectively to the `$j$-line' $k \setminus \{0\}$.
Theorem \ref{theorem: Torelli theorem bielliptic Picard curves} thus shows that this description also applies to the right hand side of \eqref{equation: equivalence of categories torelli}.
In particular, we may define the $j$-invariant of a $\mu_3$-abelian surface $(A, \lambda, \alpha)$ as $j((C, \gamma))$, where $(C, \gamma)$ is any bielliptic Picard curve with $\Prym(C, \gamma)^{\vee} = (A, \lambda, \alpha)$.

\begin{remark}
{\em
    Since $\Prym(C,\gamma)\simeq \Prym(\widehat{C}, \hat{\gamma})^{\vee}$ (Proposition \ref{proposition: bigonal duality}) and bigonal duality induces an autoequivalence on the category of marked bielliptic Picard curves, Theorem \ref{theorem: Torelli theorem bielliptic Picard curves} implies that the association $(C, \gamma) \mapsto \Prym(C, \gamma)$ also induces an equivalence of categories of the form \eqref{equation: equivalence of categories torelli}. 
    On the other hand, the association $(A, \lambda, \alpha)\mapsto (A, \lambda, \alpha^{-1})$ is again an auto-equivalence on the category of marked $\mu_3$-abelian surfaces. 
    In Corollary \ref{corollary: marked mu3-surface and dual are isomorphic up to twist} we will prove a precise relationship between these equivalences.
    }
\end{remark}

\begin{remark}
{\em
    Barth has shown that a Torelli-type result fails for the family of all bielliptic genus $3$ curves $(X, \tau)$.
    Suppose that $k$ is algebraically closed of characteristic zero and $A$ is an indecomposable $(1,2)$-polarized abelian surface $(A,\lambda)/k$.
    Then there exist many bielliptic genus $3$ curves $(X,\tau)$ over $k$ with dual Prym variety isomorphic to $(A, \lambda)$; by Proposition \ref{proposition: smooth member pencil determines the surface} any smooth member of the pencil $|\sheaf{L}|$ on $A$ will do. 
    Theorem \ref{theorem: Torelli theorem bielliptic Picard curves} shows that we can correct this when $C$ is a bielliptic Picard curve and we fix the $\mu_6$-action on $C$ and the $\mu_3$-action on $A$.
}
\end{remark}

\begin{remark}\label{remark: isomorphism of moduli stacks}
{\em
    Let $S$ be a $\Z[1/6]$-scheme.
    A marked bielliptic Picard curve over $S$ is a smooth proper morphism $C\rightarrow S$ together with a $\mu_6$-action $\gamma\colon \mu_6 \times_S C \rightarrow C$ such that for every geometric point $\bar{s} \in S$, $(C_{\bar{s}}, \gamma_{\bar{s}})$ is a marked bielliptic Picard curve.
    With the evident notion of isomorphisms, this defines the stack $\mathfrak{M}$ of (marked) bielliptic Picard curves over $\Z[1/6]$ (in the fppf topology), isomorphic to an open substack of the weighted projective stack $\P(6,12)$.
    A $\mu_3$-abelian surface over $S$ is a $(1,2)$-polarized abelian scheme $(A, \lambda)$ over $S$ 
    together with a $\mu_3$-action $\alpha \colon \mu_3\times_S  A\rightarrow A$ such that for every geometric point $\bar{s}\in S$, $(A_{\bar{s}}, \lambda_{\bar{s}}, \alpha_{\bar{s}})$ is a $\mu_3$-abelian surface. 
    This defines the stack $\mathfrak{Y}$ of  $\mu_3$-abelian surfaces over $\Z[1/6]$.
    There exists an open substack $\mathfrak{Y}^{\text{indec}}\subset \mathfrak{Y}$ of those surfaces such that each geometric fiber is indecomposable.
    Theorem \ref{theorem: Torelli theorem bielliptic Picard curves} then generalizes to the statement that the association $(C, \gamma)\mapsto \Prym(C, \gamma)^{\vee}$ induces an isomorphism of stacks $\mathfrak{M}\xrightarrow{\sim} \mathfrak{Y}^{\text{indec}}$.
}
\end{remark}

\section{Shimura curves and quaternionic multiplication}\label{connections to Shimura curves}

We relate the moduli space of bielliptic Picard curves to certain unitary and quaternionic Shimura curves, which will allow us to deduce that these bielliptic Picard Pryms surfaces $P$ have quaternionic multiplication over $\overline{k}$ (Corollary \ref{cor:QM existence}).
This will also help us study which $P$ are geometrically non-simple in \S\ref{subsec: split Prym varieties}, since these correspond to CM points on the Shimura curve side.

\subsection{The quaternion order \texorpdfstring{$\calO$}{O}}\label{subsec: the quaternion order O}

Let $B$ be the quaternion algebra $(-3,2)_{\Q}$ with $\Q$-basis $\{1,i,j,ij\}$ and relations $i^2 = -3$, $j^2 = 2$, $ij = -ji$.
This is an indefinite quaternion algebra of discriminant $6$.
Let $\omega := \frac12(-1+i)$, so that $\omega^2 + \omega + 1 = 0$.
Let $\mathcal{O} := \Z[\omega,j] = \Z  + \Z\omega + \Z j  + \Z \omega j$.
A discriminant calculation shows that $\mathcal{O}$ is a maximal order in $B$.
As an associative ring, $\calO$ has generators $\omega, j$ and relations $\omega^2+ \omega + 1=0$, $j^2=2$ and $ j \omega  =  \omega^{-1} j$.

Let $\Aut(\mathcal{O})$ denote the set of ring automorphisms of $\mathcal{O}$, which we will think of as acting on the right on $\mathcal{O}$.
By the Skolem--Noether theorem, every element of $\Aut(\mathcal{O})$ is given by conjugating by an element of $b\in B^{\times}$ (uniquely defined up to $\Q^{\times}$) normalizing $\mathcal{O}$; denote the conjugation action of such an element $b$ by $[b] \in \Aut(\calO)$.

For reasons that will become clear in \S\ref{subsec: endomorphism field}, we define the groups
\begin{align*}
    \Aut_{i}^+(\calO) &:= \{\varphi \in \Aut(\calO) \mid  i^{\varphi} = i \} ,\\
    \Aut_{i}(\calO) &:= \{\varphi \in \Aut(\calO) \mid  i^{\varphi} = \pm i \}.
\end{align*}

\begin{lemma}\label{lemma: explicit description of AutiO}
    $\Aut_i^+(\calO)= \langle [1- \omega]\rangle$ is cyclic of order $6$ and $\Aut_i(\calO)= \langle [1- \omega] , [j] \rangle$ is dihedral of order $12$.
\end{lemma}
\begin{proof}
    A direct calculation shows that $\langle [1- \omega] , [j] \rangle \subset \Aut_{i}(\calO)$, that this subgroup is dihedral of order $12$ and that $\langle [1- \omega] , [j] \rangle \cap \Aut_i^+(\calO) = \langle [1-\omega]\rangle$.
    Conversely, suppose $[b] \in \Aut_i(\calO)$.
    After possibly multplying by $j$, we may assume $i^{[b]} = b^{-1} i b = i$. 
    In that case $b$ centralizes $i$, so $b\in \Q(i)$. 
    By \cite[Lemma 2.2.1]{LSSV-QMMazur}, $[b]$ can be represented by an element of $\calO \cap \Q(i) = \Z[\omega]$ of norm dividing $6$. 
    Therefore $[b] \in \langle [1-\omega]\rangle$, as desired.
\end{proof}

\begin{lemma}\label{lem:subfields}
    Let $G\leq \Aut_i(\calO)$ be a subgroup.
    Then the subring $\calO^G$ fixed by $G$ is: $\calO$ if $G=\{1\}$; $\Z[\omega]$ if $\{1\} \neq G \subset \langle [1-\omega]\rangle$; isomorphic to $\Z[\sqrt{2}]$ if $G$ is conjugate to $\langle [j]\rangle$; isomorphic to $\Z[\sqrt{6}]$ if $G$ is conjugate to $\langle [ij]\rangle $.
    In all other cases, $G$ is dihedral of order $6$ or $12$ and $\calO^G =\Z$.
\end{lemma}
\begin{proof}
    This is a direct calculation using the explicit description of $\Aut_i(\calO)$ from Lemma \ref{lemma: explicit description of AutiO}.
\end{proof}

\subsection{Moduli of marked \texorpdfstring{$\mu_3$}{mu3}-abelian surfaces}\label{subsec: moduli of bielliptic Picard curves as shimura curve}

In Remark \ref{remark: isomorphism of moduli stacks}, we have introduced the stack $\mathfrak{Y}$ of $\mu_3$-abelian surfaces $(A, \lambda, \alpha)$ over $\Z[1/6]$, which we now view as a stack over $\Q$.
This is a Deligne--Mumford stack; let $Y/\Q$ be its coarse space.
By a variant of \cite[Proposition 2.1]{KudlaRapoport-unitarySVs} for $(1,2)$-polarizations, $Y$ is smooth and purely one-dimensional.
The next lemma shows that $Y$ may be viewed as a compactification of the $j$-line $\P^1\setminus \{0,\infty\}$ from \S\ref{subsection: automorphisms}.

\begin{lemma}\label{lemma: Y isomorphic to P1}
    There exists a unique isomorphism $j\colon Y\xrightarrow{\sim} \P^1_{\Q}$ with the property that $j([\Prym(C, \gamma)^{\vee}]) = j(C)$ for all marked bielliptic Picard curves $(C, \gamma)$ over $\bar{\Q}$.
\end{lemma}
\begin{proof}
    The open substack $\mathfrak{Y}^{\text{indec}}\subset \mathfrak{Y}$ parametrizing those triples $(A, \lambda, \alpha)$ for which $(A, \lambda)$ is indecomposable has coarse space an open subset $Y^{\text{indec}} \subset Y$.
    We first analyze the complement.
    Every decomposable $(1,2)$-polarized abelian surface $(A, \lambda)$ over $\mathbb{C}$ must be of the form $(E_1, \lambda_O)\times (E_2, 2\lambda_O)$, where $E_i$ are elliptic curves and $\lambda_O$ is the unique principal polarization on an elliptic curve.
    If $\alpha$ is a $\mu_3$-action on $A$ satisfying the conditions of Definition \ref{definition: mu3-abelian surface}, then $\alpha$ restricts to a nontrivial action on $E_i$ for each $i$. It follows that $E_i$ must be isomorphic to $E \colon y^2 = x^3+ 1$. 
    We conclude that every decomposable $\mu_3$-abelian surface $(A, \lambda, \alpha)$ over $\C$ is isomorphic to $(E\times E, \lambda_O \times 2\lambda_O, \alpha_1\times \alpha_2)$, where $\alpha_i$ are mutually inverse nontrivial $\mu_3$-actions on $E$.
    Since there are two choices for the pair $(\alpha_1,\alpha_2)$, even up to isomorphism, $(Y\setminus Y^{\text{indec}})(\C)$ has size $2$.
    Since $Y$ is purely one-dimensional, $Y^{\text{indec}}$ is dense in $Y$.
    Theorem \ref{theorem: Torelli theorem bielliptic Picard curves} and Lemma \ref{lemma: when are two bielliptic Picard curves isomorphic} show that the morphism $\A^2_{\Q} \setminus \{\Delta_{a,b}=0\} \rightarrow Y^{\text{indec}}, (a,b)\mapsto [\Prym(C_{a,b}, \gamma_{a,b})^{\vee}]$ is surjective, factors through the $j$-invariant morphism $\A^2_{\Q} \setminus \{\Delta_{a,b}=0\} \rightarrow \P^1_{\Q}\setminus\{0,\infty\}, (a,b)\mapsto j_{a,b}$ of \eqref{equation: j-invariant}, and induces an isomorphism $\P^1_{\Q} \setminus \{0,\infty\} \simeq Y^{\text{indec}}, j \mapsto [\Prym(C_{a,b}, \gamma_{a,b})^{\vee}]$, where $a,b$ are any elements with $j = j_{a,b}$.
    Putting everything together, $Y$ is a purely one-dimensional smooth variety having an open subset isomorphic to $\P^1\setminus \{0,\infty\}$ whose complement has size $2$.
    We conclude that $Y\simeq \P^1_{\Q}$ and that the isomorphism $j\colon Y^{\text{indec}}\rightarrow \P_{\Q}^1\setminus\{0,\infty\}$ uniquely extends to an isomorphism $Y\rightarrow \P_{\Q}^1$.
\end{proof}

\begin{remark}
    {\em
    The curve $Y_{\C}$ is a disjoint union of Shimura varieties for groups of type $\mathrm{GU}(1,1)$.
    Similarly to \cite[Proposition 3.1]{KudlaRapoport-unitarySVs},  the connected components of $Y_{\C}$ are parametrized by isometry classes of rank $2$ Hermitian $\Z[\omega]$-lattices of discriminant $2$.
    Lemma \ref{lemma: Y isomorphic to P1} shows that there is precisely one such Hermitian lattice, although this is certainly not the easiest way to see that.
    }
\end{remark}

\begin{remark}
    {\em
    The isomorphism $j\colon Y\rightarrow \P^1$ identifies $Y^{\text{indec}}$ with $\P^1\setminus \{0,\infty\}$.
    The two decomposable points $0,\infty$ might be considered `cusps'.
    }
\end{remark}

\subsection{Comparison with a quaternionic Shimura curve}\label{subsec: comparison with quaternionic curve}
We now compare $Y$ to a quaternionic Shimura curve, in the same spirit as \cite[\S4.2]{Howard-unitaryshimuracurves}.
However, our analysis is a little different since we must consider twists and non-principal polarizations. 
We refer the reader to \cite[Chapter 43]{Voight-Quaternionalgebras} for background on Shimura curves.
\begin{definition}
Let $X$ be the (quaternionic) Shimura curve of discriminant $6$ over $\Q$.    
\end{definition}
The nice curve $X$ is the coarse space of the moduli stack of triples $(A, \lambda, \iota)$, where $(A, \lambda)$ is a $(1,2)$-polarized abelian surface and $\iota\colon \calO\rightarrow \End(A)$ is an embedding such that the Rosati involution restricts to $b\mapsto i \bar{b} i^{-1}$, where we recall from \S\ref{subsec: the quaternion order O} that $i=1+2\omega$.
(The moduli problem is typically formulated for principal polarizations, but the results of \cite[\S43.6]{Voight-Quaternionalgebras} have analogues for the polarized order $(\calO,i)$, see \cite[\S12]{Boutot79}.)
To such a triple $(A, \lambda, \iota)$ we would like to associate a $\mu_3$-abelian surface by restricting $\iota$ to $\Z/3\Z\simeq\langle \omega\rangle \subset \Z[\omega]^{\times} \subset \calO^{\times}$.
However, this defines a $\Z/3\Z$-action on $A$; to get a $\mu_3$-action we need to suitably twist the situation.

To this end, recall the action of the Atkin--Lehner group $W = \{1,w_2,w_3,w_6\}\simeq \Z/2\Z\times \Z/2\Z$ on $X$: if $b\in \calO$ is any element whose (reduced) norm has absolute value $d\in \{2,3,6\}$, then $w_d([A, \lambda, \iota]) = [A, \lambda, \iota\circ [b]]$, where $[b] \in \Aut(\calO)$ denotes conjugation by $b$.
The following curves will be important:
\begin{itemize}
    \item Let $X/w_3$ be the quotient of $X$ by $w_3$. The $W$-action on $X$ induces an action of $W/\langle w_3\rangle = \{1,\bar{w}\}$ on $X/w_3$.
    \item Let $(X/w_3)^{\bar{w}}$ be the twist of $X/w_3$ along the involution $\bar{w}$ and extension $\Q(\omega)/\Q$. 
    \item For each $d\in \{2,6\}$, let $X^{w_d}$ be the twist of $X$ along the involution $w_d$ and extension $\Q(\omega)/\Q$.
\end{itemize}
We recall that the twist of a nice curve $C/k$ along an involution $\tau$ and quadratic extension $K/k$ corresponds to the cocycle given by the image of the nontrivial element under $\HH^1(\Gal(K/k), \langle \tau\rangle) \rightarrow \HH^1(k, \mathbf{Aut}(C))$.
Note that $W$ acts on $X^{w_2}$ and $X^{w_6}$, $\bar{w}$ acts on $(X/w_3)^{\bar{w}}$, and we have canonical isomorphisms $X^{w_2}/w_3 \simeq X^{w_6}/w_3 \simeq (X/w_3)^{\bar{w}}$.

To compare $X$ and $Y$, we will construct morphisms $\pi_d\colon X^{w_d}\rightarrow Y$ for each $d\in \{2,6\}$. 
In the following paragraph, for a field $k/\Q$ we write $K := k\otimes_{\Q} \Q(\omega)$ and $\Aut_k(K) = \{1,\sigma\}$.
The curve $X^{w_2}$ is the coarse space of the moduli stack whose $k$-points parametrizes triples $(A, \lambda, \iota)$, where $(A, \lambda)$ is a $(1,2)$-polarized abelian surface over $k$ and $\iota\colon \calO\rightarrow \End(A_{K})$ is an embedding such that the Rosati involution restricts to $b\mapsto i\bar{b} i^{-1}$ and such that $\iota^{\sigma} = \iota \circ [j]$.
We claim that such a triple naturally gives rise to a $\mu_3$-abelian surface $(A, \lambda, \iota|_{\langle \omega\rangle})$ over $k$.
Indeed, restricting $\iota$ to $\langle \omega \rangle \subset \Z[\omega]^{\times} \subset \calO^{\times}$ defines homomorphism $\Z/3\Z\rightarrow \End(A_K)$ which due to the twisting descends to a $\mu_3$-action on $A$.
The triple $(A, \lambda, \iota|_{\langle \omega\rangle})$ satisfies first condition of Definition \ref{definition: mu3-abelian surface} since the Rosati involution sends $\iota(\omega)$ to $\iota(\omega^{-1})$.
The second condition follows from the identity $\omega^2+\omega+1=0$.
The association $(A, \lambda, \iota)\mapsto (A, \lambda, \iota|_{\langle \omega\rangle})$ also works on the level of $\Q$-schemes, hence it induces a morphism $\pi_2\colon X^{w_2}\rightarrow Y$.
Repeating the above discussion but insisting that $\iota^{\sigma} = \iota \circ [ij]$ instead of $\iota^{\sigma}= \iota \circ [j]$, we obtain a morphism $\pi_6\colon X^{w_6}\rightarrow Y$.

\begin{proposition}\label{proposition: Y is quotient of shimura curve twist}
    For each $d \in \{2,6\}$, $\pi_d\colon X^{w_d}\rightarrow Y$ induces an isomorphism $\phi_d\colon X^{w_d}/w_3\xrightarrow{\sim} Y$.
    Under the canonical isomorphisms $X^{w_2}/w_3\simeq X^{w_6}/w_3\simeq (X/w_3)^{\bar{w}}$, $\phi_2$ and $\phi_6$ correspond to the same isomorphism $(X/w_3)^{\bar{w}}\simeq Y$.
\end{proposition}
\begin{proof}
    It suffices to check these claims over $\C$, so let $\pi = \pi_{2,\C} = \pi_{6,\C}$. 
    The morphism $\pi$ is given by mapping an isomorphism class of triples $[A, \lambda, \iota]$ to $[A, \lambda, \iota|_{\langle \omega\rangle }]$.
    We have $\pi\circ w_3 = \pi$, since $w_3([A, \lambda, \iota]) = [A, \lambda, \iota\circ [i]]$ and $\iota|_{\langle\omega\rangle}=(\iota\circ [i])|_{\langle\omega\rangle}$.
    (Recall that $i = 1+2\omega$.)

    Since $\pi$ is a morphism between nice curves, it is either constant or a finite covering map.
    To prove the proposition, it suffices to show that $\pi$ has degree $2$, equivalently infinitely many fibres $\pi^{-1}(y)$ have size $2$.
    Let $x = [A, \lambda, \iota] \in X(\C)$ be a non-CM point and $y = \pi(x) = [A,\lambda, \iota|_{\langle \omega\rangle}]$; there are infinitely many (in fact uncountably many) such points $x$ by the theory of Shimura curves \cite[\S2.4]{Elkies-Shimuracurvecomputations}.
    The set $\pi^{-1}(y)$ is in bijection with the set of isomorphism classes of triples $(A',\lambda',\iota')$ such that $(A, \lambda, \iota|_{\langle \omega\rangle}) \simeq (A', \lambda', \iota'|_{\langle \omega\rangle})$.
    Replacing $(A',\lambda',\iota')$ by an isomorphic triple, we may assume that $(A',\lambda') = (A, \lambda)$ and $\iota'|_{\Z[\omega]} = \iota|_{\Z[\omega]}$.
    Therefore $\pi^{-1}(y)$ is in bijection with the orbit space 
    \begin{align}\label{equation: fibre of map pi}
        \{ \iota'\colon \calO \rightarrow \End(A) \mid \iota' \text{ compatible with Rosati involution and }\iota'|_{\Z[\omega]} = \iota|_{\Z[\omega]}\}/\Aut(A, \lambda).
    \end{align}
    Since $x$ is assumed to be a non-CM point and $\calO$ is maximal, $\End(A) \simeq \calO$.
    It follows that every $\iota'$ of \eqref{equation: fibre of map pi} is of the form $\iota\circ \alpha$ for some $\alpha\in \Aut(\calO)$.
    Since $\iota'|_{\Z[\omega]} = \iota|_{\Z[\omega]}$, $\alpha$ lies in the subgroup $\Aut_i^{+}(\calO) = \langle [1-\omega]\rangle$ of Lemma \ref{lemma: explicit description of AutiO} (using the notation of \S\ref{subsec: the quaternion order O}).
    Conversely, every element $[b]\in \langle [1-\omega]\rangle$ defines an element of \eqref{equation: fibre of map pi}, and two elements $[b], [b']$ lie in the same $\Aut(A, \lambda)$-orbit if and only if $b'b^{-1} \in \Q^{\times} \calO^{\times}$.
    It follows that this orbit space has size $|\pi^{-1}(y)|=|\langle [1-\omega]\rangle/\langle [\omega] \rangle|=2$, as desired.
\end{proof}

\begin{corollary}\label{cor:QM existence}
Every $\mu_3$-abelian surface $(A,\lambda,\alpha)$ over a field $k$ of characteristic $0$ admits an embedding $\calO \hookrightarrow \End(A_{k^{\sep}})$.
\end{corollary}
\begin{proof}
    Proposition \ref{proposition: Y is quotient of shimura curve twist} shows that the point $[A, \lambda, \alpha] \in Y(k)$ lifts to a $k^{\sep}$-point of $X^{w_2}$ under $\pi_2$.
    The corollary follows from the moduli interpretation of $X$.
\end{proof}

\begin{remark}
    {\em 
Corollary \ref{cor:QM existence} was proved  by Petkova and Shiga over $\C$ using period matrix calculations \cite[Proposition 7.1]{PetkovaShiga}. 
}
\end{remark}

Our proof of Corollary \ref{cor:QM existence} shows that there exists an action of $\calO$ on $P_{k^{\sep}}$, but it sheds no light on what this action actually is. 
In Proposition \ref{proposition: Prym has O-PQM} we will give a more geometric proof of the quaternionic multiplication of $P_{k^{\sep}}$, which works in all characteristics $\neq 2,3$  and gives finer arithmetic information, such as the field of definition of the quaternionic multiplication, see \S\ref{subsec: endomorphism field}.

\section{Identifying bielliptic Picard curves in the pencil}\label{sec: identifying bielliptic pic curves in the pencil}

Let $(C, \gamma)$ be a marked bielliptic Picard curve.
Its Prym variety $\Prym(C, \gamma)=(P, \lambda, \alpha)$ is an indecomposable  $\mu_3$-abelian surface (Lemma \ref{lemma: Prym and dual are mu3-abelian surfaces}).
Therefore the results of \S\ref{subsection: generalities (1,2)-polarized surfaces}-\ref{subsection: comptible mu3-actions (1,2)-polarized surfaces} apply, and the pencil of curves on $P$ contains exactly two bielliptic Picard curves by Proposition \ref{proposition: two mu3-fixed points are bielliptic picard with opposite signature}.
We have already seen in \S\ref{subsection: bigonal duality} that this pencil contains the bigonal dual curve $\widehat{C}$.
In this section we will determine the other bielliptic Picard curve in the pencil: it turns out to be a sextic twist of the original curve $C$!
This calculation is a crucial ingredient for the results of \S\ref{endomorphisms of the Prym variety}.

Let $\widehat{C}\subset P$ be the dual curve of $C$ from \S\ref{subsection: bigonal duality}.
Then $\sheaf{M} :=\calO_P(\widehat{C})$ is the unique line bundle on $P$ representing the $(1,2)$-polarization with base locus $P[\lambda]$ (Proposition \ref{lemma: (1,2) surfaces unique symmetric line bundle and each divisor is symmetric}). 
By Lemma \ref{lemma: bigonal dual is a bielliptic Picard curve}, $\widehat{C}$ is a bielliptic Picard curve preserved by the $\mu_3$-action $\alpha$.
By Proposition \ref{proposition: two mu3-fixed points are bielliptic picard with opposite signature}, there is exactly one other curve in $|\sheaf{M}|$ that is preserved by the $\mu_3$-action; call this curve $C_{\text{twist}}$.
We will explicitly describe $C_{\text{twist}}$ using the work of Ikeda \cite{Ikeda-Biellipticcurvesofgenusthree}.

We will use the notations from \S\ref{subsection: the prym variety P}.
If $p\in E(k)$, let $\Theta_p$ be the image of the map $\Sym^2 C\rightarrow J$ sending $Q + Q'$ to $Q + Q' - \pi^*(p)$.
Then $D_p := \Theta_p \cap P$ is a divisor on $P$.  Note that $D_{\pi(\infty)} = \widehat{C}$.
Recall that we view $E$ as an elliptic curve with origin $O_E := \pi(\infty) \in E(k)$. 
\begin{lemma}
    \begin{enumerate}
        \item For every $p\in E(k)$, $D_p\in |\sheaf{M}|(k)$.
        \item If $k = k^{\sep}$, every element of $|\sheaf{M}|$ is of the form $D_p$ for some $p \in E(k)$.
        \item For every $p\in E(k)$, $D_{\iota(p)} = D_p$, where $\iota(p)$ denotes the inverse of $p$ under the group law. 
    \end{enumerate}
\end{lemma}
\begin{proof}
    This is \cite[Lemma 3.4]{Ikeda-Biellipticcurvesofgenusthree} and \cite[Remark 3.5]{Ikeda-Biellipticcurvesofgenusthree}.
\end{proof}

Therefore the map $E\rightarrow  |\sheaf{M}|, p \mapsto D_p$ factors through the map $\phi\colon E\rightarrow |2O_E|$ and gives rise to a  $\mu_3$-equivariant map $|2O_E|\rightarrow |\sheaf{M}|, y \mapsto D_y$.
We note that if $y$ is $k$-rational, then so is $D_y$.

By Lemma \ref{lemma: when are two bielliptic Picard curves isomorphic}, we may and do assume that $(C, \gamma)= (C_{a,b}, \gamma_{a,b})$ for some $a,b\in k$.
Then $E_{a,b}$ has equation $y^3 = x^2 +a x +b$ and $y$ induces an isomorphism $|2O_E| \simeq \P^1$.
There are two fixed points for the $\mu_3$-action on $|2O_E|$, corresponding to $0, \infty \in \P^1(k)$; call them $y_0, y_{\infty} \in |2O_E|$.
Then $D_{y_{\infty}} =D_{O_E} = \widehat{C}$ and $D_{y_0} = C_{\text{twist}}$.

In order to compute equations for $D_{y_0}$, we will give a different description of the family of divisors $\{D_p\}_{p \in E}$.
For ease of notation, write $C^{(2)} = \Sym^2 C$.
For each $\xi \in \Pic^2(E)$, consider the closed subscheme 
\begin{align*}
    F_{\xi} := \{ Q + Q' \in C^{(2)}\mid \pi(Q) + \pi(Q') \sim \xi \} \subset C^{(2)}.
\end{align*}
Since $C$ is not hyperelliptic, the map $Q + Q' -\pi^*(p) \mapsto Q + Q'$ induces an isomorphism $\phi_p\colon D_p \xrightarrow{\sim} F_{2p}$.
Consider the morphism
\begin{align*}
    \psi \colon C^{(2)}&\longrightarrow \P^2  \\
    ([x:y:z],[x':y':z']) &\mapsto [yz'-y'z:xz'-x'z:xy'-x'y],
\end{align*}
which sends a pair of distinct points on $C$ to the line it spans in the canonical embedding $C\subset \P^2$.
Ikeda has shown that for every $\xi \in \Pic^2(E)$ such that $F_{\xi}$ is smooth, $\psi$ maps $F_{\xi}$ isomorphically onto a plane quartic in $\P^2$ and determined this quartic explicitly \cite[Lemma 3.12]{Ikeda-Biellipticcurvesofgenusthree}. (In fact, he has shown this for any bielliptic genus $3$ curve.)
We will use his calculation to determine $C_{\text{twist}}$, together with a descent argument along a quadratic extension.

Let $\alpha, \alpha'$ be the roots of $x^2 + ax +b$.
Let $p= (\alpha,0)$ and $p' = (\alpha',0)$.
Let $K = k(\alpha, \alpha') =  k(\sqrt{a^2-4b})$.
Then $p,p'\in E(K)$ and $C_{\text{twist}} = D_{y_0} = D_p = D_{p'}$.
Both $p$ and $p'$ give rise to $K$-isomorphisms $\phi_p \colon D_p \rightarrow F_{2p}$ and $\phi_{p'}\colon D_{p'} \rightarrow F_{2p'}$.
%We will determine the plane quartics $\psi(F_{2p})$ and $\psi(F_{2p'})$.
Let $[u:v:1]$ be affine coordinates for the target of the morphism $\psi\colon \Sym^2 C\rightarrow \P^2$.

\begin{lemma}\label{lemma: calculation plane quartic}
   $\psi(F_{2p}) = \psi(F_{2p'}) \subset \P^2_K$ is cut out by the projective closure of the equation
   \begin{align}
    (2a^2-8b)v^3 = bu^4+au^2 +1. 
   \end{align}
\end{lemma}
\begin{proof}
    This is a special case \cite[Lemma 3.12]{Ikeda-Biellipticcurvesofgenusthree}\footnote{The results of loc. cit. assume that $k = \mathbb{C}$, but the proof of the lemma works over any field of characteristic $\neq 2$.}.
    For the reader's convenience, we provide the translation to Ikeda's notation.
    In his notation: $C$ has equation $(z^2 + \frac{a}{2}y^2)^2 = x^3y + \left(\frac14a^2-b\right)y^4$, $S(x,y) = \frac14a^2y^2$, $T(x,y) = x^3y + (\frac14a^2-b)y^4$, and $p = [0:\alpha : 1]$.
    %So $s_0 =s_1 = t_1 = t_2 = 0$.
    %Also our $u$ and $v$ are not the same as his.
    %\jef{I have checked this calculation in Mathematica}
\end{proof}
Let $\mathcal{Q} = \psi(F_{2p}) \subset \P^2_K$. Lemma \ref{lemma: calculation plane quartic} shows that $\mathcal{Q} = \psi(F_{2p'})$ and that this quartic is defined over $k$.
The substitution $(u,v) \mapsto ((2a^2-b)^3b u , (2a^2-b)^2b v)$ shows that $\mathcal{Q}$ is $k$-isomorphic to the curve $y^3 = x^4+dax^2 + d^2 b$, where $d = 16b(a^2-4b)^4$. 
Therefore $\mathcal{Q}$ and $C$ are sextic twists in the sense of \S\ref{subsection: twists} and thus isomorphic over $k^{\sep}$.
However, these curves are not necessarily isomorphic over $k$, since the isomorphisms $C_{\text{twist},K} = D_{p} \xrightarrow{\phi_p} F_{2p}$ and $F_{2p} \xrightarrow{\psi} \mathcal{Q}_K$ are only defined over the (at most quadratic) extension $K/k$.
We descend down to $k$ by determining the Galois action on these isomorphisms; see \S\ref{subsec: notation and conventions} for our conventions on how $\Gal_k$ acts on morphisms.

\begin{lemma}\label{lemma: cocycle iso twist dual curve and Ikeda quartic}
    Suppose that $K \neq k$ and let $\sigma \in \Gal(K/k)$ be the unique nontrivial element.
    Let $\chi = \psi|_{F_{2p}} \circ \phi_p \colon C_{\text{twist},K} \xrightarrow{\sim} \mathcal{Q}_K$.
    Then $\chi^{\sigma} \circ \chi^{-1}\colon \mathcal{Q}_K \rightarrow \mathcal{Q}_K$ equals the involution $(u,v) \mapsto (-u,v)$.
\end{lemma}
\begin{proof}
    Let $(u,v) \in \mathcal{Q}_K$ and let $Q_1+Q_2\in F_{2p}$ be the unique element mapping to $(u,v)$ under $\psi|_{F_{2p}}$.
    Since $p^{\sigma} = p'$, a calculation shows that $\chi^{\sigma}\circ \chi^{-1}$ maps $(u,v)$ to $\psi(R_1+R_2)$, where $R_1+R_2$ is the unique effective divisor on $C$ such that $R_1+R_2 -\pi^*(p') \sim Q_1+Q_2-\pi^*(p)$.
    We interpret this geometrically: let $\ell= \psi(Q_1+Q_2)$ be the line spanned by the points $Q_1, Q_2$ in the embedding $C = C_{a,b} \subset \P^2$.
    The condition $Q_1+Q_2\in F_{2p}$ implies that $Q_1+Q_2 + \tau(Q_1) +\tau(Q_2)\sim 2\pi^*(p)$, where we recall that $\tau\colon C\rightarrow C$ denotes the bielliptic involution.
    Combining the above two linear equivalences shows that $\tau(Q_1)+\tau(Q_2)+R_1+R_2 \sim \pi^*(p)+\pi^*(p')\sim 4\infty$. 
    Since $4\infty$ is canonical, it follows that $\tau(Q_1),\tau(Q_2),R_1,R_2$ are collinear and $\psi(R_1+R_2) = \psi(\tau(Q_1)+\tau(Q_2)) = \tau(\ell)$.
    In the coordinates $(u,v)$, this shows that $\psi(R_1+R_2) = (-u,v)$, proving the lemma.
\end{proof}

\begin{theorem}\label{theorem: Ctwist is twist of original curve}
    Let $C$ be a bielliptic Picard curve of the form $C_{a,b}$ for some $a,b\in k$.
    Then $C_{\text{twist}}$ is again a bielliptic Picard curve, isomorphic to $C_{\delta a, \delta^2 b}$ where $\delta = \Delta_{a,b} = 16b(a^2 - 4b)$.
    In particular, $C$ and $C_{\text{twist}}$ are isomorphic over $k^{\sep}$.
\end{theorem}
\begin{proof}
    By Lemma \ref{lemma: calculation plane quartic}, $\mathcal{Q}\simeq C_{da,d^2b}$, where $d = 16b(a^2-4b)^4$.
    Suppose first that $K = k$. 
    Then the isomorphism $\psi|_{F_{2p}} \circ \phi_p\colon C_{\text{twist}}\rightarrow  \mathcal{Q}$ are defined over $k$, hence $C_{\text{twist}} \simeq C_{da,d^2b}$ too.
    Since $K = k(\sqrt{a^2-4b})$, $a^2 - 4b$ is a square in $k^{\times}$. Therefore $d$ and $\delta$ have the same image in $k^{\times}/k^{\times 6}$, so $C_{\text{twist}} \simeq C_{da,d^2b} \simeq C_{\delta a,\delta^2 b}$ by Lemma \ref{lemma: when are two bielliptic Picard curves isomorphic}.

    Suppose now that $K/k$ is a quadratic extension and let $\sigma \in \Gal(K/k)$ be the nontrivial element. 
    Lemma \ref{lemma: cocycle iso twist dual curve and Ikeda quartic} shows that $C_{\text{twist}}$ is isomorphic to the quadratic twist of $\mathcal{Q}$ along the extension $K/k$ and involution $(u,v)\mapsto (-u,v)$. 
    A computation (using \S\ref{subsection: twists}) shows that this twist is exactly $C_{\delta a,\delta^2 b}$.
\end{proof}

Theorem \ref{theorem: Ctwist is twist of original curve} has consequences for the Prym variety of $C$ and its dual. 
If $(X, \mu, \beta)$ is a $\mu_3$-abelian surface in the sense of Definition \ref{definition: mu3-abelian surface}, the $\mu_3$-action $\beta$ and $\mu_2$-action $[-1]$ combine to a $\mu_6$-action.
Therefore, just as in \S\ref{subsection: twists}, every $\delta\in k^{\times}/k^{\times 6}$ gives rise to a \define{sextic twist} $(X_{\delta}, \mu_{\delta}, \beta_{\delta}) = (X, \mu, \beta)_{\delta}$ of $(X, \mu,\beta)$. 
The formation of the $\mu_3$-abelian surfaces $\Prym(C, \gamma)$ and $\Prym(C, \gamma)^{\vee}$ from \S\ref{subsec: torelli} is compatible with sextic twisting: $\Prym((C, \gamma)_{\delta}) \simeq \Prym(C, \gamma)_{\delta}$ and $\Prym((C, \gamma)_{\delta})^{\vee} \simeq \Prym(C, \gamma)^{\vee}_{\delta}$.

\begin{corollary}\label{corollary: marked mu3-surface and dual are isomorphic up to twist}
 Let $(C, \gamma) = (C_{a,b}, \gamma_{a,b})$ be a marked bielliptic Picard curve over $k$.
 Let $\delta = 16b(a^2 - 4b)$.
 Let $\Prym(C, \gamma) = (P, \lambda, \alpha)$ and $\Prym(C,\gamma)^{\vee} = (A, \widehat{\lambda}, \hat{\alpha})$.
 Then $(P, \lambda, \alpha) \simeq (A, \widehat{\lambda}, \hat{\alpha}^{-1})_{\delta}$.
 In particular, $P_{k^{\sep}}$ and $A_{k^{\sep}}$ are isomorphic as $(1,2)$-polarized abelian surfaces and $P_{k^{\sep}}$ is self-dual.  
\end{corollary}
\begin{proof}
    The $\mu_3$-action $\alpha$ and the involution $[-1]$ on $P$ combine to a $\mu_6$-action that restricts to $\mu_6$-actions $\hat{\gamma}$ and $\gamma_{\text{twist}}$ on $\widehat{C}$ and $C_{\text{twist}}$ respectively. 
    Lemma \ref{lemma: bigonal dual is a bielliptic Picard curve} shows that $(\widehat{C}, \hat{\gamma})$ is a marked bielliptic Picard curve; so is $(C_{\text{twist}}, \gamma_{\text{twist}}^{-1})$ by Proposition \ref{proposition: two mu3-fixed points are bielliptic picard with opposite signature}.
    Therefore $(P, \lambda, \alpha) \simeq \Prym(\widehat{C}, \hat{\gamma})^{\vee}$
    and $(P, \lambda, \alpha^{-1}) \simeq \Prym(C_{\text{twist}}, \gamma_{\text{twist}}^{-1})^{\vee}$ by Proposition \ref{proposition: smooth member pencil determines the surface}.

    It remains to show that $(P, \lambda, \alpha) \simeq (A, \widehat{\lambda}, \hat{\alpha}^{-1})_{\delta}$, or equivalently that $\Prym(C_{\text{twist}}, \gamma_{\text{twist}}^{-1})^{\vee}\simeq (A, \widehat{\lambda}, \hat{\alpha})_{\delta}$.
    Recall that $(C_{\delta}, \gamma_{\delta})$ denotes the sextic twist of $C$ along $\delta$.
    Theorem \ref{theorem: Ctwist is twist of original curve} shows that $(C_{\text{twist}}, \gamma_{\text{twist}}^{-1}) \simeq (C_{\delta}, \gamma_{\delta})$. 
    Therefore 
    \[\Prym(C_{\text{twist}}, \gamma_{\text{twist}}^{-1})^{\vee} \simeq \Prym(C_{\delta}, \gamma_{\delta})^{\vee} \simeq (A, \hat\lambda, \hat\alpha)_{\delta},\]
    as desired.
\end{proof}

\begin{remark}
{\em
    In contrast to the results of \S\ref{subsection: the dual prym A}, \ref{subsection: bigonal duality}, \ref{subsection: generalities (1,2)-polarized surfaces}, Corollary \ref{corollary: marked mu3-surface and dual are isomorphic up to twist} crucially uses properties specific to bielliptic Picard curves. 
    Indeed, the last sentence of the corollary might fail for a general bielliptic genus $3$ curve.
    }
\end{remark}

\begin{remark}\label{remark: universal iso between Prym and dual}
{\em
    The isomorphism $A \simeq P_{\delta^{-1}}$ of Corollary \ref{corollary: marked mu3-surface and dual are isomorphic up to twist} can be chosen `universally'.
    More precisely, let $R := \Z[1/6,a,b,\Delta_{a,b}^{-1}]$ and $S := \Spec(R)$.
    Equation \eqref{equation: bielliptic picard curve equation} defines the universal marked bielliptic Picard curve (in the sense of Remark \ref{remark: isomorphism of moduli stacks}) with Prym variety $\mathcal{P}\rightarrow S$ and dual $\mathcal{A}\rightarrow S$.
    Then there exists an isomorphism $\Phi\colon \mathcal{A}\rightarrow \mathcal{P}_{\delta^{-1}}$ that restricts to an isomorphism of Corollary \ref{corollary: marked mu3-surface and dual are isomorphic up to twist} for every geometric point of $S$.
    (This can be seen, for example, by spreading out an isomorphism over the generic point to all of $S$ using \cite[\S2, Lemma 1]{Faltings-finitenesstheorems}.)
}
\end{remark}

As a first application of Corollary \ref{corollary: marked mu3-surface and dual are isomorphic up to twist}, we show that the Atkin--Lehner involution on $X/w_3$ corresponds to inversion on the $j$-line $Y \simeq \P^1_j$.

\begin{lemma}
    Under the isomorphisms $(X/w_3)^{\bar{w}}\simeq Y\simeq \P^1_j$ of Proposition \ref{proposition: Y is quotient of shimura curve twist} and Lemma $\ref{lemma: Y isomorphic to P1}$, the involution $\bar{w}$ on $(X/w_3)^{\bar{w}}$ corresponds to the involution $j\mapsto 1/j$ on $\P^1_j$.
\end{lemma}
\begin{proof}
    We use the notations of \S\ref{subsec: moduli of bielliptic Picard curves as shimura curve}, \ref{subsec: comparison with quaternionic curve}.
    It suffices to check the claim over $\C$ on the locus of indecomposable $\mu_3$-abelian surfaces $Y^{\text{indec}}\subset Y$, so let $y = [A,\lambda, \alpha] \in Y^{\text{indec}}(\C)$ be such a point. 
    Let $x = [A, \lambda, \iota] \in X^{w_2}(\C) = X(\C)$ be a lift of $y$ under $\pi_2$, so $\alpha = \iota|_{\langle \omega \rangle}$.
    Then $\bar{w}(y) = \pi_2(w_2(x))$.
    Since $j$ has norm $2$, $w_2(x) = [A, \lambda, \iota \circ [j]]$.
    Moreover $[j]$ maps $\omega$ to $\omega^{-1}$, so $\iota\circ [j]$ restricts to $\alpha^{-1}$. 
    It follows that $\pi_2(w_2(x)) = [A, \lambda, \alpha^{-1}]$.
    On the other hand, if $(A, \lambda, \alpha) \simeq  \Prym(C, \gamma)^{\vee}$, then Corollary \ref{corollary: marked mu3-surface and dual are isomorphic up to twist} shows that $(A, \lambda, \alpha^{-1})\simeq \Prym(C, \gamma)$, which is itself isomorphic to $\Prym(\widehat{C}, \hat{\gamma})^{\vee}$, where $(\widehat{C}, \hat{\gamma})$ is the bigonal dual curve from \S\ref{subsection: bigonal duality}.
    Since the $j$-invariant of $\widehat{C}$ equals the inverse of the $j$-invariant of $C$ (Lemma \ref{lemma: j-invariant bigonal dual}), the lemma follows from the explicit description of the isomorphism $Y\simeq \P^1_j$. 
\end{proof}

\section{Explicit quaternionic multiplication}\label{endomorphisms of the Prym variety}

In Corollary \ref{cor:QM existence} we proved that there exists an action of $\calO$ on $P_{k^{\sep}}$. Next we use Theorem \ref{theorem: Ctwist is twist of original curve} to describe this action explicitly. 
For the remainder of this paper, we fix a universal isomorphism $\Phi\colon \mathcal{A}\rightarrow \mathcal{P}_{\delta^{-1}}$ as in Remark \ref{remark: universal iso between Prym and dual}.

\subsection{Explicit quaternionic multiplication}\label{subsec: the prym has QM}

Let $(C, \gamma) = (C_{a,b}, \gamma_{a,b})$ be a marked bielliptic Picard curve over $k$.
Let $\delta = \Delta_{a,b}= 16 b(a^2-4b)$, which is nonzero since $C$ is smooth.
Consider the $\Gal_k$-set:
\begin{align}\label{equation: definition of set S}
\mathcal{S} := \{(\zeta, \varepsilon) \in k^{\sep}\times k^{\sep} \mid \zeta \text{ is a primitive sixth root of unity and }\varepsilon^6 = \delta \}.
\end{align}
Consider $\mu_3$-abelian surfaces $\Prym(C, \gamma) = (P, \lambda, \alpha)$ and $\Prym(C, \gamma)^{\vee} = (A, \widehat{\lambda}, \hat{\alpha})$.
For every $(\zeta, \varepsilon)\in \mathcal{S}$, we will define elements $r_{\zeta}, s_{\varepsilon} \in \End(P_{k^{\sep}})$.
Since $\zeta^2$ is a third root of unity, we may set $r_{\zeta} := \alpha(\zeta^2)  \in \End(P_{k^{\sep}})$.

To define $s_{\varepsilon}$, we use Corollary \ref{corollary: marked mu3-surface and dual are isomorphic up to twist}.
By that corollary and Remark \ref{remark: universal iso between Prym and dual}, the specialization of the universal isomorphism $\Phi$ at the $k$-point $(a,b) \in S(k)$ is an isomorphism $\phi \colon (A, \widehat{\lambda}, \hat{\alpha}^{-1}) \xrightarrow{\sim} (P, \lambda, \alpha)_{\delta^{-1}}$.
The element $\varepsilon$ determines an isomorphism $C_{\delta^{-1}, k^{\sep}} \rightarrow C_{k^{\sep}}, (x,y)\mapsto (\varepsilon^3x,\varepsilon^4y)$, hence an isomorphism $\chi_{\varepsilon} \colon (P, \lambda, \alpha)_{\delta^{-1},{k^{\sep}}} \rightarrow (P, \lambda, \alpha)_{k^{\sep}}$ between their Prym varieties.
Let $s_{\varepsilon}$ be the composition:
\begin{align}\label{equation: definition s}
 P_{k^{\sep}} \xrightarrow{\lambda} A_{k^{\sep}} \xrightarrow{\phi} (P_{\delta^{-1}})_{k^{\sep}} \xrightarrow{\chi_{\varepsilon}} P_{k^{\sep}}.  
\end{align}

\begin{proposition}\label{proposition: Prym has O-PQM}
For every $x = (\zeta, \varepsilon) \in \mathcal{S}$ the assignment $\omega\mapsto r_{\zeta}, j\mapsto s_{\varepsilon}$ extends to an embedding 
\begin{align}
    \iota_x\colon \calO \rightarrow \End(P_{k^{\sep}}) = \End(P_{\bar{k}}).
\end{align}
\end{proposition}
%For a reference as to why End(P_{k^{\sep}}) and End(P_{\bar{k}}) are the same, see Theorem 2.1 in Conrad's `Gross--Zagier revisited'.

\begin{proof}
We may assume $k = k^{\sep}$.
Fix $x = (\zeta, \varepsilon) \in \mathcal{S}$ and write $r = r_{\zeta}$, $s = s_{\varepsilon}$.
We will show that $r^2+ r+1 = 0$, $s^2=2$ and $sr = r^{-1}s$. 
Since the associative ring $\calO$ from \S\ref{subsec: the quaternion order O} is generated by $\omega,j$ and satisfies the same three relations, this will imply that there exists a unique ring homomorphism $\iota_x\colon \calO \rightarrow \End(P_{k^{\sep}})$ with $\iota_x(\omega) = r$ and $\iota_x(j) = s$.
Such a homomorphism is automatically an embedding since $\calO\otimes \Q$ is a simple algebra.

Equation \eqref{equation: endomorphism satisfies r^2+r+1=0} of Lemma \ref{lemma: mu3-fixed points Prym} shows that $r^2+ r+ 1=0$.
To check that $sr = r^{-1} s$, note that the surfaces $P, P_{\delta^{-1}}, A$ have $\mu_3$-actions $\alpha, \alpha_{\delta^{-1}}, \hat{\alpha}$ respectively. 
The morphisms $\lambda$ and $\chi_{\varepsilon}$ from \eqref{equation: definition s} are equivariant with respect to these $\mu_3$-actions.
On the other hand, $\phi$ intertwines the $\mu_3$-action on the domain with the \emph{inverse} of the $\mu_3$-action on the target. 
The same must therefore be true for their composition $s$, so $sr = r^{-1}s$.

It suffices to prove that $s^2=2$; we will use the results of \S\ref{subsec: comparison with quaternionic curve} to achieve this.
Write $\psi = \chi_{\varepsilon} \circ \phi$.
The morphism $\psi$ is compatible with the polarizations since $\phi$ and $\chi_{\varepsilon}$ are, so $\psi^{\vee} \circ \lambda \circ \psi = \widehat{\lambda}$.
Therefore 
\begin{align}\label{equation: s^2 = 2}
    s^2 = \psi  \lambda \psi  \lambda = (\psi (\psi^{\vee})^{-1}) (\psi^{\vee} \lambda  \psi  \lambda) = (\psi  (\psi^{\vee})^{-1}) (\widehat{\lambda}  \lambda) = 2(\psi (\psi^{\vee})^{-1}).
\end{align}
It follows that $s^2=2$ if and only if $\psi^{\vee} = \psi$.
To prove the latter statement, we will first show that $(\psi^{\vee})^{-1}\circ \psi \in \Aut(A,\widehat{\lambda}, \hat{\alpha})$.
The identity $\psi^{\vee}\circ \lambda = \widehat{\lambda}\circ \psi$ implies that $\lambda \circ (\psi \circ \lambda \circ \psi^{\vee} - \widehat{\lambda}) \circ \lambda = 0$.
Since $\lambda$ is an isogeny, the inner expression must be zero, so $\psi \circ \lambda \circ \psi^{\vee} = \widehat{\lambda}$. In other words $\psi^{\vee}$ is compatible with the polarizations.
Moreover, $\psi$ intertwines the $\mu_3$-action $\hat{\alpha}$ on $A$ with the \emph{inverse} of the $\mu_3$-action $\alpha$ on $P$: $\psi(\hat{\alpha}(\omega)(x)) = \alpha(\omega)^{-1}(\psi(x))$ for all $x\in A$ and $\omega \in \mu_3$.
Applying duality to this identity and using that $\alpha^{\vee}(\omega) = \hat{\alpha}(\omega^{-1})$, $\psi^{\vee}$ again intertwines the $\mu_3$-action with the inverse of the other. 
We conclude that $(\psi^{\vee})^{-1}\circ\psi \in \Aut(A, \lambda, \alpha)$.

To show that this automorphism is the identity, we will use that its construction can be carried out in families.
More precisely, let $R = \Z[1/6,a,b,\Delta_{a,b}^{-1}]$ and $\tilde{R} = R[\varepsilon]/(\varepsilon^6-\Delta_{a,b})$.
Equation \eqref{equation: bielliptic picard curve equation} defines a universal marked bielliptic Picard curve $(\mathcal{C}, \gamma)$ over $S = \Spec(R) \subset \A^2_{\Z[1/6]}$ with Prym variety $(\mathcal{P}, \lambda, \alpha)$ and dual $(\mathcal{A}, \widehat{\lambda}, \hat{\alpha})$. 
On the cover $\tilde{S} = \Spec (\tilde{R}) \rightarrow S$ given by adjoining a sixth root of $\Delta_{a,b}$, there exists an isomorphism $\Psi\colon \mathcal{A}_{\tilde{S}} \rightarrow \mathcal{P}_{\tilde{S}}$ with the property that $\Psi$ specializes to $\psi$ for all morphisms $\Spec(k) \rightarrow \tilde{S}$ and that $(\Psi^{\vee})^{-1}\circ\Psi \in \Aut_{\tilde{S}}(\mathcal{A}, \lambda, \alpha)$.
Since the moduli stack of $(1,2)$-polarized abelian surfaces is separated and Deligne--Mumford, $\mathbf{Aut}_{\tilde{S}}(\mathcal{A}, \lambda, \alpha) \rightarrow \tilde{S}$ is finite and unramified.
Since finite unramified morphisms are \'etale locally on the target disjoint unions of closed immersions \cite[Tag \href{https://stacks.math.columbia.edu/tag/04HJ}{04HJ}]{stacksproject} and the base $\tilde{S}$ is connected, it suffices to show that $(\Psi^{\vee})^{-1}\Psi$ is the identity for a \emph{single} specialization $s\in \tilde{S}$.

The specialization $(a_0,b_0,\varepsilon_0) = (1,1/8,1)$ defines a $\Q$-point $s_0$ of $\tilde{S}$.
Theorem \ref{theorem: Torelli theorem bielliptic Picard curves} and Lemma \ref{lemma: automorphism group bielliptic Picard curves} shows that $\mathbf{Aut}((\mathcal{A}, \lambda, \alpha)_{s_0}) = \mu_6$.
Since $\mu_6(\Q) = \{\pm 1\}$, either $(\Psi^{\vee})^{-1}\Psi = 1$ or $(\Psi^{\vee})^{-1}\Psi = -1$.
Assume for the sake of contradiction that $(\Psi^{\vee})^{-1}\Psi = -1$.
Then \eqref{equation: s^2 = 2} shows that $s^2=-2$.
This implies that $r$ and $s$ determine an embedding of the \emph{definite} quaternion algebra $(-3,-2)_{\Q}$ into $\End(\mathcal{P}_{s})\otimes \Q$ for every specialization $s\in \tilde{S}(\C)$.
By the classification of endomorphism algebras of complex abelian surfaces \cite[Proposition 5.5.7, Exercise 9.10(1) and Exercise 9.10(4)]{BirkenhakeLange-Complexabelianvarieties}, this implies that $\mathcal{P}_s$ is isogenous to the square of a CM elliptic curve.
This is a contradiction, since there are (uncountably) many specializations $s$ for which $\mathcal{P}_{s}$ is simple, using Proposition \ref{proposition: Y is quotient of shimura curve twist} and CM theory for the quaternionic Shimura curve $X$.
We conclude that $\psi^{\vee} = \psi$ and hence by \eqref{equation: s^2 = 2} that $s^2 = 2$, as desired.
\end{proof}

If $P_{\bar{k}}$ is simple, then $\calO$ is the full ring of endomorphisms of $P_{\bar{k}}$:

\begin{lemma}\label{lemma: if P simple then O = End(P)}
    Let $C/k$ be a bielliptic Picard curve with Prym variety $P$ and suppose that $k$ is algebraically closed. Then $P$ is not simple $\Leftrightarrow$ 
    $\calO \not\simeq \End(P)$
    $\Leftrightarrow$
    $P$ is isogenous to the square of a CM elliptic curve.
    Consequently, if $P$ is simple then $\End(P) =\iota_x(\calO)$ for every $x\in \mathcal{S}$.
\end{lemma}
\begin{proof}
    This is a direct consequence of the Albert classification \cite[Proposition 5.5.7, Exercise 9.10(1) and Exercise 9.10(4)]{BirkenhakeLange-Complexabelianvarieties}.
\end{proof}
\begin{remark}{\em 
If $\mathrm{char}(k) = 0$ or if $k = \overline{\F}_p$ with $p > 3$, then these conditions are also equivalent to the condition that $P$ is {\it isomorphic} to a product of elliptic curves. See \cite{MitaniShioda} for the former case and \cite[Theorem 3.1]{RibetBimodules} for the latter.}
\end{remark}
We will study geometrically split Prym varieties in more detail in \S\ref{subsec: split Prym varieties}.

\subsection{Galois action on the endomorphisms}\label{subsec: endomorphism field}

Keep the notations of the previous section.
We proceed to determine how the various embeddings $\iota_x\colon \calO \rightarrow \End(P_{k^{\sep}})$ are related and deduce an explicit description of the Galois action on $\calO$.

Let $\Sym(\mathcal{S})$ be the group of bijections $\mathcal{S} \rightarrow \mathcal{S}$ acting on the right on the $\Gal_k$-set $\mathcal{S}$ of \eqref{equation: definition of set S}.
Define the subgroup $D_6\subset \Sym(\mathcal{S})$ generated by $\rho, \sigma$, where $(\zeta, \varepsilon)^{\rho} =(\zeta, \zeta^{-1} \varepsilon)$ and $(\zeta, \varepsilon)^{\tau} \mapsto (\zeta^{-1}, \varepsilon)$.
Then $D_6$ is a dihedral group of order $12$ and $D_6$ acts simply transitively on $\mathcal{S}$.

Recall the subgroup $\Aut_i(\calO)$ from Lemma \ref{lemma: explicit description of AutiO}.
The assignment $\rho\mapsto [1-\omega]$, $\tau\mapsto [j]$ induces an isomorphism $\varphi \colon D_6\xrightarrow{\sim} \Aut_i(\calO)$.

\begin{lemma}\label{lem: compatibility endo}
    For all $x \in \mathcal{S}$, $g \in D_6$ and $b\in \calO$, $\iota_{x^g}(b^{\varphi(g)}) =  \iota_{x}(b)$
\end{lemma}
\begin{proof}
    It suffices to prove the claimed identity when $g\in \{\tau, \rho\}$ and $b\in \{ \omega,j\}$, for which it can be checked using routine but intricate calculations.
    Write $x = (\zeta, \varepsilon)$.
    Suppose first that $g = \tau$, then $\iota_{x^g}(b^{\varphi(g)}) = \iota_{(\zeta^{-1}, \varepsilon)}(j^{-1}bj)$.
    If $b = \omega$, this equals $\iota_{(\zeta^{-1}, \varepsilon)}(j^{-1}\omega j) = \iota_{(\zeta^{-1}, \varepsilon)}(\omega^{-1}) = r_{\zeta^{-1}}^{-1} = r_{\zeta} = \iota_x(\omega)$.
    If $b = j$, this equals $\iota_{(\zeta^{-1}, \varepsilon)}(j) =s_{\varepsilon} = \iota_x(j)$.
    
    Suppose next that $g = \rho$. 
    If $b = \omega$, both sides are simply $r_{\zeta}$.
    If $b=j$, then
    \begin{align*}
        \iota_{x^g}(b^{\varphi(g)}) = \iota_{x^g}(j^{[1-\omega]}) 
        =\iota_{x^g}(-\omega^{-1}j) =-\iota_{(\zeta, \zeta^{-1}\varepsilon)}(\omega^{-1}) \iota_{(\zeta, \zeta^{-1}\varepsilon)}(j) 
        =-r_{\zeta}^{-1}s_{\zeta^{-1} \varepsilon}.
    \end{align*}
    The definition of $s_{\varepsilon}$ in \eqref{equation: definition s} shows that $s_{\zeta^{-1} \varepsilon}$ is given by postcomposing $s_{\varepsilon}$ with the automorphism of $P$ induced by $\gamma(\zeta^{-1})$. 
    A simple calculation shows that this is exactly $-r_{\zeta}$, so 
    $s_{\zeta^{-1} \varepsilon} = -r_{\zeta} s_{\varepsilon}$.
    Therefore $\iota_{x^g}(j^g) = -r_{\zeta}^{-1} (-r_{\zeta} s_{\varepsilon}) = s_{\varepsilon} = \iota_x(b)$.
\end{proof}

The Galois action on $\mathcal{S}$ determines a homomorphism $\Gal_k \rightarrow \Sym(\mathcal{S})$.
This homomorphism factors through an injection $\Gal(L/k)\hookrightarrow D_6$, where $L = k(\omega, \sqrt[6]{\delta})$ is the splitting field of $f(T) = T^6- \delta$.
Denote the composition $\Gal_k\rightarrow D_6\xrightarrow{\varphi} \Aut_i(\mathcal{O})$ again by $\varphi$.
Recall from \S\ref{subsec: notation and conventions} that if $f\colon X_{k^{\sep}} \rightarrow Y_{k^{\sep}}$ is a $k^{\sep}$-morphism between $k$-varieties and $\sigma \in \Gal_k$, then $f^{\sigma}$ denotes the $k^{\sep}$-morphism $x\mapsto f(x^{\sigma^{-1}})^{\sigma}$.

\begin{theorem}\label{theorem: galois action endo field on O}
    For all $x\in \mathcal{S}$, $\sigma \in \Gal_k$ and $b\in \calO$, $\iota_x(b)^{\sigma} = \iota_x(b^{\varphi(\sigma^{-1})})$.
    Consequently, there exists an embedding $\calO \subset \End(P_{k^{\sep}})$ such that $f^{\sigma} = f^{\varphi(\sigma^{-1})}$ for all $f\in \calO$ and $\sigma \in \Gal_k$.
    In particular, $\calO$ is $\Gal_k$-stable and defined over $L = k(\omega, \sqrt[6]{\delta})$.
\end{theorem}
\begin{proof}
    Write $x = (\zeta, \varepsilon)$.
    The explicit construction of $r_{\zeta}$ and $s_{\varepsilon}$ shows that $r_{\zeta}^{\sigma} = r_{\zeta^{\sigma}}$ and $s_{\varepsilon}^{\sigma} = s_{\varepsilon^{\sigma}}$.
    Therefore $\iota_{x}(b)^{\sigma} = \iota_{x^{\sigma}}(b)$.
    By the previous lemma, $\iota_{x}(b)^{\sigma} = \iota_{x^{\sigma}}(b) = \iota_x(b^{\varphi(\sigma^{-1})})$.
\end{proof}

The \define{endomorphism field} of an abelian variety is the smallest field extension over which all endomorphisms are defined \cite{GuralnickKedlaya-Endomorphismfield}.
If $P$ is geometrically simple, then $\End(P_{\bar{k}}) = \mathcal{O}$ by Lemma \ref{lemma: if P simple then O = End(P)}. 
Since the action of $D_6$ on $\calO$ is faithful, Theorem \ref{theorem: galois action endo field on O} has the following immediate corollary.

\begin{corollary}\label{corollary: determination endo field}
    Let $C = C_{a,b}$ be a bielliptic Picard curve over $k$.
    Suppose that $P_{\bar{k}}$ is simple.
    Then the endomorphism field of $P/k$ is $k(\omega, \sqrt[6]{16 b(a^2- 4b)})$.
\end{corollary}

The endomorphism field is invariant under isogeny, so the corollary holds verbatim for the dual Prym $A/k$.

\begin{example}\label{example: sato-tate group large}
{\em 
    Corollary \ref{corollary: determination endo field} gives many examples of abelian surfaces with large endomorphism fields.
    For example, for the curve $C = C_{3,4}: y^3 = x^4 + 3x^2 +4$ over $\Q$, the Prym $P = P_{3,4}$ is geometrically simple (see Proposition \ref{proposition: all CM j-invariants}) and its endomorphism field is $\Q(\sqrt{-3}, \sqrt[6]{-7})$, a dihedral degree $12$ extension of $\Q$.
    In the notation of \cite{FKRS-STgenus2}, this seems to be the first published example of a geometrically simple abelian surface over $\Q$ with Sato--Tate group \href{https://www.lmfdb.org/SatoTateGroup/1.4.E.12.4a}{$J(E_6)$}.
    Similarly, $P_{-4,2}$ is a geometrically simple abelian surface with endomorphism field $\Q(\sqrt{-3}, \sqrt[3]{2})$ and Sato--Tate group \href{https://www.lmfdb.org/SatoTateGroup/1.4.E.6.1a}{$J(E_3)$}.
    }
\end{example}

We also record a simple description of the Artin representations $\End^0(P_{k^{\mathrm{sep}}})$ and $\NS(P_{k^{\mathrm{sep}}}) \otimes \Q$. Recall that $\NS(P_{k^{\mathrm{sep}}})$ is the group of line bundles on $P_{k^\mathrm{sep}}$ modulo algebraic equivalence.
Write $G = \Gal(L/k)$ and view it as a subgroup of $D_6$.
Let $\chi\colon D_6\rightarrow \{\pm 1\}$ be the unique homomorphism with kernel $C_6 =\langle \rho \rangle$ and let $\mathrm{std}\colon D_6\rightarrow \GL_2(\Q)$ be a model for the unique $2$-dimensional irreducible faithful representation of $D_6$.
\begin{corollary}\label{lem:NS galois action}
    Suppose that $P_{\bar{k}}$ is simple.
    Then there are isomorphisms of $G$-representations $\End^0(P_{k^{\mathrm{sep}}})\simeq (1\oplus \chi \oplus \mathrm{std})|_G$ and $\NS(P_{k^{\mathrm{sep}}})\otimes \Q \simeq (1 \oplus \mathrm{std})|_G$.
\end{corollary}
\begin{proof}
    The first part immediately follows from Theorem \ref{theorem: galois action endo field on O}: $1\in B$ generates the trivial representation, $i\in B$ generates $\chi$ and $\{j,ij\}$ generate $\mathrm{std}$.
    To prove the second isomorphism, recall that the $\Gal_k$-equivariant map $L \mapsto \lambda^{-1}\phi_L$ gives an isomorphism between $\NS(P_{k^\mathrm{sep}}) \otimes \Q$ and the subgroup of $\End^0(P_{k^\mathrm{sep}})$ fixed by the Rosati involution. 
    The Rosati involution on $\End^0(P_{k^{\mathrm{sep}}})\simeq B$ is given by $b\mapsto i\bar{b}i^{-1}$, the unique positive involution on $B$ that induces complex conjugation on $\Q(i) = \Q(\omega)$.
    It follows that $\NS(P_{k^{\mathrm{sep}}})\otimes \Q \simeq \{ f \in B \mid i\bar{f} i^{-1} = f\} = \text{span} \{ 1\} \oplus \text{span}\{j,ij\} \simeq (1\oplus \mathrm{std})|_G$.
\end{proof}

\subsection{Prym surfaces of \texorpdfstring{$\GL_2$}{GL2}-type}\label{subsec:gl2type}

%Let $k(\omega)$ be the smallest field extension of $k$ that contains a primitive third root of unity.
Let $C = C_{a,b}$ be a bielliptic Picard curve over $k$ with Prym variety $P$, and let $\iota \colon \calO \hookrightarrow \End(P_{k^{\sep}})$ be an embedding satisfying the conclusion of Theorem $\ref{theorem: galois action endo field on O}$.
Recall that $k(\omega)$ denotes the smallest field extension of $k$ containing a primitive third root of unity.

\begin{corollary}\label{corollary: pryms of gl2 type}
The following conditions are equivalent:
\begin{enumerate}
    \item $16b(a^2 - 4b)$ is a sixth power in $k(\omega)$;
    \item $\iota(\calO) \subset \End(P_{k(\omega)})$;
    \item $\iota(\calO) \cap \End(P)$ is either an order in a real quadratic field or all of $\iota(\calO)$.
\end{enumerate}
\end{corollary}
\begin{proof}
$(1)$ and $(2)$ are equivalent by Theorem \ref{theorem: galois action endo field on O}.
To prove the equivalence between $(2)$ and $(3)$, let $G$ be the image of the homomorphism $\varphi\colon \Gal_k\rightarrow \Aut_i(\calO)$ of Theorem \ref{theorem: galois action endo field on O}, and let $H = \text{Im}(\Gal_{k(\omega)}\rightarrow \Aut_i^+(\calO))$.
Then $H = G \cap \Aut_i^+(\calO)$, where $\Aut_i^+(\calO) = \langle [1-\omega]\rangle$.
A group theory calculation using Lemma \ref{lem:subfields} shows that $H = \{1\}$ (in other words, $(2)$ holds) if and only if $\calO^G$ contains an order in a real quadratic field (in other words $(3)$ holds).
\end{proof}

\begin{remark}
{\em
    Recall that an abelian surface $X/k$ is \define{of $\GL_2$-type} if $\End(X)$ contains a quadratic ring.
    If $P$ is geometrically simple and $k \neq k(\omega)$
    then $\iota(\calO) \cap \End(P) = \End(P)$ (Lemma \ref{lemma: if P simple then O = End(P)}) and $\End(P)$ is either $\Z$ or an order in a real quadratic field by Lemma \ref{lem:subfields},
    so the  conditions above are equivalent to $P$ being of $\GL_2$-type. If $k = k(\omega)$ then $\Z[\omega] \hookrightarrow \End(P)$, so $P$ is always of $\GL_2$-type.
}
\end{remark}
We can be more precise about the ring of endomorphisms $\End(P)$. 
A calculation shows that an element $\delta\in k^{\times}$ is a sixth power in $k(\omega)$ if and only if either $\delta$ or $-27\delta$ is a sixth power in $k$. 
This dichotomy breaks up Corollary \ref{corollary: pryms of gl2 type} into the following two cases.

\begin{corollary}\label{cor:sqrt(2)-multiplication}
    Assume that $j(C) \neq 1$. Then the following conditions are equivalent.
    \begin{enumerate}
        \item $\Delta_{a,b} = 16b(a^2 - 4b)$ is a sixth power in $k$.
        \item $\End(P) \cap \iota(\calO)$ contains a subring isomorphic to $\Z[\sqrt{2}]$. 
        \item There exist $t,d \in k^\times$ such that 
        $(a,b) = (2(t^2 + 1)^2td^3,(t^2+1)^3t^2d^6)$. 
     \end{enumerate}
     If these conditions are satisfied then $P$ and $A = P^\vee$ are $k$-isomorphic. A sextic twist of $P$ satisfies the above conditions if and only if $j(C) = -t^2$ for some $t \in k^\times$.
\end{corollary}
\begin{proof}
$(1)$ implies $P \simeq A$ by Corollary \ref{corollary: marked mu3-surface and dual are isomorphic up to twist}. 
To show $(1) \Leftrightarrow (2)$, let $H\leq G$ be the subgroups of $\Aut_i(\calO)$ defined in the proof of Corollary \ref{corollary: pryms of gl2 type}.
Then $\calO^G$ contains a subring isomorphic to $\Z[\sqrt{2}]$ if and only if $G$ is conjugate to a subgroup of $\langle [j]\rangle$ by Lemma \ref{lem:subfields}.
Theorem \ref{theorem: galois action endo field on O} and Galois theory show that this is equivalent to $T^6-\delta$ having a $k$-rational root.
This proves that $(1)\Leftrightarrow (2)$. The equivalence of $(1)$ and $(3)$ as well as the last claim of the corollary are routine algebra.     
\end{proof}

\begin{corollary}\label{cor:sqrt(6)-multiplication}
    Assume that $j(C) \neq 1$. Then the following conditions are equivalent.
    \begin{enumerate}
        \item $-27\Delta_{a,b} = -432b(a^2 - 4b)$ is a sixth power in $k$.
        \item $\End(P) \cap \iota(\calO)$ contains a subring isomorphic to $\Z[\sqrt{6}]$. 
        \item There exist $t,d\in k^\times$ such that $(a,b) = (18d^3t(1 - 3t^2)^2, 3^4d^6t^2(1 - 3t^2)^3)$.
     \end{enumerate}
     If these conditions hold then $P \simeq A_{-27}$, i.e.\ $P$ is isomorphic to the quadratic twist of $A$ along $k(\omega)$. A sextic twist of $P$ satisfies the above conditions if and only if $j(C) = 3t^2$ for some $t \in k^\times$.
\end{corollary}
\begin{proof}
The proof is similar to that of Corollary \ref{cor:sqrt(2)-multiplication} and is omitted.
\end{proof}

\begin{remark}
{\em
    Corollaries \ref{cor:sqrt(2)-multiplication} and \ref{cor:sqrt(6)-multiplication} show that a $j$-invariant $j\in Y(\Q) = \P^1(\Q)$ lifts to a $\Q$-rational point under the morphism $\pi_2$ (respectively $\pi_6$) from \S\ref{subsec: comparison with quaternionic curve} if and only if $j \in -\Q^{\times 2}$ (respectively $j\in 3\Q^{\times 2}$). We also note that the equivalence of $(1)$ and $(2)$ holds even when $j(C) = 1$.
}
\end{remark}

\begin{corollary}\label{cor:principal polarization}
 Suppose $k \neq k(\omega)$ and $P$ is geometrically simple. Then 
 \begin{enumerate}
 \item $\End(P)$ is isomorphic to either $\Z$, $\Z[\sqrt{2}]$ or $\Z[\sqrt{6}]$.
 \item $P$ carries a principal polarization if and only if $\End(P) \simeq \Z[\sqrt{2}]$, and in this case the unique principal polarization is $\frac12\lambda(2 -\sqrt2)$.
 \end{enumerate}
\end{corollary}

\begin{proof}
 $(1)$ follows from Lemma \ref{lem:subfields}. 
 For $(2)$,  if $\End(P) = \Z$, then $\Hom(P,A) = \Z\lambda$ so cannot contain a principal polarization. Otherwise, we use the fact that any abelian surface with $\Z[\sqrt{2}]$-RM is principally polarized, whereas any $(1,2)$-polarized abelian surface $P$ with $\End(P) \simeq \Z[\sqrt{6}]$ is not, by \cite[Proposition 3.11]{GonzalesGuardiaRotger}.   If $\End(P) = \Z[\sqrt{2}]$, the unique principal polarization $\lambda' \colon P \to A$ is $\frac12\lambda(2 - \sqrt{2})$; see \cite[Proposition 2.1]{GonzalesGuardiaRotger}. Equivalently, $\lambda' = \lambda - \psi^{-1}$, in the notation of the proof of Proposition \ref{proposition: Prym has O-PQM}.
\end{proof}

\subsection{Geometrically split Prym varieties}\label{subsec: split Prym varieties}

In this subsection, we assume that $k$ has characteristic zero.
Let $C/k$ be a bielliptic Picard curve with Prym variety $P$.
If any of the equivalent conditions of Lemma \ref{lemma: if P simple then O = End(P)} is satisfied for $P_{\bar{k}}$ we call $C$ (or $P$) \define{CM}.
Being CM only depends on the $j$-invariant \eqref{equation: j-invariant} of $C$.
Choose an embedding $\iota\colon \calO \hookrightarrow \End(P_{\bar{k}})$ of the form $\iota_x$ for some $x\in \mathcal{S}$ as in \S\ref{subsec: the prym has QM}. Let
\begin{align}
    \End_{\calO}(P_{\bar{k}}) := \{f\in \End(P_{\bar{k}}) \mid f\circ \iota(b) =\iota(b) \circ f \text{ for all } b\in \calO\}.
\end{align}
This subring does not depend on the choice of $x\in \mathcal{S}$.
If $P$ is not CM, then $\End_{\calO}(P_{\bar{k}}) = \Z$; if $P$ is CM, then $\End_{\calO}(P_{\bar{k}})$ is an order $R$ in an imaginary quadratic field, and we say that \define{$P$ has CM by the order $R$}.

\begin{example}\label{example: special curve has CM by -4}
    {\em The Prym variety $P$ of the special bielliptic Picard curve $C:y^3 = x^4+1$ has CM by the order $\Z[i]$, hence $j(C) =1$ is a CM $j$-invariant. (In this example alone, the symbol $i$ means $\sqrt{-1}$.)
    Indeed, over $\bar{\Q}$, the curve $C$ has an automorphism $\beta(x,y) = (ix,y)$ commuting with the $\mu_6$-action on $C$.
    A signature calculation similar to Lemma \ref{lemma: mu3-action on P has char poly T^2+T+1} shows that $\beta$ induces an order $4$ automorphism $\beta_*$ on $P$ and has characteristic polynomial $X^2+1$.
    In the notation of \S\ref{subsec: the prym has QM}, choose a pair $(\zeta, \varepsilon) \in \mathcal{S}$.
    By construction, the elements $r_{\zeta}$ and $s_{\varepsilon}$ commute with $\beta_*$, so $\Z[\beta_*] = \Z[i] \subset \End_{\calO}(P_{\bar{k}})$.
    }
    %Here is why \beta_* commutes with $r_{\zeta}$ and $s_{\varepsilon}$. It commutes with $r_{\zeta}$ because it commutes with the $\mu_3$-action on the curve. To show that it commutes with $s_{\varepsilon}$, it suffices to show that we can choose the isomorphism $C\rightarrow C_{\text{twist}}$ to be equivariant with respect to $\beta_*$, which can be checked explicitly. 
\end{example}

We now state some basic properties of the set of CM $j$-invariants, completely analogous to the elliptic curve setting, using the results of \S\ref{subsec: comparison with quaternionic curve}.
In the notation of that subsection, a $j$-invariant $j\in Y(\C)$ is CM if and only if it lifts to a CM point of $X$ under the map $\pi_2$.

\begin{lemma}\label{lemma: CM points are algebraic and there are finitely many of bounded degree}
    Every CM $j$-invariant is algebraic over $\Q$.
    Moreover, for every $d\in \Z_{\geq 1}$, there are only finitely many CM $j$-invariants $j$ with $[\Q(j):\Q]\leq d$.
\end{lemma}
\begin{proof}
Follows immediately from Proposition \ref{proposition: Y is quotient of shimura curve twist} and the corresponding facts concerning CM points Shimura curves \cite[\S2.4]{Elkies-Shimuracurvecomputations}. 
\end{proof}

We conclude by determining all $\Q$-rational CM $j$-invariants, equivalently all geometrically split $\Prym(C, \gamma)$ that can be defined over $\Q$. 
To this end, we will use Elkies' computations of all rational CM points on the full Atkin--Lehner quotient $X^*:= X/W\simeq \P^1_{\Q}$.
Following Elkies \cite[\S3.1]{Elkies-Shimuracurvecomputations}, let $t$ be the unique coordinate $X^* \xrightarrow{\sim} \P^1_t$ with the property that $t=0,1,\infty$ corresponds to the unique CM point by the order $\Z[\sqrt{-6}]$, $\Z[i]$, $\Z[\omega]$ respectively. 
Recall from Lemma \ref{lemma: Y isomorphic to P1} the isomorphism $j\colon Y\rightarrow \P^1$.
The isomorphism $Y\simeq (X/w_3)^{\bar{w}}$ of Proposition \ref{proposition: Y is quotient of shimura curve twist} determines a quotient map $\Pi\colon Y \simeq  (X/w_3)^{\bar{w}} \rightarrow ((X/w_3)^{\bar{w}})/\langle\bar{w}\rangle  = X^*$.

\begin{lemma}\label{lemma: comparison j-function and elkies coordinate}
    In the coordinates $j$ and $t$ of $Y$ and $X^*$ respectively, the map $\Pi\colon Y\rightarrow X^*$ is given by $j\mapsto \frac{(j+1)^2}{4j}$.
    A $k$-point $t\in k$ on $X^*$ lifts to a $k$-point on $Y$ if and only if $t(t-1)$ is a square in $k$.
\end{lemma}
\begin{proof}
    Elkies \cite[\S3.1, end of p 17]{Elkies-Shimuracurvecomputations} has shown that the biquadratic extension of function fields $\Q(X)/\Q(X^*)$ is given by adjoining $\sqrt{-t}$ and $\sqrt{3(t-1)}$ to $\Q(X^*)=\Q(t)$.
    The three intermediate quadratic subfields correspond to the three intermediate Atkin--Lehner quotients of $X$; analyzing the ramification shows that the morphism $X/w_3\rightarrow X^*$ can be realized as the double cover of $X^* = \P^1_t$ branched along the function $-3t(t-1)$, so $X/w_3$ has equation $s^2 = -3t(t-1)$.
    Moreover, the involution $(t,s)\mapsto (t,-s)$ corresponds to the involution $\bar{w}$.
    Proposition \ref{proposition: Y is quotient of shimura curve twist} shows that $Y\rightarrow X^*$ is the quadratic twist of $X/w_3$ along $\Q(\sqrt{-3})/\Q$, so $Y\rightarrow \P_t^1$ has equation $s^2= t(t-1)$. 
    %(Alternatively: $Y\rightarrow \P^1_t$ is branched at $0, 1$ so given by square root of $ct(t-1)$ for some $c\in \Q^{\times}$. Since the $\Z[\omega]$-CM points are rational above $\infty$, we may take $c=1$.)
    Rationally parametrizing this conic shows that $Y$ has a rational coordinate $r$ such that $s=(1-r^2)/(4r)$, $t = (1+r)^2/(4r)$ and the involution $\bar{w}$ corresponds to $r\mapsto 1/r$.
    
    We claim that after possibly replacing $r$ by $1/r$, which does not affect the expression $t = (1+r)^2/(4r)$, we have $j=r$. 
    Analyzing the ramification of $\Pi$, we see that $Y$ has two CM points by $\Z[\omega]$ (corresponding to $r=0,\infty$) and one CM point by $\Z[i]$ (corresponding to $r=1$).
    The points $j=0,\infty$ correspond to decomposable $\mu_3$-abelian surfaces, described in the proof of Lemma \ref{lemma: Y isomorphic to P1}, and have CM by $\Z[\omega]$.
    Example \ref{example: special curve has CM by -4} shows that $j=1$ corresponds to the CM point by $\Z[i]$.
    So $j$ and $r$ (or $1/r$) agree on the three points $0,1, \infty$, so they must agree everywhere, proving the claim.
\end{proof}

\begin{proposition}\label{proposition: all CM j-invariants}
    Table $\ref{table 1}$ is a complete list of rational CM $j$-invariants and the discriminants of the corresponding quadratic orders. 
    Consequently, a bielliptic Picard curve $C/\Q$ has geometrically non-simple Prym variety if and only if $j(C)$ appears in this table.
\end{proposition}
\begin{proof}
    A point in $Y(\Q)$ is CM if and only if its image under $\Pi\colon Y\rightarrow X^*$ is CM. 
    It therefore suffices to find all rational CM points on $X^*$ and determine which ones lift to $Y(\Q)$.
    Elkies \cite[Table 1]{Elkies-Shimuracurvecomputations} has determined all rational CM points in his $t$-coordinate.
    Some of these were only conjecturally CM, but \cite{Errthum-singularmodulishimuracurves} confirms this table unconditionally. 
    The proposition then follows from Lemma \ref{lemma: comparison j-function and elkies coordinate} and an elementary calculation.
\end{proof}

\begin{table}
\centering
\begin{tabular}{c | c | c  }
    $\Disc(R)$ & $|\Disc(R)|$ & $j$ or $1/j$ \\
    \hline 
    $-3$ & $3$ & $0$ \\
    $-4$ & $2^2$ & $1$ \\
    $-24$ & $2^3\cdot 3$ &$-1$ \\
    $-75$ & $3\cdot 5^2$ &$256/135$ \\
    $-84$ & $2^2 \cdot 3 \cdot 7$ & $-27$ \\
    $-120$& $2^3 \cdot 3 \cdot 5$ & $27/125$ \\
    $-228$ & $2^2 \cdot 3 \cdot 19$ & $15625/729$ \\
    $-147$ & $3 \cdot 7^2$ & $-48384/15625$ \\
    $-372$& $2^2 \cdot 3 \cdot 31$  & $-1771561/421875$ \\
    $-408$ & $2^3 \cdot 3 \cdot 17$ & $-11390625/4913$ \\
\end{tabular}
\caption{Rational CM $j$-invariants and their discriminants}
\label{table 1}
\end{table}

\section{\texorpdfstring{$6$}{6}-torsion points in the Prym variety}\label{sec:2,3,6-torsion}

To prove Theorem \ref{thm:main}, we  must analyze the Galois modules $P[n]$ of bielliptic Picard Pryms.  In this section we study $P[2]$, $P[3]$ and $P[6]$ explicitly and give various criteria for the existence of rational torsion points.

For the remainder of this section, let $(C, \gamma) = (C_{a,b}, \gamma_{a,b})$ be a marked bielliptic Picard curve over a field $k$ (always assumed of characteristic $\neq 2,3$).

\subsection{\texorpdfstring{$2$}{2}-torsion}

Recall from \S\ref{subsection: the prym variety P}-\ref{subsection: the dual prym A} that $[2] = \widehat{\lambda} \lambda$, giving rise to a short exact sequence
\begin{equation}\label{eq: SES involving P[2] and P[lambda]}
 0\rightarrow P[\lambda] \rightarrow P[2] \rightarrow A[\widehat{\lambda}]\rightarrow 0.  
\end{equation}
The diagram from \S\ref{subsection: the prym variety P} shows that there is a canonical isomorphism $P[\lambda] \simeq E[2]$ and by bigonal duality we have $A[\widehat{\lambda}] \simeq  \widehat{E}[2]$ as well. We conclude:

\begin{lemma}
There is a short exact sequence of $\Gal_k$-modules
\begin{equation}\label{eq:ses}
0 \to E[2] \to P[2]\to \widehat{E}[2] \to 0.
\end{equation}
\end{lemma}

The $k$-points of the outer terms in \eqref{eq:ses} are easy to determine.
\begin{lemma}\label{lemma: 2-torsion in elliptic curves}
Let $E = E_{a,b}$ and $\widehat{E} = \widehat{E}_{a,b}$ be as above. Then 
\begin{enumerate}
\item $E[2](k) \neq 0$ if and only if $16(a^2 - 4b)$ is a cube in $k$, in which case $P[2](k) \neq 0$. 
\item $\widehat{E}[2](k) \neq 0$ if and only if $b$ is a cube in $k$, in which case $A[2](k) \neq 0$.
\end{enumerate}
Moreover if $k(\omega) \neq k$, then $|E[2](k)| \leq 2$ and $|\widehat{E}[2](k)| \leq 2$.
\end{lemma}
\begin{proof}
    This can be read off the models $E \colon y^2 = x^3 + 16(a^2 - 4b)$ and $\widehat{E} \colon y^2 = x^3 + b$ of \S\ref{subsection: bigonal duality}.
\end{proof}
To determine $P[2](k)$, we study the extension class of $\eqref{eq:ses}$ and determine when an element of $\widehat{E}[2](k)$ lifts to $P[2](k)$.
We will use the following explicit geometric description of $P[2]$; we thank Adam Morgan for pointing it out to us.
Recall that $E$ is an elliptic curve with origin $O_E=\pi(\infty)$.
\begin{proposition}\label{prop: geometric description of P[2]}
    Each divisor class in $P[2](\bar{k})$ is represented by a unique divisor of the form $R + \pi^*(T) - 3\infty$, where $R \in C(\overline{k})$ is a ramification point of the map $\pi\colon C\rightarrow E$ and $T \in E(\overline{k})$ is such that $[2](T) = -\pi(R)$. 
\end{proposition}
\begin{proof}
    We first show that every such divisor defines an element of $P[2](\bar{k})$, so let $x= R+\pi^*(T)-3\infty$ be such a divisor class. 
    Recall that $\pi(\infty)=O_E$ denotes the origin of $E$.
    The condition $[2](T) = -\pi(R)$ translates to an equivalence $2T+\pi(R) \sim 3O_E$.
    Pulling this equivalence back along $\pi$ shows that $2x= 0$, so $x\in J[2](\bar{k})$.
    Since both $R$ and $\pi^*(T)$ are fixed by $\tau$, $\tau(x) = x$ so $\tau(x) + x = 2x = 0$ so $x\in P[2](\bar{k})$.
    We now claim that two such divisors $R+\pi^*(T)-3\infty$ and $R'+\pi^*(T')-3\infty$ are linearly equivalent if and only if $R=R'$ and $T=T'$.
    Indeed, Proposition \ref{proposition: properties of the embedding of C into A}(1) shows that $R=R'$, and then $T=T'$ follows from the fact that $\pi^*\colon E\rightarrow J$ is injective.
    
    Since there are $16$ divisors of this form and $P[2](\bar{k})$ has order $16$ too, every element of $P[2](\bar{k})$ has a unique representative of this form.
\end{proof}
This description of $P[2]$ is compatible with the sequence \eqref{eq:ses}:
the map $E[2]\rightarrow P[2]$ sends $T$ to $\pi^*(T)-2\infty$;
the map $P[2]\rightarrow \widehat{E}[2]$ sends $R+\pi^*(T)-3\infty$ with $R = (0,t)\in C(\bar{k})$ to $(-t,0)\in \widehat{E}[2](\bar{k})$, using the coordinates of the equations in \S\ref{subsection: bigonal duality}.

\begin{theorem}\label{thm:2-division poly}
An element $(-t,0) \in \widehat{E}[2](k)$ with $t\in k$ lifts to $P[2](k)$ under \eqref{eq:ses} if and only if 
$g_{a,t}(z) := z^4-6tz^2+4az-3t^2$ has a $k$-rational root.
Consequently, $(P[2]\setminus P[\lambda])(k)\neq \varnothing$ if and only if there exists $t\in k$ such that $t^3 = b$ and such that $g_{a,t}$ has a $k$-rational root.
\end{theorem}
\begin{proof}
By the above proposition, the point lifts to $P[2](k)$ if and only if the corresponding point $(4t,4a)\in E(k)$ is divisible by $2$ in $E(k)$.
Using the $2$-descent map $E(k)/2E(k)\hookrightarrow \HH^1(k,E[2])$ and its interpretation via binary quartic forms, this is equivalent to $z^4-24tz^2+32az-48t^2$ having a root in $k$ \cite[\S2, (11)]{AlpogeHo}. Changing $z$ by $2z$ and dividing by $16$ results in the polynomial $g_{a,t}$.
\end{proof}

\begin{corollary}\label{cor: P[2] param}
    $(P[2]\setminus P[\lambda])(k)\neq \varnothing$ if and only if $(a,b)= ((4s + 3)(4s^2 - 3)d^3,(4s+3)^3d^6)$ for some $s,d \in k$.
\end{corollary}
\begin{proof}
 Theorem \ref{thm:2-division poly} shows that $P[2]\setminus P[\lambda]$ has a $k$-point if and only if there exists a $t\in k$ with $t^3=b$ such that $g_{a,t}(z) = z^4-6tz^2+4az-3t^2$ has a $k$-rational root.
 A calculation shows that if $(a,b)$ is of the above form for some $s,d\in k$ then $t=(4s+3)$ is such that $g_{a,t}(z)$ has the root $z=-4s-3$ in $k$.
 Conversely, suppose that $t \in k$ is a cube root of $b$ such that $g_{a,t}(z)$ has a root in $k$.
 If $a=0$, a calculation shows that we may take $s =\pm \sqrt{3}/2$, so we may assume that $a\neq 0$.
 Note that $\lambda g_{a,t}(z) = g_{\lambda^3a,\lambda^2 t}(\lambda z)$ for all $\lambda \in k$.
 Choosing $\lambda = t/a$ and setting $v=t^3/a^2$, it follows that $g_{v,v}(z)$ has a $k$-rational root.
 Since $t\neq 0$, this root must be nonzero.
 The locus of $(v,z)$ with $g_{v,v}(z)=0$ and $z\neq 0$ is isomorphic to an open subset of a smooth conic.
 After rationally parametrizing this conic we find that $v$ must be of the form $(4s+3)/(4s^2-3)^2$ for some $s\in k$. It follows that $(a,b)$ is of the above form with $d=a/t$.
\end{proof}

In the following we assume for simplicity that $k(\omega) \neq k$, so that $E[2](k) \neq E[2]$. For abstract groups $G$ and $G'$, we sometimes write $G \subset G'$ as shorthand for ``there exists an embedding $G \hookrightarrow G'$''.

\begin{proposition}\label{prop:(Z/2)^2 criterion}
Suppose $k(\omega) \neq k$. Then the following conditions are equivalent:
\begin{enumerate}
    \item $(\Z/2\Z)^2 \subset P(k)$; 
     \item $E[2](k) \neq 0$ and there exists a branch point of $C\rightarrow E$ different from $O_E$ lying in $2E(k)$;
    \item $(a,b) = ((16 w^6+40 w^3-2)d^3,\left(8 w^3+1\right)^3d^6)$ for some pair $w,d \in k$;
    \item $(\Z/2\Z)^2 \subset A(k)$.
\end{enumerate}
\end{proposition}
\begin{proof}
That $(1)$ and $(2)$ are equivalent follows from \eqref{eq:ses}, Lemma \ref{lemma: 2-torsion in elliptic curves}, the fact that $k(\omega)\neq k$ and Proposition \ref{prop: geometric description of P[2]}. 
Similarly, Proposition \ref{prop:(Z/2)^2 criterion} shows that $(\Z/2\Z)^2 \subset P_{a,b}(k)$ if and only if $(a,b) = ((4s + 3)(4s^2 - 3)d^3,(4s+3)^3d^6)$ for some $s,d\in k$ and $16(a^2-4b) = 16 d^6(2 s-3) (2s+1)^3 (4s+3)^2$ is a cube. 
The latter is equivalent to $(3-2s)/4(4s+3)$ being a cube $w^3$ in $k$.
Solving for $s$, plugging in back to $a,b$ and absorbing common factors in $d$ proves the equivalence between $(1)$ and $(3)$.

A calculation shows that if $(a,b)$ is of the form $(3)$ for some $w,d\in k$, then $(8a,16(a^2-4b))$ is of the form $(3)$ with $w'=-1/2w$ and $d'=-(2t)^2d$.
Therefore $(3)\Leftrightarrow (4)$ by bigonal duality from \S\ref{subsection: bigonal duality}.
\end{proof}

Note that when $k(\omega) \neq k$, we have $|P[2](k)| \leq 4$ by $\eqref{eq:ses}$ and Lemma \ref{lemma: 2-torsion in elliptic curves}. So the conditions in Proposition \ref{prop:(Z/2)^2 criterion} are also equivalent to the condition $P[2](k) \simeq (\Z/2\Z)^2$.

\subsection{\texorpdfstring{$\sqrt{-3}$}{sqrtminus3}-torsion}\label{subsec: sqrt3-torsion}
We use the $\mu_3$-action on $P$ to explicitly analyze a $\Gal_k$-submodule of $P[3]$.
Let $\mathfrak{p}$ be the ideal $(1 - \omega) \subset \Z[\omega]$, viewed as a subgroup of $\End(P_{k^{\sep}})$. Since $\mathfrak{p}$ is $\Gal_k$-stable, so is the subgroup $A[\mathfrak{p}]$ of points fixed by $\omega$. We have $\mathfrak{p} = (\sqrt{-3})$, so $P[\mathfrak{p}]$ has order $9$, and is the kernel of an isogeny $P \to B$ over $k$. In fact, the quotient $B = P/P[\mathfrak{p}]$ is isomorphic to the sextic twist $P_{-27}$, which is also the quadratic twist of $P$ by the extension $k(\omega)/k$ \cite[Rem\ 2.8]{ShnidmanWeiss}. We therefore have an exact sequence
\begin{equation}\label{eq:3torsion}
    0 \to P[\mathfrak{p}] \to P[3] \to P_{-27}[\mathfrak{p}] \to 0
\end{equation}
Of course, $A = P^\vee$ also has a $\mu_3$-action, so we can similarly define $A[\mathfrak{p}] \subset A[3]$, a subgroup of order $9$ that sits in an analogous short exact sequence.

\begin{lemma}\label{lem:3-tors A =P}
    $A[\mathfrak{p}] \simeq P[\mathfrak{p}]$ as $\Gal_k$-modules.
\end{lemma}
\begin{proof}
    The restriction of the polarization $\lambda$ induces such an isomorphism.
\end{proof}
Let $f(x) := x^4 + ax^2 +b \in k[x]$ and let $\mathcal{R} = \{\pm \alpha_1, \pm \alpha_2\}$ be the $4$ roots of $f$ in $k^{\sep}$.

\begin{lemma}\label{lemma: sqrt3-torsion generators}
The Galois module $P[\p]$ is generated by the classes $D_1 := (\alpha_1,0) - (-\alpha_1,0)$ and $D_2 := (\alpha_2,0) - (-\alpha_2,0)$.  
\end{lemma}
\begin{proof}
These classes live in $P$, are non-trivial, and are annihilated by $1 - \omega$ so lie in $P[\p]$. They are linearly independent over $\F_3$ since otherwise we would have $D_2\sim \pm D_1$, which is impossible since $C$ is not hyperelliptic.  
\end{proof}

Let $D_4\subset \Sym(\mathcal{R})$ the subset of permutations of $\mathcal{R}$ such that $(-\alpha)^{\sigma} = -\alpha^{\sigma}$ for all $\alpha\in \mathcal{R}$ and $\sigma \in \Sym(\mathcal{R})$. This is a dihedral group of order $8$, generated by its four reflections $\{\tau_1,\tau_2, \hat{\tau}_1, \hat{\tau}_2\}$, where $\tau_i$ maps $\alpha_i$ to $-\alpha_i$ and fixes the other roots, $\hat{\tau}_1$ swaps $\alpha_1 \leftrightarrow \alpha_2$ and $-\alpha_1 \leftrightarrow -\alpha_2$, and $\hat{\tau}_2$ swaps $\alpha_1 \leftrightarrow -\alpha_2$ and $-\alpha_1 \leftrightarrow \alpha_2$.
(The reader is invited to picture $D_4$ acting on a square with labels $\{\alpha_1, \alpha_2, -\alpha_1, -\alpha_2\}$.)
Let $V$ be the $\F_3$-vector space with basis $\{v_{\alpha}\mid \alpha \in \mathcal{R}\}$, modulo the relations $v_{-\alpha} = -v_{\alpha}$. 
The $D_4$-action on $\mathcal{R}$ induces a linear $D_4$-action on $V$.
After choosing the basis $\{v_{\alpha_1}, v_{\alpha_2}\}$, this induces a representation $\varphi\colon D_4\rightarrow \GL_2(\F_3)$, isomorphic to the mod $3$ reduction of the reflection representation of $D_4$.
The $\Gal_k$-action on $\mathcal{R}$ determines a homomorphism $\chi\colon \Gal_k\rightarrow \GL_2(\F_3)$.
Lemma \ref{lemma: sqrt3-torsion generators} shows:

\begin{proposition}\label{proposition: sqrt3-torsion as mod 3 reflection rep}
    The $\Gal_k$-action on $P[\mathfrak{p}](k^{\sep})$ in the $\F_3$-basis $\{D_1,D_2\}$ is given by the composition $\Gal_k \xrightarrow{\chi} D_4\xrightarrow{\varphi} \GL_2(\F_3)$. 
\end{proposition}

\begin{corollary}\label{cor:Z/3 x Z/3}
We have $P[\p](k) \simeq (\Z/3\Z)^2$ if and only if $f(x)$ splits completely in $k[x]$.    
\end{corollary}
\begin{proof}
    Use Proposition \ref{proposition: sqrt3-torsion as mod 3 reflection rep} and the fact that $\varphi$ is injective.
\end{proof}

\begin{corollary}\label{cor:semisimplicity}
The $\F_3[\Gal_k]$-module $P[\p]$ is semisimple.    
\end{corollary}
\begin{proof}
    Use Proposition \ref{proposition: sqrt3-torsion as mod 3 reflection rep} and the fact that $3$ is coprime to the order of $D_4$.
\end{proof}

Let $\hat{f} := x^4 + 8ax + 16(a^2 - 4b)$.

\begin{corollary}\label{corollary:root3-torsion}
$P[\mathfrak{p}](k) \neq 0$ if and only if either $f$ or $\hat{f}$ has a root in $k$.    
\end{corollary}
\begin{proof}
Let $G$ be the image of $\chi\colon \Gal_k\rightarrow D_4$.
Then $P[\mathfrak{p}](k) \neq 0$ if and only if $G$ fixes some nonzero element of $V = \F_3^2$.
A group theory calculation shows that only reflections have fixed points, so this is equivalent to $G \subset \langle \tau_i\rangle $ or $G\subset \langle \hat{\tau}_i\rangle$ for some $i=1,2$.
We have $G \subset \langle \tau_i\rangle$ for some $i$ if and only if $f$ has a $k$-rational root.
By Lemma \ref{lemma: bigonal dual Galois theory}, $G \subset \langle \hat{\tau}_i\rangle = \text{Stab}_{D_4}(\alpha_1+\alpha_2)$ for some $i$ if and only if $\hat{f}$ has a $k$-rational root.
\end{proof}

\begin{proposition}\label{prop: 3-torsion families}
    $\Z/3\Z \subset P[\mathfrak{p}](k)$ if and only if either
    \begin{enumerate}
        \item $P = P_{-(c+1)d^2,cd^4}$ for some $c,d \in k$ or
        \item $P = P_{-8(c + 1)d^2, 16(c -1)^2d^4}$ for some $c,d \in k$.
    \end{enumerate}
\end{proposition}
\begin{proof}
The polynomial $f(x)$ has a $k$-rational root if and only if $f(x) =(x^2-t^2)(x^2-c)$ for some $c,t\in k$, if and only if $P \simeq P_{-c-t^2,ct^2}$ for some $c,t \in k$ if and only if $P$ is as in $(1)$.
Similarly $\widehat{f}$ has a $k$-root if and only if $P$ is as in $(2)$. The proposition then follows from Corollary \ref{corollary:root3-torsion}.
\end{proof}

\begin{proposition}\label{prop:Z/3 x Z/3 classification}
$(\Z/3\Z)^2 \subset P[\mathfrak{p}](k)$ if and only if $P = P_{-(c^2 +1)d^2, c^2d^4}$ for some $c,d \in k$.  
\end{proposition}

\begin{proof}
By Corollary \ref{cor:Z/3 x Z/3}, this happens if and only if $f(x)$ splits completely, in which case we have $f(x) = (x^2 - t^2)(x^2 - c^2)$ for some $c,t \in k$.  Up to cubic twist we may take $t = 1$, hence we arrive at the form $P_{-(c^2 +1)d^2, c^2d^4}$, for some $c, d\in k$.
\end{proof}

\begin{question}{\em 
Do there exist bielliptic Picard curves $C/\Q$ such that $P[3](\Q) \neq 0$ but $P[\p](\Q)= 0$? 
}
\end{question}

\subsection{\texorpdfstring{$6$}{6}-torsion}\label{subsec: 6-torsion}
We combine our results on $P[2]$ and $P[3]$ to produce  
examples of points of order $6$ in $P(k)$. Here we assume $k = \Q$ for simplicity.
\begin{proposition}\label{prop:Z/6}
$\Z/6\Z \subset P[2\mathfrak{p}](\Q)$ if and only if one of the following holds
\begin{enumerate}
    \item $P = P_{2^4(c+1)(c-1)^2, 2^8c(c-1)^4}$  for some $c \in \Q$; or
    \item $P = P_{8c(1-c),16c^2(1+c)^2}$  for some $c \in \Q$.
    \item $P = P_{-(c+1)c^4, 16(1-c)^2c^8}$ where $c = \frac{(3t-2)(5t-2)^3}{t(7t-4)^3}$ for some $t \in \Q$.
    \item $P = P_{\frac14(1-v)^3(3v+1)^3(3v^4 + 6v^2 -1),v^6(v-1)^6(3v+1)^6}$ for some $v \in \Q$.
\end{enumerate}
\end{proposition}

\begin{proof}
Each family above is the fiber product of one of the two families of rational $2$-torsion in $P$ (Lemma \ref{lemma: 2-torsion in elliptic curves} and Corollary \ref{cor: P[2] param}) with one of the two families in Proposition \ref{prop: 3-torsion families}. We omit the details, as it is tedious but straightforward algebra, in a spirit similar to Proposition \ref{prop:(Z/2)^2 criterion}. 
\end{proof}

\begin{proposition}\label{prop: order 18 torsion}
$\Z/3\Z \times \Z/6\Z \subset P[2\p](\Q)$ if and only if there exists $c \in \Q \setminus \{0,\pm 1\}$ such that  
 $a = -16(c^2+1)(c^2-1)^2$ and $b = 2^8c^2(c^2 -1)^4$. 
\end{proposition}

\begin{proof}
This follows by combining Proposition \ref{prop:Z/3 x Z/3 classification} with both Lemma \ref{lemma: 2-torsion in elliptic curves} and Corollary \ref{cor: P[2] param}. In the first case, we are forced to set $d = 4(c^2 -1)$ which leads to the formulas for $a$ and $b$ above.  In the second case, we must first set $d = c$ in order for $b$ to be a cube, with $b = t^3$ where $t = c^2$. Then we must see when the polynomial
\[g_{a,t}(z) = g_c(z) = z^4 - 6c^2z^2 -4c^4z   - 4c^2z - 3c^4\]
has a root. The plane curve $\{(c,z) \colon g_c(z) = 0\}$ is irreducible of geometric genus $1$, and with the help of Magma \cite{Magma} we find that it is birational to an elliptic curve with Mordell-Weil group of order $8$. None of these rational points correspond to smooth bielliptic Picard curves, so this second case gives no new examples. 
\end{proof}

Simple algebra shows that  for $(a,b)$ as in Proposition \ref{prop: order 18 torsion}, the curve $C_{a,b}$ has affine model
\[2(c^2 -1)^2y^3 = (x^2 -1)(x^2 -c^2),\]
recovering the equation given in the introduction.

\begin{remark}{\em 
Combining Proposition \ref{prop: order 18 torsion} and Remark \ref{rem: genus two curve}, we are led to the family of genus two curves
 \[C_t \colon (t^2 + 1)y^2 = (2x^2 + 2x - 1)((t^2-1)^2x^4 + 2(t^2 -1)^2x^3 + 4t^2x -t^2).\]
For all but finitely many $t\in \Q$, the Jacobian $J = J_t = \Jac(C_t)$ satisfies $\End^0(J_{\overline{\Q}}) \simeq B$. For all integers $1 \leq t \leq 10000$ we compute $J_t(\Q)_{\mathrm{tors}} \simeq \Z/2\Z \times (\Z/3\Z)^2$ in Magma \cite{Magma}. 
We suspect this holds for all (but finitely many) $t\in \Q$, since the isogeny of Remark \ref{rem: genus two curve} should have degree $2$.  
 }
\end{remark}

\subsection{Infinite families}
Table \ref{table: prym torsion} summarizes many  of the computations of Section \S\ref{sec:2,3,6-torsion}.  For each group $G$ in Theorem \ref{thm:main}, the table gives an infinite family of curves $C_{a,b}$ such that $G \hookrightarrow P(\Q)_{\mathrm{tors}}$. 
Since the $j$-invariant of every family is not constant, each family contains infinitely many distinct $\bar{\Q}$-isomorphism classes, and all but finitely many of these are geometrically simple by Lemma \ref{lemma: CM points are algebraic and there are finitely many of bounded degree}.
For certain groups $G$ there exist multiple such families, but we only write one. In particular, we do not claim that every Prym $P$ such that $G \hookrightarrow P(\Q)_{\mathrm{tors}}$ appears in this table.

\begin{table}
\centering
\begin{tabular}{c | c | c  | c}
    $G$ & $a$ & $b$ & $j$\\
    \hline 
    $\{1\}$ & $a$ & $b$ & $\frac{4b-a^2}{4b}$\\
    \hline
     $\Z/2\Z$ & $2s$ & $s^2 - t^3$ & $\frac{t^3}{t^3 -s^2}$\\
     \hline
   $\Z/3\Z$ & $-(c+1)d^2$ & $cd^4$ & $\frac{-(c-1)^2}{4c}$ \\
   \hline
     $(\Z/2\Z)^2$ & $(16 w^6+40 w^3-2)d^3$ & $ \left(8 w^3+1\right)^3d^6$ & $-\left(\frac{4 w \left(w^3-1\right)}{\left(8 w^3+1\right)}\right)^3$ \\
    \hline
    $\Z/6\Z$ & $8c(1-c)$ & $16c^2(1+c)^2$ & $\frac{4c}{(c+1)^2}$ \\
    \hline
    $(\Z/3\Z)^2$ & $-(c^2+1)d^2$ & $c^2d^4$ & $\frac{(c^2 - 1)^2}{-4c^2}$\\
    \hline
    $\Z/6\Z \times \Z/3\Z$ & $-16(c^2+1)(c^2-1)^2$ &  $2^8c^2(c^2 -1)^4$ & $\frac{(c^2 -1)^2}{-4c^2}$\\
    \hline
\end{tabular}
\caption{Infinite families of Pryms $P_{a,b}$ with $G \hookrightarrow P_{a,b}(\Q)_{\mathrm{tors}}$.}
\label{table: prym torsion}
\end{table}

\subsection{\texorpdfstring{$2$}{2}-torsion in $A$ using bitangents}\label{subsec: 2-torsion in A using bitangents}

Even though it is possible to use the description of $P[2]$ from \S\ref{prop: geometric description of P[2]} and bigonal duality to describe $A[2]$, we give a more intrinsic description of the $\Gal_k$-module using the bitangents of the quartic curve $C$. 
This will be used in \S\ref{subsec:no 4-torsion using bitangents} to rule out rational points of order $4$ in $P$.

We first recall the connection between the bitangents on a smooth plane quartic curve $C$ over a general field $k$ (of characteristic not $2$ nor $3$) and $2$-torsion points in its Jacobian; see \cite[\S6]{Dolgachev-classical} for more details.
Let $X\subset \P^2_k$ be a smooth plane quartic curve. 
A \define{bitangent} of $X$ is a line $\ell \subset \P^2_k$ which intersects $X$ with even multiplicity at every point, i.e. $\ell \cap X = 2D$ for some effective divisor $D$ of $X$ of degree $2$.
Since $X$ is canonically embedded, $2D \sim K_X$ and so $D$ is a \define{theta characteristic}.
This sets up a bijection between the $28$ bitangents of $X_{\bar{k}}$ and the odd theta characteristics of $X_{\bar{k}}$.
If $D,E$ are two odd theta characteristics of $X_{\bar{k}}$, then $D-E \in \Jac_X[2](\bar{k})$.
Every element of $\Jac_X[2](\bar{k})$ is a sum of points of this form.

We now specialize to our setting where $(C, \gamma)$ is a bielliptic Picard curve over $k$. 
Note that the line at infinity is a bitangent with theta characteristic $2\infty$. 
\begin{lemma}
The action of $\langle \tau\rangle$ on the $28$ bitangents of $C_{\bar{k}}$ has four fixed points $($including the line at infinity$)$ and $12$ orbits of size $2$.
\end{lemma}
\begin{proof}
   Using the explicit equation \eqref{equation: bielliptic picard curve equation}, $\tau(x,y) = (-x,y)$.
   Therefore $\tau$ is fixes a bitangent line (different from the line at infinity) if and only if it is of the form $y=c$. 
   Such an equation defines a bitangent line if and only if $x^4+ax^2+b - c$ is the square of a polynomial, which happens if and only if $c^3 = b - a^2/4$, which has $3$ solutions in $\bar{k}$. 
\end{proof}

Write $p\colon J\rightarrow A$ for the projection map, the dual of $P\hookrightarrow J$. 
Given any bitangent $\ell$ of $C_{\bar{k}}$ with theta characteristic $D$, write $x_{\ell} = D-2\infty$ for the associated $2$-torsion point in $J(\bar{k})$.  

\begin{proposition}\label{prop:bitangent description}
If $\ell$ is a bitangent line fixed by $\tau$, then $p(x_\ell) =0$.    If $\ell$ and $\ell'$ are two distinct bitangent lines of $C$ not fixed by $\tau$, then $p(x_{\ell}) = p(x_{\ell'})$ if and only if $\ell' = \tau(\ell)$. 
The map $\ell\mapsto p(x_{\ell})$ induces a bijection between the $\langle\tau\rangle$-orbits of bitangents of $C_{\bar{k}}$ of size $2$ and $A[2](\bar{k}) \setminus A[\widehat{\lambda}](\bar{k})$.
\end{proposition}
\begin{proof}
   We may assume that $k = \bar{k}$.
   If $\ell$ is fixed by $\tau$, then $\ell \cap C = 2\pi^{-1}(R)$ for some $R \in E[2]$.
   Since the morphism $p$ has kernel $\pi^*(E)$, $p(x_{\ell})=0$.
   We claim that if $\ell$ is not fixed by $\tau$ then $p(x_\ell)\not\in A[\widehat{\lambda}]$. Since $J \simeq (P \times E)/(P\cap \pi^*(E))$, this is equivalent to saying that $x_\ell \notin P[2]$. This follows from the assumption that $D \not\sim \tau(D)$ and since $\tau$ acts as $-1$ on $P$.
   Since $A[2] \setminus A[\widehat{\lambda}]$ and the set of $\langle\tau\rangle$-orbits of bitangents of size $2$ both have size $12$, the remainder of the lemma follows from Proposition \ref{proposition: A is a quotient of Sym^2C}.
\end{proof}

\section{Classifying rational torsion in Prym varieties}\label{sec: classifying rational torsion subgroups}
In Section \ref{sec:2,3,6-torsion} we showed that the finite  groups mentioned in Theorem \ref{thm:main}  all arise as subgroups of Pryms surfaces, and in fact infinitely often.  In this section we finish the proof of the theorem by showing that all other finite abelian groups do not arise. 

Let $C = C_{a,b}$ be a bielliptic Picard curve over $\Q$, with double cover $\pi \colon C \to E$,  Prym $P$, and dual Prym $A = P^\vee$, as usual.

\subsection{Eliminating \texorpdfstring{$\ell$}{ell}-torsion for \texorpdfstring{$\ell \geq 5$}{ell5}}

We first show that $P[\ell](\Q)=0$ for all primes $\ell \geq 5$.
When $P$ is geometrically simple, this follows from  the more general result \cite[Theorem 1.1]{LSSV-QMMazur}. We briefly explain how the techniques of that paper can be adapted to handle the geometrically split case as well and in the process give a simplified version of the proof (which is possible due to the specific nature of our family).

Fix a prime $\ell \geq 5$.
Theorem \ref{theorem: galois action endo field on O} shows that there exists an embedding $\iota\colon \calO \rightarrow \End(P_{\bar{\Q}})$ whose image is $\Gal_{\Q}$-stable; we will fix such an embedding in what follows. 
Let $\calO_{\ell} := \calO \otimes_{\Z} \F_{\ell}$.
Since $\calO$ has discriminant $6$ and $\calO$ is maximal at $\ell$, there is an isomorphism $\calO_{\ell} \simeq \Mat_2(\F_{\ell})$. 
Let $M := P[\ell](\bar{\Q})$, a free left $\calO_{\ell}$-module of rank $1$.
Both $\calO_{\ell}$ and $M$ are (right) $\Gal_{\Q}$-modules, and the action of $\calO_{\ell}$ on $M$ is $\Gal_{\Q}$-equivariant. 

\begin{proposition}\label{proposition: reduction to gl2 type}
    If  $P[\ell](\Q) \neq 0$, then $\calO \subset \End(P_{\Q(\omega)})$.
\end{proposition}
\begin{proof}
    The proof of \cite[Theorem 6.0.1]{LSSV-QMMazur} carries over essentially without change; we briefly sketch the details.
    Suppose that $m\in M^{\Gal_{\Q}}$ is nonzero.
    An argument identical to \cite[Lemma 6.2.3]{LSSV-QMMazur} shows that $\calO_{\ell}\cdot m\subset M$ has order $\ell^2$.
    Let $S := \Z[\omega] \subset \calO$ and $S_{\ell} := S\otimes_{\Z}\F_{\ell}$.
    Then $S$ is a $\Gal_{\Q}$-stable subring of $\calO$, and the induced action of $\Gal_{\Q}$ on $S$ factors through $\Gal(\Q(\omega)/\Q)$.
    Since $S_{\ell}$ has no $\Gal_{\Q}$-stable proper nonzero ideals, the map $S_{\ell}\rightarrow \calO\cdot m$ sending $x\mapsto x\cdot m$ is injective.
    By cardinality reasons, it is also surjective, so $S_{\ell} \cdot x = \calO_{\ell} \cdot x$.
    The (purely linear-algebraic) \cite[Lemma 6.2.6]{LSSV-QMMazur} then shows that  $\Gal_{\Q(\omega)}$ acts trivially on $\calO_{\ell}$.
    Since $\ell\geq 5$, by \cite[Lemma 3.5.7]{LSSV-QMMazur} this implies that $\Gal_{\Q(\omega)}$ acts trivially on $\calO$.
    In other words, $\calO\subset \End(P_{\Q(\omega)})$, as desired. 
\end{proof}

The quaternionic multiplication places strong restrictions on the reduction type of $P$ at a prime $p$.
First of all, $P$ has potentially good reduction at $p$, and acquires good reduction over a totally ramified extension $K$ of $\Q_p$ \cite[Lemma 4.1.2]{LSSV-QMMazur}.
The special fiber of the Neron model of $P_K$ is an abelian surface over $\F_p$; we will denote this abelian surface by $P_{\F_p}$ and (by slight abuse of language) call it the reduction of $P$ mod $p$. (Its isomorphism class might depend on the choice of $K$.)
The reduction $P_{\F_p}$ is geometrically isogenous to the square of an elliptic curve.
If $p$ divides the discriminant of $\calO$ (that is, if $p \mid 6$), then $P_{\F_p}$ is supersingular \cite[\S2]{Jordan86}.
Additionally, the prime-to-$p$ torsion subgroup of $P(\Q)$ embeds in $P_{\F_p}(\F_p)$ for every prime $p$ by formal group considerations. 
We use these remarks to prove the following:

\begin{proposition}\label{proposition: no primes >3 in torsion}
    Let $\ell \geq 5$ be a prime. Then $P[\ell](\Q) =0$.
\end{proposition}
\begin{proof}
Suppose instead that $P[\ell](\Q)\neq 0$.
    Then $\ell$ divides the order of $P_{\F_3}(\F_3)$.
    Moreover, $P_{\F_3}$ is supersingular.
    Proposition \ref{proposition: reduction to gl2 type} shows that $\calO\subset \End(P_{\Q(\omega)})$, so the base change $P_{\Q(\omega)}$ has quaternionic multiplication over $\Q(\omega)$.
    Since $3$ is ramified in $\Q(\omega)$, we must also have $\calO \subset \End(P_{\F_3})$.
    By Honda-Tate theory (conveniently recorded in the results of the following  \href{https://www.lmfdb.org/Variety/Abelian/Fq/?start=&count=50&hst=List&q=3&p=&g=2&p_rank=0&initial_coefficients=&simple=&geom_simple=&primitive=&polarizable=&jacobian=&simple_quantifier=&newton_polygon=&abvar_point_count=&curve_point_count=&simple_factors=&angle_rank=&jac_cnt=&hyp_cnt=&twist_count=&max_twist_degree=&geom_deg=&p_rank_deficit=&geom_squarefree=&dim1_factors=&dim2_factors=&dim3_factors=&dim4_factors=&dim5_factors=&dim1_distinct=&dim2_distinct=&dim3_distinct=&number_field=&galois_group=&search_type=List&sort_order=abvar_count&showcol=&hidecol=&sort_dir=}{LMFDB} \cite{lmfdb} search), we see that the only option is $\ell=7$.
    
    We conclude by excluding the case $\ell=7$.
    Let $L/\Q$ be the endomorphism field, namely the smallest field extension with the property that $\End(P_L) = \End(P_{\bar{\Q}})$.
    Since $\Gal(\Q(\omega)/\Q)$ acts nontrivially on the subring $\Z[\omega]\subset \calO\subset \End(P_{\Q(\omega)})$, we have $\Q(\omega)\subset L$.
    A result of Silverberg \cite[Theorem 4.2]{Silverberg92a} shows that $L$ is unramified at all places of good reduction of $P$.
    Since $3$ is ramified in $\Q(\omega)\subset L$, it follows that $P$ has bad reduction at $3$.
    By \cite[Proposition 4.1.3(b)]{LSSV-QMMazur} and Lemma \ref{corollary: pryms of gl2 type}, the reduction of $P$ at $3$ is totally additive; in other words the identity component of the special fibre of the Neron model of $P$ over $\Z_3$ is unipotent.
    A result of Lorenzini \cite[Corollary 3.25]{Lorenzini-groupofcomponentsneronmodel} then shows that $7 \nmid |P(\Q_3)_{\mathrm{tors}}|$, a contradiction.    
\end{proof}

\begin{remark}
{\em
    The elliptic curve $X \colon y^2 = x^3 + 48(\omega + 5)$ has a $\Q(\omega)$-rational point of order $7$. It follows that the decomposable $\mu_3$-abelian surface $X \times X$ has (two independent) $\Q(\omega)$-rational points of order $7$. 
    %The conductor of $X$ has norm $49$.
    This shows that the argument above is sharp in a certain sense.
}
\end{remark}

\subsection{Eliminating small groups}
We rule out certain small groups of order $2^i3^j$ from appearing as subgroups of $P(\Q)$, using arguments that are specific to our family. 

\begin{proposition}\label{prop: no 9-torsion}
    $\Z/9\Z \not\subset P(\Q)$ and $P(\Q)[3]\subset (\Z/3\Z)^2$.
\end{proposition}
\begin{proof}
    Using the notation and remarks made before Proposition \ref{proposition: no primes >3 in torsion}, the prime-to-$2$ torsion subgroup of $P(\Q)$ injects into $P_{\F_2}(\F_2)$ and $P_{\F_2}$ is a supersingular abelian surface.
    Consulting the LMFDB \cite{lmfdb}, we see that the $3$-part of $P_{\F_2}(\F_2)$ is of order at most $9$, and there is a unique \href{https://www.lmfdb.org/Variety/Abelian/Fq/2/2/a_e}{isogeny class} of supersingular abelian surfaces $B/\F_2$ with $9 \mid |B(\F_2)|$.
    This isogeny class is the square of a supersingular elliptic curve $E$ over $\F_2$ and $E(\F_2)\simeq \Z/3\Z$.
    By \cite[Lemma 7.2.1]{LSSV-QMMazur} (and its proof) and the fact that $\Z[\sqrt{-2}]$ is a PID, every member of this isogeny class is in fact isomorphic to $E^2$, so $B(\F_2) \simeq (\Z/3\Z)^2$.
    We conclude that the $3$-part of $P_{\F_2}(\F_2)$, hence also the $3$-part of $P(\Q)_{\mathrm{tors}}$,  is a subgroup of $(\Z/3\Z)^2$.
\end{proof}

\begin{proposition}\label{proposition: no order 12}
 $(\Z/2\Z)^2 \times \Z/3\Z \not\subset P(\Q)$.
\end{proposition}

\begin{proof}
Suppose that $(\Z/2\Z)^2 \times (\Z/3\Z) \subset P(\Q)$. 
By Proposition \ref{prop:(Z/2)^2 criterion}, we have $j = -\left(\frac{4 w \left(w^3-1\right)}{\left(8 w^3+1\right)}\right)^3$ for some $w\in \Q$.
Since $\Z/3\Z\subset P(\Q)$, the exact sequence \eqref{eq:3torsion} shows that $P[\mathfrak{p}](\Q)\neq 0$ or $P_{-27}[\mathfrak{p}](\Q)\neq 0$.
Proposition \ref{prop: 3-torsion families} then shows one of $j$ or $1/j$ is of the form $-4t/(t-1)^2$ for some $t\in \Q$. 
Replacing $j$ by $1/j$ if necessary, using Proposition \ref{prop:(Z/2)^2 criterion}, Lemma \ref{lem:3-tors A =P} and bigonal duality, we may assume $j = -4t/(t-1)^2$ for some $t\in \Q$.
This is equivalent to 
\begin{align}
    (4j-2)^2-4=\frac{1024 w^3 \left(w^3-1\right)^3 \left(8 w^6+20 w^3-1\right)^2}{\left(8 w^3+1\right)^6}
\end{align}
being a square in $\Q$. Since $w=0,1,-1/2$ leads to $j=0, \infty$, we may assume $w$ is distinct from these values. 
Then this expression is a square if and only if $z^2 = w(w^3-1)$ for some $z\in \Q$.
This equation defines the affine part $C^{\circ}$ of a genus $1$ curve $C/\Q$ which is a double cover of $\P^1$.
The point $(w,z)=(0,0)$ endows $C$ with the structure of an elliptic curve with Weierstrass model $y^2 = x^3+1$. (This can be seen using the invariant theory of binary quartics \cite[\S2.1]{AlpogeHo}.)
This elliptic curve has Mordell-Weil group $\Z/6\Z$, so $C^{\circ}(\Q)=\{(0,0), (-2,0), (1,\pm 3)\}$.
We conclude that $w\in\{0,1,-2\}$, but we already observed these all correspond to $j=0,\infty$.
\end{proof}

\subsection{Eliminating points of order \texorpdfstring{$4$}{4}}\label{subsec:no 4-torsion using bitangents}

We rule out the existence of order $4$ points in $P(\Q)$ using the description of $A[2]$ via bitangents from \S\ref{subsec: 2-torsion in A using bitangents}.

\begin{proposition}\label{proposition: no 4-torsion}
    $\Z/4\Z \not\subset P(\Q)$.
\end{proposition}
\begin{proof} 
    For the sake of contradiction, let $x\in P(\Q)$ be a point of order $4$. Let $A = P^\vee$ be the dual Prym.    We have $\widehat{\lambda}\circ \lambda =[2]$, so either $\lambda(x) \in A[2]$ or $y=\lambda(x)\in A(\Q)$ has order $4$ and $\widehat{\lambda}(y)\in P[2]$.
    After possibly replacing $P$ by $A$ and applying bigonal duality (Proposition \ref{proposition: bigonal duality}), we may therefore assume that $\lambda(x)\in A[2](\Q)$.
    Since $x$ has order $4$, we have $\lambda(x)\not\in A[\widehat{\lambda}]$.  Using Proposition \ref{prop:bitangent description} and its notation, there exists a unique $\tau$-orbit of odd theta characteristics $\{D, D'\}$ such that $p(D-2\infty) =\lambda(x)$.
    This orbit is defined over $\Q$ but $D$ and $D'$ may only be defined over a quadratic extension.

    Recall that $\lambda$ is the composition $P\hookrightarrow J\xrightarrow{p} A$ and that $p$ has kernel $\pi^*(E)$, where $\pi\colon C\rightarrow E$ is the double cover.
    Therefore $x-(D-2\infty) \in \pi^*(E)$. Since $\pi^*$ is injective, we may write
    \begin{align}\label{equation: proof no 4-torsion}
    x-(D-2\infty) \sim \pi^*(y-O_E)    
    \end{align} 
    for some unique $y\in E(\bar{\Q})$.
    The left hand side has order $4$, so $y$ has order $4$ too.

    We claim that the subgroup $\{0,y,2y,3y\}$ generated by $y$ is $\Gal_{\Q}$-stable. 
    Indeed, let $\sigma \in\Gal_{\Q}$, so $\sigma(D) \in \{D, \tau(D)\}$. 
    If $\sigma(D) = D$, then $\sigma(y) =y$ by \eqref{equation: proof no 4-torsion}.
    Suppose $\sigma(D) = \tau(D)$. 
    Using that $\tau(x)= -x$ and applying $\tau$ to \eqref{equation: proof no 4-torsion}, we see that $-x -(\tau(D)-2\infty) = \tau(\pi^*(y)-O_E) = \pi^*(y-O_E)$.
    Applying $\sigma$ to \eqref{equation: proof no 4-torsion} shows that $x -(\tau(D)-2\infty) = \pi^*(\sigma(y)-O_E)$.
    Adding the two equations from the last two sentences shows that $\sigma(y) = -y=3y$, so $\{0,y,2y,3y\}$ is indeed $\Gal_{\Q}$-stable.

    But $E$ is an elliptic curve with a faithful $\mu_3$-action, hence has CM by $\Z[\omega]$, and such curves have no rational cyclic $4$-isogenies, by CM theory \cite[Theorem 6.18(c)]{BourdonClark-torsionpointsCMellcurves}. Contradiction! 
\end{proof}

\subsection{Proof of Theorem \ref{thm:main}}
The groups appearing in the theorem are realized infinitely often in Table \ref{table: prym torsion}, so it suffices to prove that $G:=P(\Q)_{\mathrm{tors}}$ is isomorphic to one of these.
Proposition \ref{proposition: no primes >3 in torsion} shows that the order of $G$ is of the form $2^i 3^j$.
Propositions \ref{prop: no 9-torsion} and \ref{proposition: no 4-torsion} show that $G$ is $6$-torsion, of the form $(\Z/2\Z)^m \times (\Z/3\Z)^n$ for some $m\geq 0$ and $n \leq 2$.
Equation (\ref{eq: SES involving P[2] and P[lambda]}) and Lemma \ref{lemma: 2-torsion in elliptic curves} show that $m\leq 2$ as well. Proposition \ref{proposition: no order 12} shows that if $m=2$ then $n=0$.
Therefore $G$ must be one of the groups appearing in Theorem \ref{thm:main}.

We do not classify all the possibilities for $J(\Q)_{\mathrm{tors}}$, but Theorem \ref{thm:main} quickly implies the following:

\begin{corollary}
    Let $C/\Q$ be a bielliptic Picard curve with Jacobian $J$. 
    Then $J(\Q)_{\mathrm{tors}} = J(\Q)[12]$.
\end{corollary}
\begin{proof}
    Using the notation of \S\ref{subsection: the prym variety P}, the surjective maps $J\rightarrow A$ and $J\rightarrow E$ combine to an isogeny $\theta\colon J\rightarrow A\times E$ whose kernel is isomorphic to $E[2]$.
    Theorem \ref{thm:main} and bigonal duality (Proposition \ref{proposition: bigonal duality}) show that $A(\Q)_{\mathrm{tors}}$ is $6$-torsion. 
    Moreover elementary arguments (using reduction mod $3$ and $5$) show that $E(\Q)_{\mathrm{tors}}$ is $6$-torsion. 
    It follows that for every $x\in J(\Q)_{\mathrm{tors}}$, $6\theta(x) = 0$, so $6 x \in \ker(\theta) \simeq E[2]$, which is $2$-torsion. 
\end{proof}

\section{Classifying rational torsion in Pryms of \texorpdfstring{$\GL_2$}{GL2}-type}\label{sec:gl2type}

Combining our knowledge of rational torsion subgroups and rational endomorphism rings of Pryms $P = P_{a,b}$, we classify the finite $\End(P)$-modules that arise as $P(\Q)_{\mathrm{tors}}$ for Pryms $P$ of $\GL_2$-type.  

Define the rational functions
\[a_2(t) := \dfrac{-4(t^2 +1)}{t^2(t^2 - 1)^2} \hspace{3mm} \mbox{ and } \hspace{3mm} b_2(t) := \dfrac{16}{t^2(t^2 -1)^4}.\] 

\begin{proposition}\label{prop:sqrt(2) and 3-torsion}
For all but finitely many $t \in \Q$, the Prym 
$P = P_{a_2(t),b_2(t)}$ is geometrically simple and satisfies $\End(P) \simeq \Z[\sqrt{2}]$ and $(\Z/3\Z)^2 \subset  P[\p](\Q)$. Conversely, if $P_{a,b}/\Q$ is a Prym with these three properties then $(a,b) = (a_2(t)\lambda^6,b_2(t)\lambda^{12})$ for some $t, \lambda \in \Q$. 
\end{proposition}
\begin{proof}
    First note that by Lemma \ref{lemma: CM points are algebraic and there are finitely many of bounded degree}, the surface $P_{a_2(t),b_2(t)}$ will be non-CM (and hence geometrically simple) for all but finitely many specializations of $t$, so for such $t$ the condition $\Z[\sqrt{2}]\simeq \End(P)$ is equivalent to $\Z[\sqrt{2}]\subset \End(P)$ by Corollary \ref{cor:principal polarization}.
    By Proposition \ref{prop:Z/3 x Z/3 classification} and Corollary \ref{cor:sqrt(2)-multiplication}, $(\Z/3\Z)^2\subset P[\mathfrak{p}]$ and $\Z[\sqrt{2}]\subset \End(P)$ if and only if $(a,b) = (-(t^2+1)d^2,t^2d^4)$ for some $t,d\in k$ and $16b(a^2-4b)\in \Q^{\times 6}$. 
    The latter is equivalent to $d = 16t^2(t^2-1) \lambda^3$ for some $\lambda \in \Q^{\times}$, proving the proposition.
\end{proof}
According to \cite[Theorem 1.4]{LSSV-QMMazur}, this is the largest torsion subgroup that can arise among maximal PQM abelian surfaces over $\Q$ of $\GL_2$-type.

Note that we cannot simultaneously have $\End(P) \simeq \Z[\sqrt{2}]$ and $P[3](\Q) \simeq \Z/3\Z$ since $3$ is inert in $\Z[\sqrt{2}]$. However, using our family of $\Z[\sqrt{6}]$-RM Pryms, we give examples of Pryms of $\GL_2$-type with $P[3](\Q) \simeq \Z/3\Z$. For this we define
\[a_6(t) = 36\frac{3t^2 -1}{t^2(3t^2 + 1)^2} \hspace{3mm} \mbox{ and } \hspace{3mm} b_6(t) = -\frac{3888}{t^2(3t^2 + 1)^4} \]

\begin{proposition}\label{prop: sqrt(6) and 3-torsion}
 For all but finitely values of $t \in \Q$
the surface $P_{a_6(t),b_6(t)}$ is geometrically simple and satisfies  $\End(P) \simeq \Z[\sqrt{6}]$ and $\Z/3\Z \subset P[\p](\Q)$. Conversely, if $P_{a,b}/\Q$ is a Prym with these three properties then either $P_{a,b}$ or $A_{a,b}$ is isomorphic to $P_{ a_6(t),b_6(t)}$ for some $t \in \Q$. 
\end{proposition}

\begin{proof}
    The proof is as before, this time combining Corollary \ref{cor:sqrt(6)-multiplication} with the two cases of Proposition \ref{prop: 3-torsion families}. Notice that the properties of having $\Z[\sqrt{6}]$-multiplication and a non-trivial rational $\p$-torsion point are preserved by duality, and the two families of Proposition \ref{prop: 3-torsion families} are interchanged by duality as well. Thus, it is enough to consider the first family, which leads to the functions $a_6(t)$ and $b_6(t)$. 
\end{proof}

\begin{question}
{\em 
The minimal conductor of a geometrically simple abelian surface $A/\Q$ of $\GL_2$-type with potential quaternionic multiplication is $3^{10}$, corresponding to a Galois orbit of weight two eigenforms $f \in S_2(\Gamma_0(243))$ with coefficients in $\Z[\sqrt{6}]$  \cite[Table 1]{GonzalezGuardia}. One can show that the corresponding optimal quotient $A$ of $J_0(243)$  has a rational point of order $3$, a $(1,2)$-polarization, and endomorphism field $L = \Q(\omega)$. The same is also true for its dual.  Using the results of this paper, it is not hard to show that either $A$ or its dual is isomorphic to $P_{a_6(t),b_6(t)}$ for some value of $t$. Which value of $t$ is it?   
}
\end{question}

\begin{proof}[Proof of Theorem $\ref{thm:gl2main}$]
We have $\End(P) \simeq  \Z[\sqrt{D}]$ for some $D\in \{2,6\}$ by Corollary \ref{cor:principal polarization}.
First assume $D = 2$, so that $\mathfrak{a}_2 = (\sqrt{2})$ and $\mathfrak{a}_3 = (3)$. Corollary \ref{cor:sqrt(2)-multiplication} gives a parameterization (in terms of $t$ and $d$) of those $P$ with $\End(P) \simeq \Z[\sqrt{2}]$. Since $A \simeq P$ in this case (again by Corollary \ref{cor:sqrt(2)-multiplication}), we have $P_{a,b}[2](\Q) \neq 0$ if and only if $b$ is a cube. Thus, choosing $t$ to be a cube guarantees that $\Z[\sqrt{2}]/\mathfrak{a}_2 \simeq \Z/2\Z \subset P[2](\Q)$.  Proposition \ref{prop:sqrt(2) and 3-torsion} gives a one-parameter family of examples with $\Z[\sqrt{2}]/\mathfrak{a}_3 \simeq \F_9 \subset P(\Q)$. By Theorem \ref{thm:main}, it remains to rule out $(\Z/2\Z)^2$ and $\Z/3\Z \times \Z/6\Z$ as subgroups of $P(\Q)$. If $(\Z/2\Z)^2 \subset P(\Q)$ then $P = P_{a,b}$ must be a specialization of the family in Corollary \ref{prop:(Z/2)^2 criterion}, and moreover $16b(w,d)(a(w,d)^2 - 4b(w,d))$ must be a sixth power. The latter is automatically a cube, so it is enough to show that it cannot be a square. Writing it out explicitly, we must show that the affine curve $Y \colon y^2 = w(w^3 - 8)(w^3 + 1)$ has no rational points with $y \neq 0$. With the help of Magma \cite{Magma} we find that the smooth projectivization $\overline{Y}$ of $Y$ is a double cover of the elliptic curve $y^2 = x^3 + 6x - 7$, which has four rational points. Checking pre-images, we find that $Y(\Q)$ consists of the three rational points with $y = 0$, as desired. Finally, we must rule out $\Z/3\Z \times \Z/6\Z \subset P(\Q)$. If $P$ has this property, then since $P[\p](\Q)$ and $P_{-27}[\p](\Q)$ are $\F_9$-vector spaces, one of them is isomorphic to $\F_9$ by (\ref{eq:3torsion}). Thus, exchanging $P$ with $P_{-27}$ if necessary (which is allowed since $P[2](\Q) \simeq P_{-27}[2](\Q)$), we may assume $P$ is a specialization of the family in Proposition \ref{prop:sqrt(2) and 3-torsion}, and hence it is enough to show that $b_2(t)$ is never a cube. It is then enough to show that there are no $\Q$-rational points on the curve $Y' \colon 4t^2(t^2-1) = y^3$ with $y \neq 0$ and $t \notin \{0,\pm 1\}$. Note the double cover $Y' \to X$, where $X \colon 4t(t -1) = y^3$ is an elliptic curve (minus the origin) such that  $X(\Q) = \{(0,0),(1,0)\}$. It  follows that $Y'$ has no interesting rational points, finishing the proof in the case $D = 2$.

Next consider $D = 6$, so that $\mathfrak{a}_2 = (2-\sqrt{6})$ and $\mathfrak{a}_3 = (3 + \sqrt{6})$. 
Corollary \ref{cor:sqrt(6)-multiplication} then gives a parameterization (in terms of $t$ and $d$) of those $P$ with $\End(P) \simeq \Z[\sqrt{6}]$. Since $A \simeq P_{-27}$ in this case, we have $P_{a,b}[2](\Q) \neq 0$ if and only if $b$ is a cube. Choosing $t$ so that $3t^2$ is a cube guarantees that $\Z[\sqrt{2}]/\mathfrak{a}_2 \simeq \Z/2\Z \subset P[2](\Q)$.  Proposition \ref{prop: sqrt(6) and 3-torsion} gives a one-parameter family of examples with $\Z[\sqrt{2}]/\mathfrak{a}_3 \simeq \F_3 \subset P(\Q)$. By Theorem \ref{thm:main}, it remains to rule out $(\Z/2\Z)^2$ and $\Z/6\Z$ as subgroups of $P(\Q)$. The argument for $(\Z/2\Z)^2$ is exactly as before so we omit it. 

It remains to rule out $\Z/6\Z \subset P(\Q)$. 
Let $I \subset \calO$ be the maximal two-sided ideal in $\calO$ above $3$ (which is unique since $\calO$ is ramified at $3$), so that $\calO/I \simeq \F_9$. Note that the completion $I \otimes \Z_3 \subset \calO \otimes \Z_3$ is generated by any element of $\calO \otimes \Z_3$ of minimal positive valuation.  Since $\mathfrak{a}_3 \subset \Z[\sqrt{6}]$ is generated by the element $3 + \sqrt{6}$ of norm $3$, and similarly $\mathfrak{p} \subset \Z[\omega]$ is generated by the element  $1 - \zeta$ of norm $3$, we  have $A[\mathfrak{a}_3] = A[I] = A[\mathfrak{p}]$, where $A[I]$ is the subgroup of $A$ killed by $\iota(I)$, where $\iota\colon \mathcal{O}\rightarrow \End(P_{\bar{\Q}})$ is a choice of embedding satisfying the conclusion of Theorem \ref{theorem: galois action endo field on O}. There is also an exact sequence 
\[0 \to P[\mathfrak{a}_3] \to P[3] \to P[\mathfrak{a}_3] \to  0\]
where the surjection is multiplication by $3 + \sqrt{6}$. Thus, if $P[3](\Q) \neq 0$,
we must also have $P[\p](\Q) = P[\mathfrak{a}_3](\Q) \neq 0$.  It follows that $P$ is a specialization of the family in Proposition \ref{prop: sqrt(6) and 3-torsion}. So it suffices to show that both $b_6(t)$ and $16(a_6(t)^2 - 4b(t))$ are never cubes. These two quantities turn out to be inverses of each other modulo cubes, so it is enough to show $b_6(t)$ is never a cube. 
More precisely, it is enough to show that the curve $Y' \colon y^3 = 12t^2(3t^2 + 1)$ has no rational points with $y \neq 0$. Using the double cover $Y' \to X$ where $X \colon y^3 = 12t(3t+1)$, which is (the affine part of) an elliptic curve $\overline{X}$ with $\overline{X} \simeq \Z/6\Z$, we check that $Y'(\Q) = \{(0,0)\}$, which completes the proof.
\end{proof}

\bibliographystyle{abbrv}
%\bibliography{references}

\end{document}